\documentclass[preprint]{imsart}
\usepackage[utf8]{inputenc}
\usepackage{dsfont}
\RequirePackage[colorlinks,citecolor=blue,urlcolor=blue]{hyperref}
\usepackage{bbold}
\RequirePackage[numbers]{natbib}

\usepackage{graphicx}
\usepackage{amsmath,amsthm,amssymb}
\usepackage{mathtools,mathdots,mathrsfs}
\usepackage{cleveref}
\usepackage{verbatim}
\usepackage{bm}
\usepackage{subcaption}
\usepackage{natbib}

\usepackage[dvipsnames,table]{xcolor}
\usepackage{tikz}
\usetikzlibrary{arrows.meta, bending,     patterns     }

\newcommand\nt[1]{\textcolor{blue}{#1}}

\newtheorem{theorem}{Theorem}[section]
\newtheorem{lemma}[theorem]{Lemma}
\newtheorem{corollary}[theorem]{Corollary}
\newtheorem{proposition}[theorem]{Proposition}
\theoremstyle{remark}
\newtheorem{definition}[theorem]{Definition}
\newtheorem{remark}[theorem]{Remark}
\newtheorem*{example}{Example}

\newcommand\numberthis{\addtocounter{equation}{1}\tag{\theequation}}
\newcommand{\Qort}{\matr{Q}_{\bot}}

\newcommand{\RR}{\mathbb{R}}
\newcommand{\HH}{\mathscr{H}}
\newcommand{\PP}{\mathscr{P}}

\newcommand{\norm}[1]{\ensuremath{\left\| #1\right\|}}

\newcommand{\vect}[1]{\boldsymbol{#1}}
\newcommand{\matr}[1]{\boldsymbol{#1}}
\newcommand{\eqdef}{\stackrel{\textrm{def}}{=}}
\newcommand{\T}{{\sf T}}                

\newcommand{\E}{\mathbf{E}} \newcommand{\Proba}{\mathbf{P}}

\newcommand{\bu}{\vect{u}}       

\newcommand{\va}{\vect{\alpha}}
\newcommand{\vb}{\vect{\beta}}

\newcommand{\bM}{\matr{M}}

\newcommand{\bA}{\matr{A}}       
\newcommand{\bQ}{\matr{Q}}

\newcommand{\bB}{\matr{B}}       
\newcommand{\bZ}{\matr{Z}}       
\newcommand{\bU}{\matr{U}}

\newcommand{\bI}{\matr{I}}       
\newcommand{\bV}{\matr{V}}       
\newcommand{\bW}{\matr{W}}       
\newcommand{\bK}{\matr{K}}       
\newcommand{\bR}{\matr{R}}       
\newcommand{\bD}{\matr{D}}       
       
\newcommand{\bL}{\matr{L}}
\newcommand{\R}{\mathbb{R}}       
\newcommand{\tbL}{\widetilde{\matr{L}}}
\newcommand{\tbU}{\widetilde{\matr{U}}}
\newcommand{\tbLam}{\widetilde{\bm{\Lambda}}}

\newcommand{\X}{\mathcal{X}}
\newcommand{\Xe}{\mathcal{X}_\varepsilon}
\newcommand{\Xs}{\mathcal{X}_\star}

\newcommand{\Y}{\mathcal{Y}}
\newcommand{\Z}{\mathcal{Z}}

\newcommand{\flatlim}{\varepsilon\rightarrow0}

\renewcommand{\O}{\mathcal{O}}       
\DeclareMathOperator{\rank}{rank}
\DeclareMathOperator{\mspan}{span}
\DeclareMathOperator{\orth}{orth}
\DeclareMathOperator{\argmin}{argmin}

\DeclareMathOperator{\diag}{diag}
\DeclareMathOperator{\Tr}{Tr}

\definecolor{darkgreen}{rgb}{0,0.6,0}
\newcommand{\fin}{\color{black}}

\newcommand{\ELE}[2]{\begin{pmatrix} #1 ; #2  \end{pmatrix}}

\newcommand{\ones}{\vect{\mathbf{1}}}

\newcommand{\Ind}{\vect{\mathds{I}}}

\newcommand{\magicn}[1]{\mathbb{M}_{#1}}

\newcommand{\mDPP}[1]{|DPP|_{#1}}

\newcommand{\ppDPP}{DPP}
\newcommand{\mppDPP}[1]{|DPP|_{#1}}

\newcommand{\cW}{\mathcal{W}}
\newcommand{\cA}{\mathcal{A}}

\newcommand{\lt}{\tilde{\lambda}}

\begin{document}
\begin{frontmatter}
	
	\title{Determinantal Point Processes in the Flat Limit: Extended L-ensembles, Partial-Projection DPPs and Universality Classes}
	
	\runtitle{Determinantal Point Processes in the Flat Limit}
	
	\begin{aug}
		\author[A]{\fnms{Simon} \snm{Barthelmé}},
		\author[A]{\fnms{Nicolas} \snm{Tremblay}},
		\author[B]{\fnms{Konstantin} \snm{Usevich}}
		\and
		\author[A]{\fnms{Pierre-Olivier} \snm{Amblard}}
		
		\address[A]{CNRS, Univ. Grenoble Alpes,  Grenoble INP, GIPSA-lab}
		
		\address[B]{Universit\'{e} de Lorraine and CNRS, CRAN (Centre de Recherche en
			Automatique en Nancy)}
	\end{aug}
	
	\medskip
	\flushleft{
	\textcolor{red}{This paper has now been divided into two parts:
	\begin{itemize}
		\item Part I details extended L-ensembles as a new representation for DPPs, is entitled ``\emph{Extended L-ensembles: a new representation for Determinantal Point Processes}'' has now been published here~\cite{tremblay_extended_2021}. This first part is independent of the flat limit problem.
		\item Part II studies the flat limit of L-ensembles (that is best described by extended L-ensembles), is entitled ``\emph{Determinantal Point Processes in the Flat Limit}'' has now been published here~\cite{barthelme_determinantal_2021}. 
	\end{itemize}
	In both papers, we have added examples and illustrations, and removed some technical material, in order to clarify our main results.}}
	\medskip
	
	\begin{abstract}
		Determinantal point processes (DPPs) are repulsive point processes where the
		interaction between points depends on the determinant of a positive-semi definite matrix. The contributions of this paper are two-fold. 
		
		First of all, we introduce the concept of extended L-ensemble, a novel representation of DPPs. These extended L-ensembles are interesting objects because they fix some pathologies in the usual formalism of DPPs, for instance the fact that projection DPPs
		are not L-ensembles. Every (fixed-size) DPP is an (fixed-size) extended
		L-ensemble, including projection DPPs. This new formalism enables to introduce and analyze a subclass of DPPs, called partial-projection DPPs. 
		
		Secondly, with these new definitions in hand, we first show that partial-projection DPPs arise as perturbative limits of L-ensembles, that is, limits in $\flatlim$ of L-ensembles based on matrices of the form $\varepsilon \bA + \bB$ where $\bB$ is low-rank. We generalise this result by showing that partial-projection DPPs also arise as the limiting process of L-ensembles based on kernel matrices, when the kernel function becomes flat (so that every point interacts with every other point, in a
		sense). We show that the limiting point process depends mostly on the smoothness of the kernel function. In some cases, the limiting process is even universal, meaning that it does not depend on
		specifics of the kernel function, but only on its degree of smoothness.
	\end{abstract}
\end{frontmatter}

\pagebreak
\tableofcontents
\pagebreak

\section*{Introduction}
\label{sec:intro}

Determinantal point processes are by now perhaps the most famous example of
repulsive point processes. They first appeared as a model for the position of
fermionic particles in an energy potential \cite{Macchi:CoincidenceApproach},
but also occur in random matrix theory and graph theory. More recently they have
been advocated in machine learning as a way of providing samples with guaranteed
diversity \cite{kulesza2012determinantal}. In that framework, one has a set of
$n$ items, and one desires to produce a subset $\X$ of size $m \ll n$ such that
no two items in $\X$ are excessively similar.  A key aspect of DPPs is that
``diversity'' is defined relative to a notion of similarity represented by a
positive-definite kernel. For instance, if the items are vectors in $\R^d$,
similarity may be defined via the squared-exponential (Gaussian) kernel:
\begin{equation}
  \label{eq:squared-exp-kernel}
  \kappa_{\varepsilon}(\vect{x},\vect{y}) = \exp \left( - \varepsilon \norm{\vect{x}-\vect{y}}^2 \right)
\end{equation}
Here $\vect{x}$ and $\vect{y}$ are two items, and similarity is a decreasing
function of distance.

The class of DPPs can be separated into two subclasses: a large subclass called \emph{L-ensembles} grouping the DPPs that can sample the empty set (the probability of sampling the empty set is strictly positive); and a much smaller class grouping DPPs that cannot (the probability is strictly zero). Precise definitions are to be found in section~\ref{sec:definitions}. 

By definition, an L-ensemble based on the $n\times n$ kernel matrix $\bL=[\kappa_{\varepsilon}(\vect{x}_i,\vect{x}_j)]_{i,j}$ is a distribution over random subsets $\X$ such that:
\[ \Proba(\X) \propto \det [\kappa_{\varepsilon}(\vect{x}_i,\vect{x}_j)]_{\vect{x}_i,\vect{x}_j \in \X^2} \]
If two or more points in $\X$ are very similar (in the sense of the kernel function), then the matrix $\bL_\X = [\kappa_{\varepsilon}(\vect{x}_i,\vect{x}_j)]_{\vect{x}_i,\vect{x}_j \in \X^2}$
has rows that are nearly collinear and the determinant is small (see fig. \ref{fig:illus-dpp}). This in turns
makes it unlikely that such a set $\X$ will be selected by the L-ensemble. 

\begin{figure}[!htbp]   
  \begin{center}
    \includegraphics[width=\textwidth]{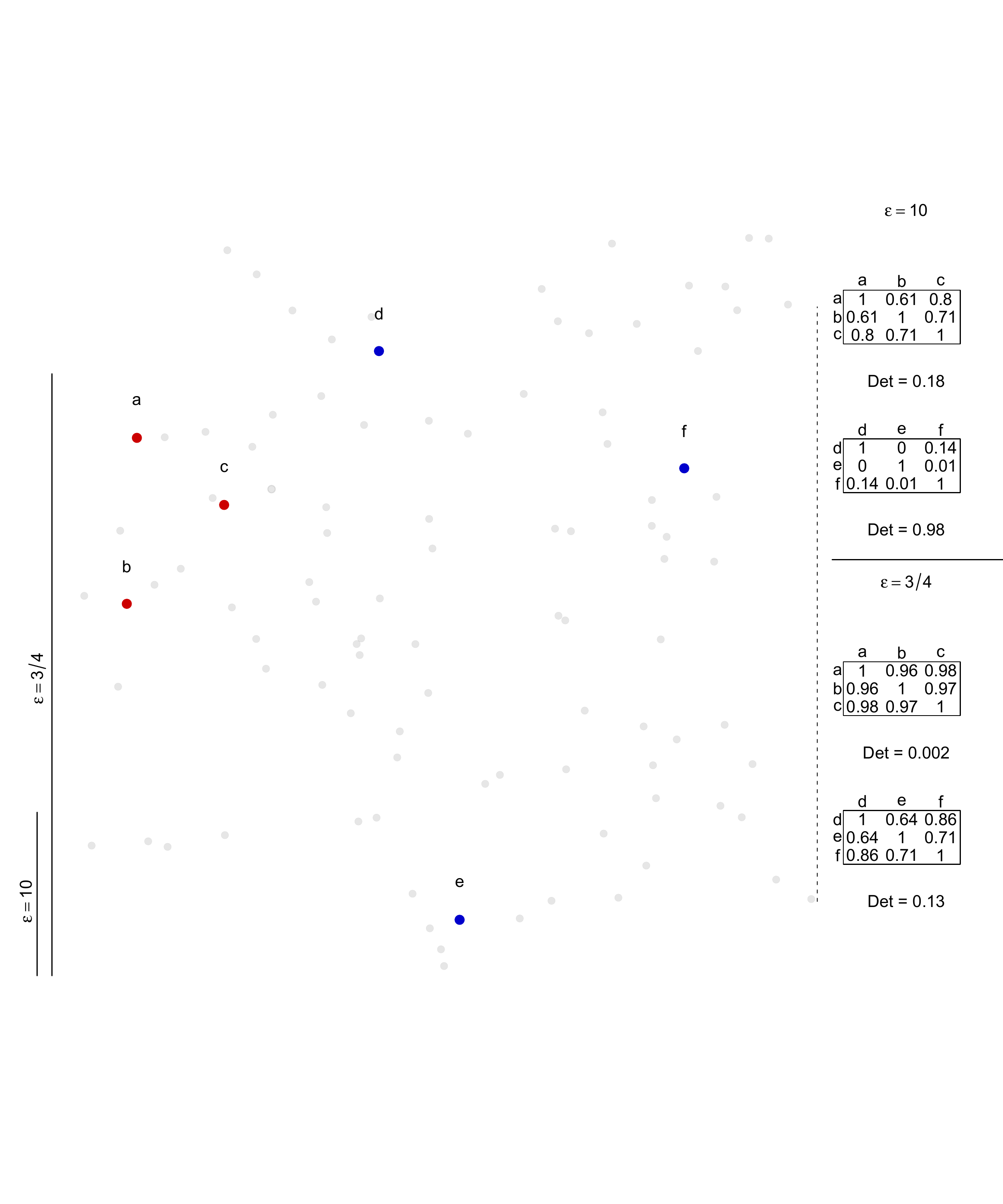}
  \end{center}
\caption{
    L-ensembles generate random subsets with probability
    proportional to the determinant of a kernel matrix. The ground set $\Omega$
    represents the items to sample from: in this figure the points in light
    gray. Two possible subsets of size 3 are represented in blue and red,
    respectively. An L-ensemble may be defined using the Gaussian kernel (eq.
    \ref{eq:squared-exp-kernel}), for instance, and $\varepsilon$ controls the
    length-scale of the kernel (the ``standard deviation'' of the Gaussian
    kernel equals $\frac{1}{2\sqrt{\varepsilon}}$, represented by the two vertical bars
    on the left). On the right, we show the kernel
    matrices corresponding to the two sets, for two values of $\varepsilon$. The
    set $X = \{a,b,c\}$ contains points that are much closer together than the set
    $X' = \{d,e,f\}$: accordingly, the kernel matrix formed from $X'$ is much better
    conditioned than one formed from $X$, which is reflected in the determinant. An L-ensemble is therefore much
    more likely to sample $X'$ than $X$. }
  \label{fig:illus-dpp} 
\end{figure}

\begin{figure}
\begin{center}
    \includegraphics[width=\textwidth]{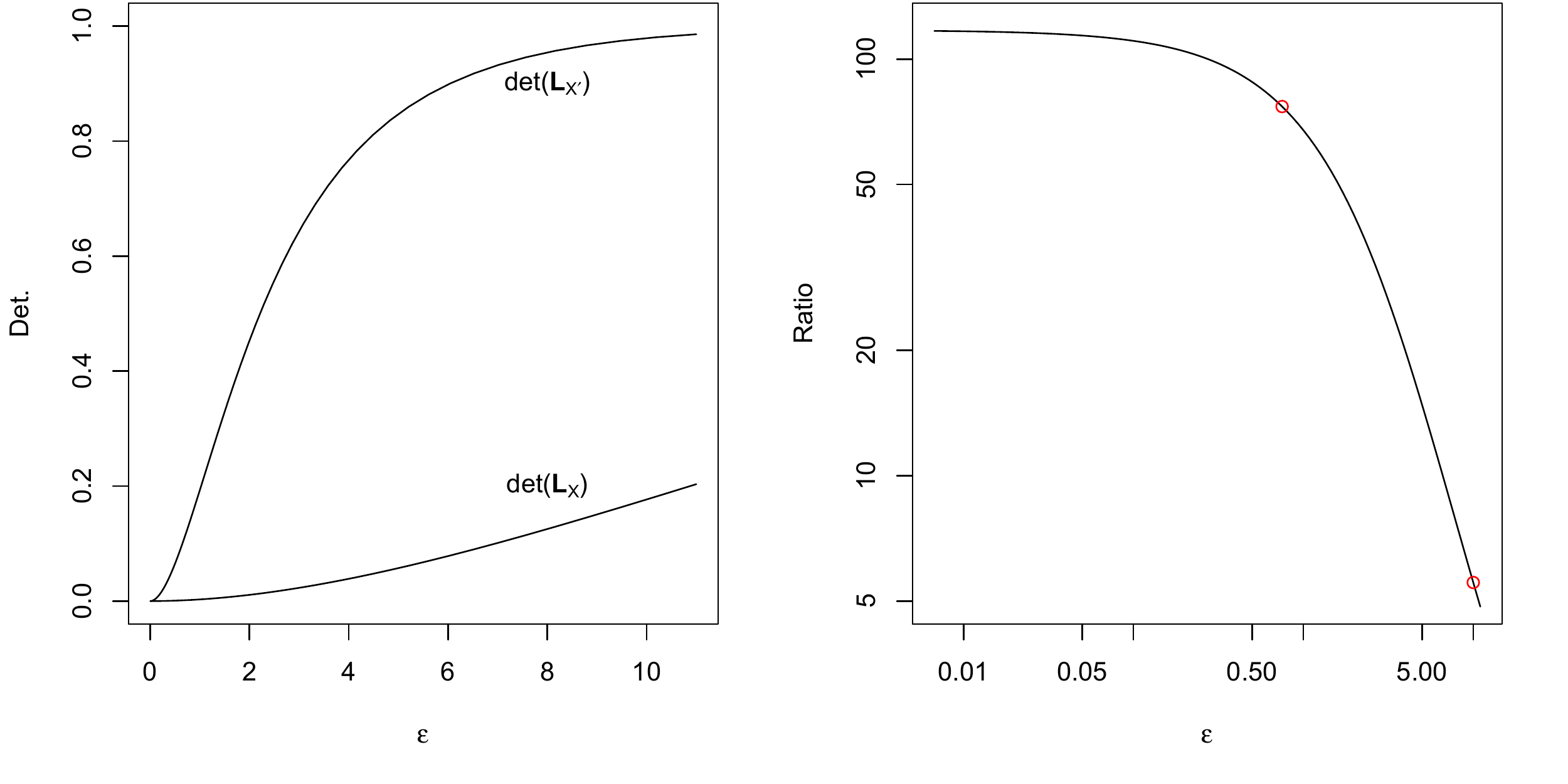}
  \end{center}

  \caption{
In this article, we study the limit of L-ensembles as $\flatlim$, meaning that
the length-scale of the kernel goes to infinity. Although all kernel
matrices are equal to the constant matrix in that limit, and all
determinants go to 0, \emph{ratios} of two determinants go to a fixed quantity.
This is what the figure shows: the left-hand part shows the determinants of
the two kernel matrices from fig. \ref{fig:illus-dpp}  corresponding to $\X$ and $\X'$, as as function of
$\varepsilon$. The right-hand part shows their ratio. The two red dots are for $\varepsilon=10$ and
$\varepsilon=3/4$.
As $\flatlim$, set $\X'$  is roughly 100 times more
    likely than set $\X$ to be sampled. }
  \label{fig:illus-dpp-flat-lim} 
\end{figure}

Importantly, how fast similarity decreases with distance
is determined by the inverse-scale parameter $\varepsilon$. Like other kernel
methods, L-ensembles are plagued with hyperparameters and finding the ``right'' value
for $\varepsilon$ is no easy task. Partial answers to this difficulty may be obtained via the study of the so-called ``flat limit'', originally studied by Driscoll \& Fornberg in Radial Basis
Function interpolation, which simply consists in taking $\varepsilon \rightarrow 0$ in
eq. \eqref{eq:squared-exp-kernel} (or similar kernels). 

This paper addresses the question of the behaviour of L-ensembles based on similarity kernels for which $\varepsilon \rightarrow 0$. To this end, we build upon the work in~\cite{BarthelmeUsevich:KernelsFlatLimit}, where general results on the spectral properties of kernel matrices are established in the flat limit.

\begin{figure}
	\centering
  \includegraphics[width=0.8\textwidth]{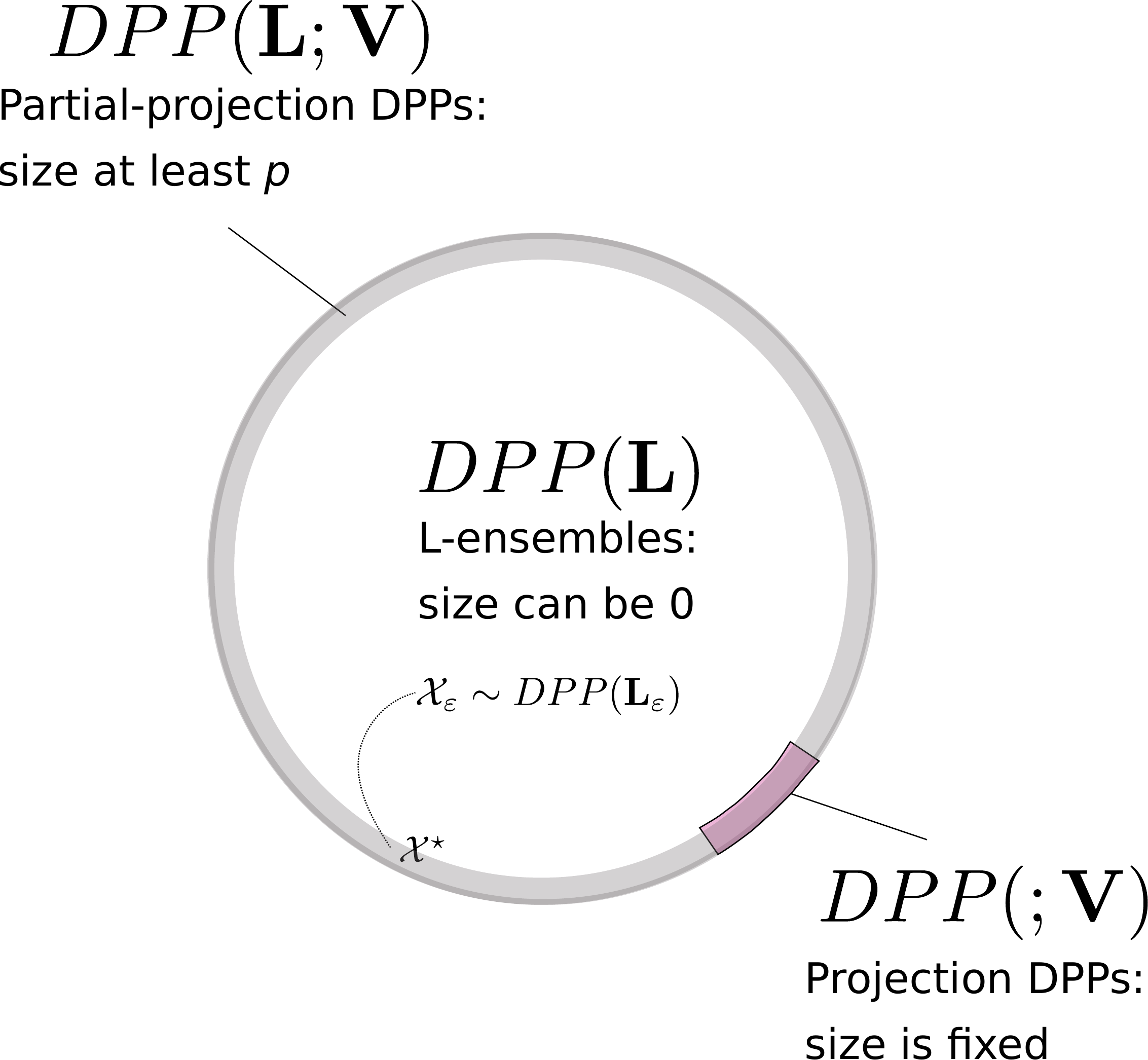}
  \caption{
    An overview of the space of DPPs as studied in this article. The whole sphere represents the class of DPPs. The interior of the sphere represents the (large) subclass of L-ensembles. In L-ensembles, the size of the point process is always allowed to be zero. 
    The gray boundary represents DPPs that do not allow the empty set. In such processes, $|\X| \geq p>0$ almost surely. We call such DPPs ``partial-projection DPPs'' for reasons explained in section
    \ref{sec:ppDPPs}. A special case of partial projection DPPs are the
    projection DPPs, in which $|\X|$ is fixed. While L-ensembles are based on a
    single matrix $\bL$, we show in section 2 that partial projection DPPs can be
    defined based on a pair of matrices $\bL$ and $\bV$. This is
    in fact a valid representation for all DPPs: in L-ensembles the $\bV$ part
    of the pair is empty, and in projection DPPs it is $\bL$ that is empty. We
    call this generic representation of DPPs ``extended L-ensembles''. Sections
    \ref{sec:intro} and \ref{sec:ppDPPs} introduce these concepts. 
    In the second part of the manuscript (section \ref{sec:ppDPP-as-limits} and
    onwards), we study limits of L-ensembles $\Xe$ as $\flatlim$. As illustrated
    here, many interesting limits of L-ensembles ``hit the boundary'' and become
    partial-projection DPPs, which is why we need the extended L-ensemble
    representation.
  }
  \label{fig:venn-diagram} 
\end{figure}

\subsection*{Contributions}
\label{sec:contrib}

Our contributions go beyond a study of the flat limit. As it turns out, the limit processes belong to a specific subclass of DPPs we call ``partial-projection DPPs'', which precisely groups all DPPs that are not L-ensembles (thus sampling sets with size always strictly superior to zero). In order to manipulate joint probability mass functions for DPPs in this subclass, we have to introduce our first contribution: extended L-ensembles.

Section~\ref{sec:ppDPPs} is devoted to the definition of extended
L-ensembles, a novel representation of DPPs that we believe is interesting in
itself. Extended L-ensembles provide a unified description of DPPs: whereas not
all DPPs are L-ensembles, all DPPs are extended L-ensembles. In addition, they
let us write easy-to-understand, explicit formulas for joint probabilities even in cases where the DPP at hand is not an L-ensemble.

With these definitions in hand, we first study the limiting process of an L-ensemble based on the perturbed matrix $\varepsilon \bA + \bB$ (where $\bB$ is low-rank) as $\varepsilon$ tends to zero. We show that this limiting process is a partial projection DPP; meaning that partial-projection DPPs form in a sense the exterior boundary of the space of L-ensembles.  Such perturbative limits form the topic of section~\ref{sec:ppDPP-as-limits}. Figure \ref{fig:venn-diagram} summarises some of the main concepts used here. 

The next sections are devoted to the flat limit proper, that is: the study of the limiting process of an L-ensemble based on a kernel matrix, as $\varepsilon$ tends to zero.  We show the following results:
\begin{itemize}
\item Surprisingly, in the flat limit, such L-ensembles stay well-defined (see fig. \ref{fig:illus-dpp-flat-lim} for an intuitive explanation of why that occurs)
\item The limiting process depends mostly on the smoothness of the kernel function 
\item In particular cases (depending on the dimension $d$), they exhibit universal
  limits, i.e. all kernels within the same smoothness class lead to the same
  limiting L-ensemble 
\end{itemize}

As an example of our results, we can prove the following (the notation is made precise later):
let $\Omega \subset \R$ (a finite set of points on the real line), and $\X$ an L-ensemble on $\Omega$. Let $\kappa_{\varepsilon}$ be a kernel
function that is $C^\infty$ in both $x$ and $y$ at $\vect{0}$ and analytic in $\varepsilon$  (e.g., the Gaussian). Pick
an odd integer $p<2|\Omega|-1$. Then, applying Thm~\ref{thm:Xe_varying_univariate}, as $\flatlim$ the L-ensemble based on the matrix 
$[\varepsilon^{-p}
\kappa_{\varepsilon}({x_i},{x_j})]_{{x}_i,{x}_j \in \Omega^2}$
has the law:
\begin{align}
\label{eq:first_example}
     p\left(\X= \{x_1, \ldots, x_m\} \right) =
  \begin{cases}
    \frac{1}{Z} \prod_{i<j} (x_i -
    x_j)^2 & \mbox{\ if\ } m = \frac{p+1}{2}, \\
    0& \mbox{\ otherwise}.
  \end{cases}
\end{align}
On the other hand, if the kernel function is only once differentiable at 0, e.g.
with $\kappa_\varepsilon(x,y) = \exp(-\varepsilon|x-y|)$, then taking the limit
of the L-ensemble based on the matrix $[\varepsilon^{-1}
\kappa_{\varepsilon}({x}_i,{x}_j)]_{({x}_i,{x}_j) \in
  \Omega^2}$ we obtain a different process, with joint probability:
\[     p\left(\X= \{ x_1, \ldots, x_m \} \right) =
  \begin{cases}
    \frac{1}{Z} \gamma^m \prod_{i=1}^{m-1} (x_{i+1} -
    x_i), & \mbox{\ if\ } m \geq 1, \\
    0, & \mbox{\ otherwise},
  \end{cases}
\]
where we have ordered the points so that $x_1 \leq x_2 \leq \ldots \leq x_m$.
Whereas the previous limit was completely universal, in the sense that the
limiting distribution is the same for all $C^\infty$ kernels, this other limit
is almost universal, but not quite: the limit is the same for all $C^1$ kernels,
except for the value of  $\gamma$ which depends on the kernel. 

Our results are much more general, and the general case involves some
subtleties. The main (and most general) results on the flat
limit are Th. \ref{thm:nd-finite-smooth-magic-case}, Th. \ref{thm:general-case-smooth-fixed-size},
and Th. \ref{thm:Xe_varying_multivariate}, but the statements require that we
set up a bit of notation. In addition, theorem
\ref{thm:equivalence-extended-spectral} is a generalisation of the Cauchy-Binet
lemma which may be of independent interest.

Because the results require a bit of background to explain properly, we show in
fig. \ref{fig:teaser} a teaser meant to motivate the reader to pursue reading at
least until section \ref{sec:universal-limits}, where the key to the mystery is
revealed. The teaser shows counter-intuitive behaviour of L-ensembles in the flat limit
(in dimension 2). 

\begin{figure}
   \begin{center}
\includegraphics[width=0.8\textwidth]{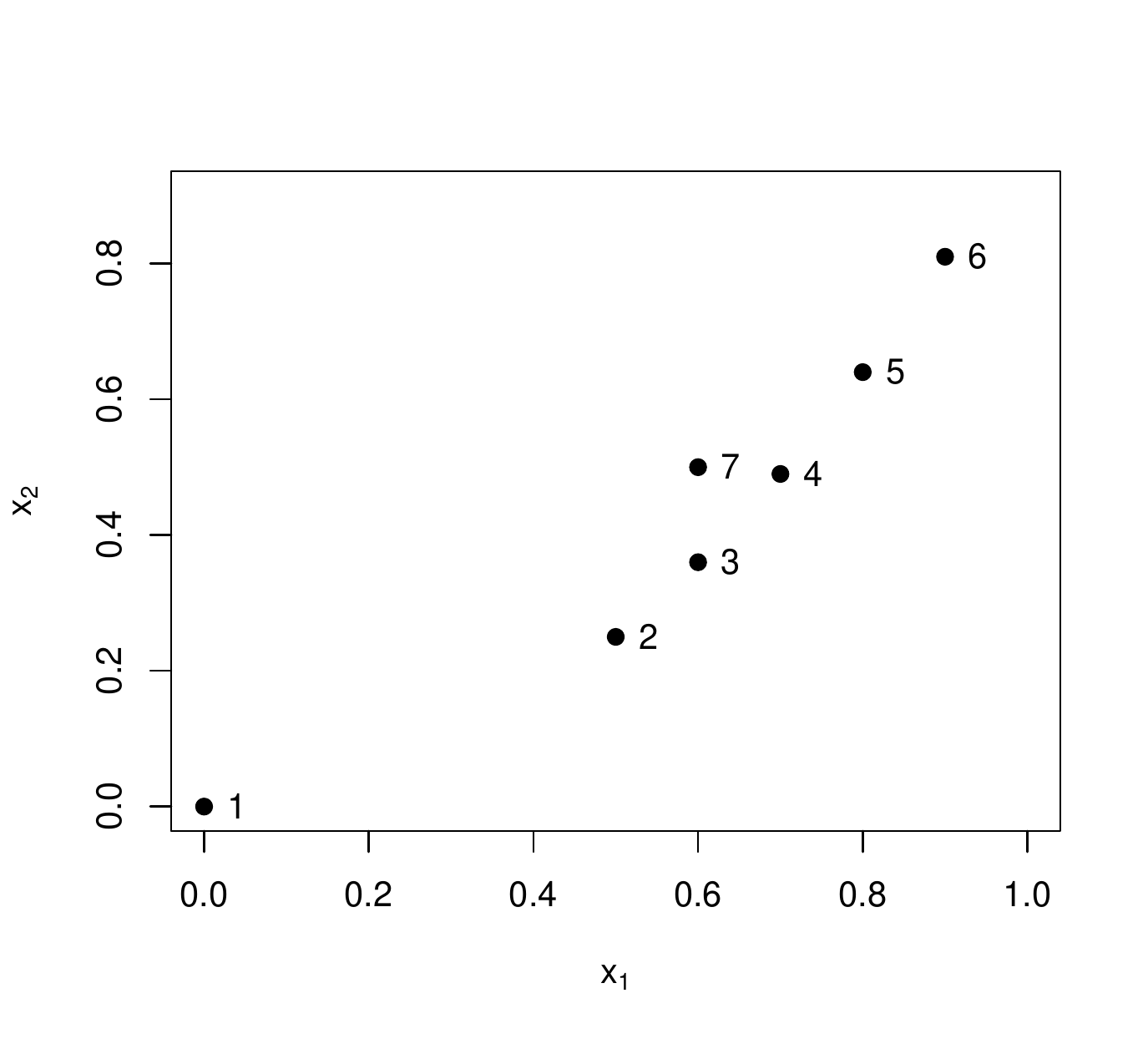}
  \caption{Suppose a (fixed-size) L-ensemble is used to sample 6 of the 7 labelled
    points shown on the figure. With a Gaussian kernel, as $\flatlim$, the set
    $X = \{1,2,3,4,5,6\}$ has a probability 0 of being sampled, while the set
    $X' = \{2,3,4,5,6,7\}$, which is less spread-out, has a small but non-zero
    probability of being sampled. With an exponential kernel, on the other hand,
    both sets have a non-zero probability of being sampled, but in this case $X$
    is much \emph{more} likely to be sampled that $X'$. The explanation for that
    counter-intuitive behaviour is to be found in section
    \ref{sec:universal-limits}.} \label{fig:teaser}
   \end{center}
\end{figure}

The limitations of our results are as follows. We focus on stationary kernels,
and only look at finite DPPs, leaving aside the continuous case. All results
should extend to continuous DPPs on a compact subset of $\R^d$, with the
appropriate change in notation. The case of continuous DPPs on a \emph{non-compact} subspace of $\R^d$ appears to us harder to deal with.

\subsection*{Practical implications}

The practical-minded reader might object to the abstract nature of this work.
However, we stress that flat limits are an elegant way of partially answering
the questions of hyper-parameter tuning, and, to a lesser extent, the choice of
similarity function.

One outcome of this work is that as $\flatlim$, DPPs have limits that
are sensible, repulsive and so should behave reasonably in applications. One
advantage of directly sampling from the limiting DPP is that there is no spatial
scaling parameter to choose from. The only one that remains is how many points
one wishes to sample. This assumes of course that one has chosen a particular
kernel function, which leads us to our secound point. 

The second conclusion of our work is that what the exact kernel is, matters much
less than what its smoothness order is. If one where to speculate based on the
results in the unidimensional case, kernels with low regularity lead to mostly
local repulsion whereas kernels with high regularity lead to a more global form
of repulsion; and this is borne out as well by some numerical evidence. Kernels
with high regularity lead to some surprising long-distance repulsiveness
properties, as fig. \ref{fig:teaser} illustrates.

In addition, we suspect that there are computational implications of our results
as well, enabling faster sampling of DPPs, but we leave this for future work.

\subsection*{Structure of the paper}
We begin with some definitions and background in section \ref{sec:definitions}.
Section \ref{sec:ppDPPs} introduces extended L-ensembles and partial-projection DPPs 
and gives some major properties. Partial-projection DPPs
arise as limits of L-ensembles, and section \ref{sec:ppDPP-as-limits} explains how in a
simple case of an L-ensemble based on a linearly perturbed matrix. Some of the results proved there
should help understand what happens in the flat limit.

For clarity, flat limit results are given in increasing order of complexity. We
begin with results on the limits of \emph{fixed-size} L-ensembles (the ``k-DPPs'' of
\cite{KuleszaTaskar:FixedSizeDPPs}), because these results are much easier to state and
serve as a building block for the case of \emph{variable-size} L-ensembles. Thus,
section \ref{sec:univariate-results} and section \ref{sec:results-multivariate} study fixed-size L-ensembles in the flat limit. For pedagogical reasons, we begin with univariate results (where the points are a
subset of the real line), before giving the results for the multivariate case,
which require some background on multivariate polynomials. Limits of varying-size L-ensembles
are covered in section \ref{sec:varying-size}, which again has a subsection on the univariate case that serves as a
warm-up for the more difficult multivariate case. 

\section{Definitions and background}
\label{sec:definitions}

We briefly recall some  definitions. For details we refer the reader to
\cite{Barthelme:AsEqFixedSizeDPP} and  \cite{KuleszaTaskar:FixedSizeDPPs}. All
of the results below are classical.

DPPs are based on determinants of kernel matrices, so we begin with some
material on kernel functions and determinants. We then introduce DPPs along with
fixed-size DPPs, a useful variant (as well as L-ensembles and fixed-size L-ensembles). Our proofs require that we work with asymptotic
expansions of probability mass functions, which we do via two lemmas that we 
introduce. We then give some very simple results from matrix
perturbation theory. They are not necessary for our proofs but help build an
understanding of the limits we investigate. Finally, we provide the necessary background material
on multivariate polynomials, as they are very important for flat limits and appear here or there in our developments. 

\subsection{Kernels, smoothness orders}
\label{sec:intro-kernels}

We only outline the basic concepts needed to express the results from
\cite{BarthelmeUsevich:KernelsFlatLimit}, which our analysis is based on. For more on kernels the reader is
invited to consult \cite{stein1999interpolation} or \cite{wendland2004scattered}.
A kernel is a positive definite function $\kappa : \R^d \times \R^d \rightarrow \R$. We call the kernel
stationary if $\kappa(\vect{x},\vect{y}) = f(\norm{\vect{x}-\vect{y}}_2)$ for
some function $f$, i.e. it only depends on the (Euclidean) distance between
$\vect{x}$ and $\vect{y}$.
We assume further that $f$ is analytic\footnote{We choose this assumption for
  simplicity, but it can be relaxed to an assumption of differentiability up to a required order. } at 0, and expand it as:
\begin{equation}
  \label{eq:kernel-expansion}
  f(\norm{\vect{x}-\vect{y}}_2) =  f_0 +  f_1 \norm{\vect{x}-\vect{y}}_2 + f_2 \norm{\vect{x}-\vect{y}}_2^2  + f_3 \norm{\vect{x}-\vect{y}}_2^3 +  \ldots
\end{equation}
where $f_i = \frac{f^{(i)}(0)}{i!} $, i.e. the rescaled derivatives at 0 of $f$. The
smoothness order of the kernel is defined with respect to the \emph{odd}
derivatives of $f$ at 0. Specifically:
\begin{definition}
  The smoothness order $r$ of a stationary kernel $\kappa(\vect{x},\vect{y}) =
  f(\norm{\vect{x}-\vect{y}}_2)$ is defined as:
  \begin{equation}
    \label{eq:smoothness-order}
    r = \min \{r | f_{2r-1} \neq 0\}
  \end{equation}
  i.e, the smallest $r$ such that the $r$-th \emph{odd} derivative is non-zero. 
\end{definition}

A kernel like the squared-exponential (eq. \eqref{eq:squared-exp-kernel})
depends on the squared distance and so has $r = \infty$. We call such kernels
\emph{completely smooth}. Kernels with finite values of $r$ are called
finitely smooth (f.s.). An example of a kernel with $r=1$ is the exponential
kernel:
\begin{equation}
  \label{eq:exp-kernel}
  \kappa_{\varepsilon}(\vect{x},\vect{y}) = \exp \left( -\varepsilon \norm{\vect{x}-\vect{y}}_2 \right)
\end{equation}
An example of a kernel with $r=2$ is:
\begin{equation}
  \label{eq:r2-kernel}
  \kappa_{\varepsilon}(\vect{x},\vect{y}) = \left(1+\varepsilon \norm{\vect{x}-\vect{y}}_2\right)\exp \left( -\varepsilon \norm{\vect{x}-\vect{y}}_2 \right)
\end{equation}
The Mat{\`e}rn kernels \cite{stein1999interpolation}, popular in spatial statistics, are a generic family of
kernels which have $r$ as a parameter. Other examples of finitely-smooth kernels
can be found in our numerical results, for instance in fig. \ref{fig:convergence-cond-1d}. 

\subsection{Some determinant lemmas}
\label{sec:determinant-lemmas}

 Let $\matr{A}$ be a $n \times n$ matrix, and
$Y$, $Z$ be two subsets of indices. Then $\matr{A}_{Y,Z}$ is the submatrix
of $\matr{A}$ formed by retaining the rows in $Y$ and the columns in $Z$.  Furthermore, $\matr{A}_{:,Y}$ (resp. $\matr{A}_{Y,:}$) is the matrix made of the full columns (resp. rows) indexed by $Y$. Finally, we
let $\matr{A}_Y = \matr{A}_{Y,Y}$.
Also, for a matrix $\matr{V}$, by $\mspan(\matr{V})$ we denote its column span, and by $\orth(\matr{V})$ the orthogonal complement of $\mspan(V)$. 

We shall need a number of basic results on determinants. The Cauchy-Binet lemma is central to the theory of DPPs and generalises the
well-known relationship $\det(\matr{A}\matr{B})=\det(\matr{A})\det(\matr{B})$
(for square $\matr{A}$ and $\matr{B}$) to rectangular matrices. 

\begin{lemma}[Cauchy-Binet]
  \label{lem:cauchy-binet}
  Let $\matr{M} = \matr{A} \matr{B}$, with $\matr{A}$ a $m \times n$ matrix, $\matr{B}$ a $n \times m$ matrix. Then:
  \begin{equation}
    \label{eq:CauchyBinet}
    \det \matr{M} = \sum_{Y, |Y| = m} \det \matr{A}_{:,Y} \det \matr{B}_{Y,:}
  \end{equation}
  where the sum is over all subsets $Y \subseteq \{1, \ldots, n\}$ of size $m$.
\end{lemma}
We will also frequently use the following simple corollary of the Cauchy-Binet lemma.
\begin{corollary}
Let $\matr{M} =  \matr{U} \matr{\Lambda} \matr{U}^{\top}$, where $\matr{U}$ is $m\times n$ , $n \ge m$ and $\matr{\Lambda}$ is a diagonal matrix.
Then:
\[
\det \matr{M} = \sum_{Y, |Y| = m} (\det (\matr{U}_{:,Y}))^2 \det(\matr{\Lambda}_{Y}).
\]
\end{corollary}
The next result is a well-known determinantal counterpart of the
Sherman-Woodbury-Morrisson lemma:

\begin{lemma}
  \label{lem:low-rank-update-det}
  Let $\matr{A}$ be an invertible matrix of size $n \times n$, $\matr{U}$
  of size $n \times m$, and $\matr{W}$ an invertible matrix of size $m \times m$. Then it holds that:
  \begin{equation}
    \label{eq:determinant-update}
    \det(\matr{A}+\matr{UWU}^\top)=\det(\matr{A})\det(\matr{W})\det(\bW^{-1}+\matr{U}^\top\matr{A}^{-1}\matr{U}).
  \end{equation}
\end{lemma}

Finally, a related lemma is useful for block matrices:

\begin{lemma}
  \label{lem:block-det}
  Let $\matr{M} =
  \begin{pmatrix}
    \matr{A} & \matr{U} \\
    \matr{U}^\top & \matr{W}
  \end{pmatrix}
  $,  with $\matr{A}$ invertible.  Then
  \begin{equation}
    \label{eq:block-det}
    \det(\matr{M}) = \det(\matr{A})\det(\matr{W}-\matr{U}^\top\matr{A}^{-1}\matr{U}).
  \end{equation}
\end{lemma}

The next two lemmas concern so-called ``saddle-point matrices'', and are proved in
\cite[Appendix A]{BarthelmeUsevich:KernelsFlatLimit}.  

\begin{lemma}[{\cite[Lemma 3.10]{BarthelmeUsevich:KernelsFlatLimit}}]
  \label{lem:det-saddlepoint}
  Let $\bL \in \R^{n \times n}, \bV \in \R^{n \times p}$, 
with $\bV$ of full column rank and $p \leq n$. 
  Let 
  $\bQ \in \R^{n \times (n-p)}$ be an  orthonormal basis for $\orth({\bV})$ (i.e., $\bQ^{\top} \matr{V} = \matr{0}$, $\rank(\bQ) = n-p$).  
  Then:
  \begin{equation}
    \label{eq:det-saddlepoint}
  \det       \begin{pmatrix}
    \bL & \bV \\
    \bV^\top & \matr{0} 
  \end{pmatrix} = (-1)^p \det(\bV^\top\bV)\det(\bQ^\top\bL\bQ).
\end{equation}
\end{lemma}

In the next lemma, we use $[t^r] g(t)$ to denote the coefficient corresponding to $t^r$ in the power series $g$. 
For instance, if $g(t) = 1-t^2+2t^3$, then $[t^0]g(t) = 1$ and $[t^3]g(t) = 2$. 
\begin{lemma}[{\cite[Lemma 3.11]{BarthelmeUsevich:KernelsFlatLimit}}]\label{lem:det-coef-polynomial}
Let $\bL \in \R^{n \times n}$ and $\bV \in \R^{n\times p}$.
Then:
\[
[t^p]\det( \bL +t \bV \bV^\top ) =  (-1)^{p} \det\begin{pmatrix}
      \matr{L} &  \matr{V} \\
      \matr{V}^{\T} & 0 
    \end{pmatrix}. 
\] 
\end{lemma}

\begin{remark}\label{rem:det-coef-degree}
The polynomial  $g(t) = det( \bL +t \bV \bV^\top )$ is of degree at most $p$, i.e., lemma~\ref{lem:det-coef-polynomial} gives the coefficient for the highest possible power of $t$.
While this remark is missing in the original statement of
lemma~\ref{lem:det-coef-polynomial} (see \cite[Lemma
3.11]{BarthelmeUsevich:KernelsFlatLimit}), it can be easily verified by
inspecting the proof of the lemma in \cite[Appendix
A]{BarthelmeUsevich:KernelsFlatLimit}.
\end{remark}

\subsection{Determinantal processes}
\label{sec:def-DPPs}
\subsubsection{DPPs}
Let $\Omega=\{ \vect{x}_1, \ldots, \vect{x}_n\} \subset \R^d$ be a collection of
vectors called the \emph{ground set}. A finite point process $\X$ is a random
subset $\X \subseteq \Omega$. Abusing notation, we sometimes use $\X$ to
designate the indices of the items, rather than the items themselves. Which one we mean should be clear from context.

\begin{definition}[Determinantal Point Process]\label{def:dpp}
	Let $\bK \in \R^{n \times n}$ be
	a positive semi-definite matrix verifying $\bm{0}\preceq \bK \preceq \bm{I}$. 
	In this context, $\bK$  is called a marginal kernel. Then, $\X$ is a DPP with
	marginal kernel $\bK$ if
	\begin{equation}
	\label{eq:def-dpp}
	\forall A\subseteq\Omega\qquad\Proba(A\subseteq\mathcal{X}) = \det \bK_A,
	\end{equation}
	where by convention, $\det \bK_{\varnothing} = 1$. 
\end{definition}
This definition is the historical one \cite{Macchi:CoincidenceApproach} and
determines what we will refer to as the \emph{class of DPPs}. However,
manipulating inclusion probabilities rather than the joint probability
distribution itself is often cumbersome. This usually leads authors to consider
a slightly less general class of DPPs: the L-ensembles \cite{borodin2005eynard}. 
\begin{definition}[L-ensemble]\label{def:dpp_via_L}
	Let $\bL \in \R^{n \times n}$ designate
	a positive semi-definite matrix. An L-ensemble based on $\bL$ is a point process $\X$ defined as
	\begin{equation}
	\label{eq:def-dpp_via_L}
	\Proba(\X=X) = \frac{\det \bL_X}{Z},
	\end{equation}
	where by convention, $\det \bL_{\varnothing} = 1$. Thus: $\Proba(\X=\varnothing) =1/Z>0$. 
\end{definition}
In Eq. \eqref{eq:def-dpp_via_L}, \fin $Z = \sum_{X \subseteq \Omega} \det \bL_X$ is a normalisation constant and can be shown~\cite{kulesza2012determinantal} to equal
$\det(\bI + \bL)$. 

L-ensembles are indeed a subclass of DPPs:
\begin{lemma}\label{lem:LisDPP}
	An L-ensemble based on the positive semi-definite matrix $\bL$ is a DPP. It is noted  $\X\sim DPP(\bL)$ and its marginal kernel verifies
	\begin{align}
	\label{eq:connection_KL}
	\bK=\bL(\bI+\bL)^{-1}.
	\end{align}
\end{lemma}
\begin{proof}
	See, \emph{e.g.}, Thm 2.2 of~\cite{kulesza2012determinantal}; or the discussion in  Appendix~\ref{sec:marginal-kernel-ppDPPs-proof}.
\end{proof}

L-ensembles are in fact a strict subset of all DPPs:
\begin{lemma}
	\label{lem:DPPsubclass}
	A DPP with marginal kernel $\bK$ is an L-ensemble if and only if $\bK$ verifies $\bm{0}\preceq\bK\prec\bI$ (note the $\prec$ sign, implying that no eigenvalue of $\bK$ is allowed to be equal to one). If $\X$ is a DPP with such a marginal kernel, then $\X\sim DPP(\bL)$, with $\bL$ verifying:
	$$\bL=\bK(\bI-\bK)^{-1}.$$
\end{lemma}
\begin{proof}
	($\Leftarrow$) If $\bK$ does not contain any eigenvalue equal to 1, then Eq.~\eqref{eq:connection_KL} inverts as $\bL=\bK(\bI-\bK)^{-1}.$ ($\Rightarrow$) We show the contraposition. If $\X$ is a DPP with a marginal kernel $\bK$ containing at least one eigenvalue equal to one, then its size $|\X|$ is necessarily larger than one (see lemma~\ref{lemma:bernoulli}). Thus, it cannot be an L-ensemble (L-ensembles have a non-null probability of sampling $\varnothing$). 
\end{proof}
\begin{remark}
As a consequence, the class of DPPs can be separated in two: the L-ensembles
(all DPPs with marginal kernel verifying $\bm{0}\preceq \bK\prec\bI$), and the rest (all DPPs with marginal kernel whose spectrum contains at least one
eigenvalue equal to one). \end{remark}

In DPPs, the size (cardinal) of $\X$, denoted by $|\X|$, is a random variable. Its
distribution is as follows \cite{Hough:DPPandIndep}:

\begin{lemma}
	\label{lemma:bernoulli}
	Let $\bm{0}\preceq\bK\preceq\bI$ be a marginal kernel with eigenvalues $\mu_1,\ldots, \mu_{n}$. Let $\X$ be a DPP with this marginal kernel. Then, $|\X|$ has the same distribution as $\sum_{i=1}^n B_i$, 
	where $B_i$ is a Bernoulli random variable with expectation
	$\E(B_i)=\mu_i$, and the $B_i$'s are distributed independently.
	In particular, the expected size of the DPP, $\E(|\X|)$, can be directly
	deduced from the above to be
  \begin{equation}
    \label{eq:trace-K}
    \E(|\X|)= \sum \mu_i = \Tr (\bK)
  \end{equation}
\end{lemma}
\subsubsection{Fixed-size DPPs}
The cardinal of a DPP is thus in general random. Such varying-sized samples are
not practical in many applications (one desires a subset of size 50, not
something of size 50 on average but which may be of size 35 or 56); which led the authors of~\cite{kulesza2011k} to define fixed-size DPPs\footnote{They are often called k-DPPs in the literature, but we prefer ``fixed-size DPPs'' in order not to overload the
	symbol $k$ too much.}
\begin{definition}[Fixed-size Determinantal Point Process]
	\label{def:fsDPP}
	A fixed size DPP of size $m$ is a DPP $\X$ conditioned on $|\X|=m$. 
\end{definition}
A subclass of fixed-size DPPs is the class of fixed-size L-ensembles: 
\begin{definition}[Fixed-size L-ensemble]
	\label{lemma:mDPP_via_L-ens}
  Let $\bm{0}\preceq\bL$ be a positive semi-definite matrix. A fixed-size L-ensemble is a point process $\X$ defined as:
  \begin{equation}
    \label{eq:def-dpp}
  \Proba(\X=X) =    \begin{cases}
 \displaystyle     \frac{\det \bL_{X}}{Z_m} & \mathrm{if} ~|X|=m,  \\
      0 & \mathrm{otherwise.}
    \end{cases} 
  \end{equation}
  where $Z_m$ is the normalisation constant. 
\end{definition}
Using the indicator function $\Ind(\cdot)$, we may rewrite Eq.~\eqref{eq:def-dpp} more compactly as:
\[  \Proba(\X=X) = \frac{\det \bL_{X}}{Z_m} \Ind(|X|=m).\]
\begin{lemma}
	A fixed-size L-ensemble is a fixed-size DPP, and we write it $\X \sim \mDPP{m}(\bL)$. 
\end{lemma}

We use the notation $\X \sim \mDPP{m}(\bL)$ to distinguish from (standard)
random-size L-ensembles.

It is important to understand that, in general, fixed-size DPPs are not DPPs, with the exception of projection DPPs (see Sec.~\ref{sec:diag-and-proj}). In particular, whereas all DPPs have a marginal kernel, fixed-size DPPs (again with the exception of projection DPPs) do not have marginal kernels: there does not exist a matrix whose principal minors are the marginal probabilities. The question of inclusion probabilities in fixed-size
DPPs is treated at length in \cite{Barthelme:AsEqFixedSizeDPP}.

The constant $Z_m = \sum_{\X,|\X|=m} \det \bL_\X$ in Eq.~\ref{eq:def-dpp} is a normalisation constant and one can show that it 
equals the $m$-th ``elementary symmetric polynomial'' of $\bL$, a quantity that
depends only on the spectrum of $\bL$, and plays an important role in the
theory of DPPs. 
\begin{lemma}[{\cite[Theorem 1.2.12]{horn1990matrix}}]
  \label{lem:esp}
  Let $\bL \in \R^{n \times n}$ be a  matrix with eigenvalues $\lambda_1,\ldots,\lambda_n$. The
  $m$-th elementary symmetric polynomial of $\bL$ is defined as:
  \begin{equation}
    \label{eq:esp}
    e_m(\bL) := \sum_{|X|=m} \prod_{i \in X} \lambda_i,
  \end{equation}
  i.e., $e_0(\bL) = 1$, $e_1(\bL) = \sum_{i } \lambda_i = \Tr(\bL) $, $e_2(\bL) = \sum_{i    < j} \lambda_i  \lambda_j, \; \ldots, e_n(\bL) = \det(\bL)$. 
  Then:
   \begin{equation} 
    \label{eq:esp-det}
  Z_m =   \sum_{X,|X|=m} \det \bL_X = e_m(\bL).
  \end{equation}
\end{lemma}
Since $e_m(\bL)$ is the sum of all the principal minors of fixed size $m$, we immediately
obtain the following corollary on the distribution of the size of an L-ensemble:
\begin{corollary}
  The probability that $\X \sim DPP(\bL)$ has size $m$ is given by:
  \begin{equation}
    \label{eq:distr-size-dpp}
    p(|\X|=m) = \frac{e_m(\bL)}{e_0(\bL)+e_1(\bL)+\ldots+e_n(\bL)}.
  \end{equation}
\end{corollary}

\begin{remark}\label{rm:mixture}
  Since a fixed-size L-ensemble is just an L-ensemble conditioned on size, an L-ensemble may also be viewed as a mixture of fixed-size L-ensembles. The size $m$ can be drawn according to its marginal
  distribution (Eq.~\eqref{eq:distr-size-dpp}), and conditional on $|\X|=m$, the fixed-size L-ensemble can be sampled.
\end{remark}

\subsubsection{Two useful special cases}
\label{sec:diag-and-proj}

There are two special cases of (fixed-size) DPPs that are useful to
study on their own, both from a practical and theoretical viewpoint. These are
the DPPs with \emph{diagonal} kernels and those with \emph{projection} kernels. 

As it will be shown in section~\ref{sec:mixture-representation}, these two examples are the key components for sampling any DPP using the mixture representation.

\paragraph{Diagonal kernels.} Diagonal L-ensembles are in a way the most basic kind of DPPs 
(although the fixed-size case is surprisingly intricate).

\begin{lemma}
  \label{lem:diagonal-dpp-bernoulli}
  An L-ensemble based on a diagonal positive semi-definite matrix $\bL$, $\Y \sim DPP(\matr{\Lambda})$ with $\matr{\Lambda}=\diag(\lambda_1\ldots,\lambda_n)$, 
  is a Bernoulli process: each event $i \in \Y$ is independent and occurs with
  probability $\pi_i=\frac{\lambda_i}{1+\lambda_i}$.
\end{lemma}
\begin{proof}
  \begin{align*}
    \Proba(\Y = Y) &= \frac{\prod_{i\in Y} \lambda_i}{\det(\bI + \matr{\Lambda})} 
               = \frac{\prod_{i\in Y} \lambda_i}{\prod_{j=1}^n (1+\lambda_j)} = \left(\prod_{i\in Y} \frac{\lambda_i}{1+\lambda_i}\right)\left(\prod_{j \in Y^c}  \frac{1}{1+\lambda_i}\right)\\
               &= \prod_{i=1}^n (\pi_i)^{B_i}(1-\pi_i)^{B_i}
  \end{align*}
  where $B_i$ is the Bernoulli variable indicating $i \in \Y$. 
\end{proof}
\begin{remark}
  For fixed-size L-ensembles this is no longer true: $\Y \sim
  \mDPP{m}(\matr{\Lambda})$ is not a Bernoulli process, as the
  events are no longer independent but indeed negatively associated. To see why,
  note that since the total size is fixed, conditional on $i \in \Y$ other
  points are less likely to be included. 
\end{remark}
\begin{remark}
  $\Y \sim \mDPP{m}(\bI)$ is a uniform sample of size $m$ without replacement. 
\end{remark}
Fixed-size diagonal L-ensembles have been studied at some length in the past, notably in the sampling survey literature. Many important features of these processes were reported in \cite{ChenDL94}.

\paragraph{Projection DPPs.}

Projection DPPs designate DPPs formed from projection matrices.
Projection DPPs have many unique features, for instance that of being both DPPs
and fixed-size DPPs. Section \ref{sec:ppDPPs} will introduce a
generalisation called ``partial projection DPPs''. 
The definition of a projection DPP is as follows: 

\begin{definition}[Projection DPP]
	Let $\bU$ be an $n \times m$ matrix with $\bU^\top \bU = \bI_m$. A projection DPP is a DPP with marginal kernel $\bK=\bU\bU^\top$.
\end{definition}

The name ``projection DPP'' comes from the fact that $\bU \bU^\top$ is a projection
matrix (its eigenvalues are 1, with multiplicity $m$, and 0 with multiplicity
$n-m$). As $\bK$'s spectrum contains at least an eigenvalue equal to 1, a projection DPP is \emph{not} an L-ensemble (see lemma~\ref{lem:DPPsubclass}). However, a projection DPP can be equivalently defined as a fixed-size L-ensemble:

\begin{lemma}[See e.g., {\cite[Lemma 1.3]{Barthelme:AsEqFixedSizeDPP}}]
  \label{lem:marginal-kernel-proj}
  Let $\bU$ be an $n \times m$ matrix with $\bU^\top \bU = \bI_m$. A projection DPP with marginal kernel $\bU \bU^\top$ is a fixed-size L-ensemble $\X \sim \mDPP{m}(\bU\bU^\top)$.
\end{lemma}

In fact, the only class of fixed-size DPPs that admit a marginal kernel are the projection DPPs. 
The next result states that a projection DPP is what one obtains when
sampling a fixed-size L-ensemble of size $m$ from a positive semi-definite matrix $\bL$ of rank $m \leq n$. 
\begin{lemma}[ See  {\cite[result 1]{Barthelme:AsEqFixedSizeDPP}}.]
  \label{lem:max-rank-dpp}
  Let $\X \sim \mDPP{m}(\bL)$, with $\rank(\bL) = m$, and let $\bU\in \R^{n \times m}$ denote an
  orthonormal basis for $\mspan \bL$.  Then, equivalently, $\X \sim
  \mDPP{m}(\bU \bU^\top)$
\end{lemma}
\begin{proof}
  Given the assumptions, we may write $\bL = \bU \bM \bM^\top \bU^\top$ with $\bU \in
  \R^{n \times m}$, and $\bM \in \R^{m \times m}$.
  Now, bearing in mind that $|\X| = m$, we have:
  \begin{align*}
    \Proba(\X = X) &\propto \det(\bL_X)  = \det(\bU_{X,:} \bM )^2 \propto  \det(\bU\bU^\top)_X
  \end{align*}
  where we used the fact that $\bU_{X,:}$ is square and that $\det(\bM)$
  is independent of $X$. 
  Note that any orthonormal basis works, for instance the eigenvectors of $\bL$ associated with a non-null eigenvalue,
  but not only: the Q factor in the QR factorisation of $\bL$ would work as well.
\end{proof}

\begin{remark}
Note that lemma~\ref{lem:max-rank-dpp} is valid only for fixed-size L-ensembles with rank of $\bL$ exactly equal to $m$.
In the case $\rank \bL > m$, the fixed-size L-ensemble $\X \sim \mDPP{m}(\bL)$ is no longer a projection DPP.
\end{remark}

\begin{remark}
  \label{rem:normalisation-constant-proj}
  The normalisation constant is particularly simple in the case of projection
  DPPs. Let $\bU\bU^\top$ denote a projection kernel. Then (trivially), $m$ of its
  eigenvalues equal $1$ and the rest are null. By lemma \ref{lem:esp},
  \[ \sum_X \det(\bU\bU^\top)_X = e_m(\bU\bU^\top) = 1 \]
  If as above $\bL = \bU \bM \bM^\top \bU^\top$ with $\bU \in \R^{n \times m}$, then
  by the same reasoning as in the proof of lemma \ref{lem:max-rank-dpp}:
  \[ \sum_X \det(\bL_X) = \det(\bM^\top\bM) \sum_X \det(\bU\bU^\top)_X  = \det(\bM^\top\bM) \]
\end{remark}

\subsubsection{Mixture representation}
\label{sec:mixture-representation}

Determinantal point processes have a well-known representation as a mixture of
projection-DPPs (also sometimes called ``elementary DPPs'' in the literature). See \cite{Barthelme:AsEqFixedSizeDPP}
for details. 
The following mixture representation (due to \cite{Hough:DPPandIndep}) is
fundamental, both for theoretical and computational purposes, since it serves as
the basis for exact sampling of DPPs. There are two variants, one for DPPs and
one for fixed-size DPPs. For the purposes of this paper, we describe here the
mixture representation of L-ensembles only.

\begin{lemma}[Mixture representation of fixed-size L-ensembles \cite{KuleszaTaskar:FixedSizeDPPs}]
  Let $\X \sim \mDPP{m}(\bL)$ be an L-ensemble based on $\bL$, and $\bL = \bU \matr{\Lambda} \bU^\top$ be the spectral decomposition of $\bL$. Then, equivalently,
  $\X$ may be obtained from the following mixture process:
  \begin{enumerate}
  \item Sample $m$ indices $\Y \sim \mDPP{m}(\matr{\Lambda})$
  \item Form the projection matrix $\bM = \bU_{:,\Y} (\bU_{:,\Y})^\top$
  \item Sample $\X | \Y \sim  \mDPP{m}(\bM)$
  \end{enumerate}
  Equivalently, the probability mass function of $\X$ can be written as:
  \begin{equation}
    \label{eq:mixture-representation-fixed}
    \Proba(\X = X) =  \frac{\Ind(|X| = m)}{e_m(\matr{\Lambda})} \sum_{Y,|Y|=m}  \det\begin{pmatrix} \bU_{X,Y}   \end{pmatrix}^2 \prod_{i \in Y} \lambda_i
\end{equation}
\end{lemma}

The mixture representation can be understood as (a) first sample which
eigenvectors to use and (b) sample a projection DPP with the selected
eigenvectors.

The counterpart for L-ensembles looks highly similar. 

\begin{lemma}[Mixture representation of L-ensembles, see e.g. \cite{kulesza2012determinantal}]\label{lem:mixture_dpp}
  Let $\X \sim DPP(\bL)$ and $\bL = \bU \matr{\Lambda} \bU^\top$. Then, equivalently,
  $\X$ may be obtained from the following mixture process:
  \begin{enumerate}
  \item Sample indices $\Y \sim DPP(\matr{\Lambda})$
  \item Form the projection matrix $\bM = \bU_{:,\Y} (\bU_{:,\Y})^\top$
  \item Sample $\X | \Y \sim  \mDPP{|\Y|}(\bM)$
  \end{enumerate}
  Equivalently, the probability mass function of $\X$ can be written as:
  \begin{equation}
    \label{eq:mixture-representation-varying}
    \Proba(\X = X) = \frac{1}{\det(\bL + \bI)} \sum_{Y}  \det\begin{pmatrix} \bU_{X,Y}   \end{pmatrix}^2 \prod_{i \in Y} \lambda_i.
  \end{equation}
\end{lemma}

The only step that varies is the first one, where we sample from
$DPP(\matr{\Lambda})$ instead of $\mDPP{m}(\matr{\Lambda})$.

\subsection{Convergence of DPPs  from asymptotic series}
\label{sec:total-variation-convergence}
In this section we specify which type of convergence is proved in this paper. Below, 
we say that a random variable $\X_\varepsilon$ converges to a random variable
$\X_\star$ in $\flatlim$ if for all outcomes $A$
\[
\Proba(\X_{\varepsilon}=A) \to \Proba(\X_{\star}=A). 
\]
Note that since our space of outcomes is finite, this definition coincides with all possible notions of convergence.
For example, it is equivalent to convergence in total variation ($\lim_{\flatlim} D_{TV}(\X_\varepsilon,\X_\star)=0$), where for discrete random variables $\X$ and $\Y$ defined on the same
space of outcomes, the total variation distance equals:
\begin{equation}
  \label{eq:total-variation}
  D_{TV}(\X,\Y) = \sum_{A} | \Proba(\X=A) - \Proba(\Y=A) |.
\end{equation}
What the results from \cite{BarthelmeUsevich:KernelsFlatLimit} provide us with are
asymptotic expansions of the determinants involved in the probability mass
functions. To connect asymptotic expansions with  convergence of random variables we shall use the following simple lemma.

\begin{lemma}
  \label{lem:TV-convergence}
  Let $\X_{\varepsilon}$ be a family of discrete random variables (e.g.,
  a discrete point process) with values in
  the finite set $\Phi$. Let \[ \Proba(\X_\varepsilon=X)=\frac{f_\varepsilon(X)}{\sum_{Y \in
        \Phi}f_\varepsilon(Y)}, \]
  where the following asymptotic expansion holds for $f_{\varepsilon}$ and an
  integer $p$, possibly negative:
  \[ f_\varepsilon(X) = \varepsilon^p(f_0(X) + \O(\varepsilon)).\]
Then  $\X_\varepsilon$ converges to the random variable    $\X_\star$ (with values in $\Phi$), defined as
\[
\Proba(\X_\star=X)=\frac{f_0(X)}{\sum_{Y \in \Phi} f_0(Y)}.
\]

\end{lemma}
\begin{proof}
By direct inspection, we have
\[
 \Proba(\X_\varepsilon=X)=\frac{f_\varepsilon(X)}{\sum_{Y \in  \Phi}f_\varepsilon(Y)} = 
  \frac{f_0(X) + \O(\varepsilon)}{ \sum_{Y\in \Phi} ( f_0(Y) + \O(\varepsilon))} \to \frac{f_0(X)}{\sum_{Y \in \Phi} f_0(Y)},
\]
where convergence holds everywhere since  $\Phi$ is a finite set.
\end{proof}We will also encounter discrete distributions in which the
(unnormalised) probability mass function $f_{\varepsilon}$ may involve different powers of $\varepsilon$. For
instance, consider the random variable $Y_\varepsilon \in \{1,2,3 \}$ with
unnormalised mass function $f_\varepsilon(Y_\varepsilon= 1) = \alpha_1\varepsilon$,
$f_\varepsilon(Y_\varepsilon= 2) = \alpha_2$, and $f_\varepsilon(Y_\varepsilon = 3) =
\alpha_3\varepsilon^{-1}$. What is the law of $Y_\varepsilon$ as $\flatlim$?
After normalisation, we have:
\begin{align*}
  \Proba(Y_\varepsilon = 1) &= \frac{\alpha_1 \varepsilon}{\alpha_1 \varepsilon + \alpha_2  + \alpha_3 \varepsilon^{-1}} = \frac{\alpha_1 \varepsilon^{2}}{\alpha_3 + \O(\varepsilon)} = \O(\varepsilon^2) \\
  \Proba(Y_\varepsilon = 2) &= \frac{\alpha_2 }{\alpha_1 \varepsilon + \alpha_2  + \alpha_3 \varepsilon^{-1}} = \frac{\alpha_2 \varepsilon}{\alpha_3 + \O(\varepsilon)} = \O(\varepsilon) \\
  \Proba(Y_\varepsilon = 3) &= \frac{\alpha_3 \varepsilon }{\alpha_1 \varepsilon + \alpha_2  + \alpha_3 \varepsilon^{-1}} = \frac{\alpha_3 }{\alpha_3 + \O(\varepsilon)} = 1 + \O(\varepsilon) 
\end{align*}
The diverging order wins, and $Y_\varepsilon$ equals 3 almost surely as
$\flatlim$.

This line of reasoning can be easily generalised to obtain the following lemma,
which simply says that the smallest order in $\varepsilon$ always wins:
\begin{lemma}
  \label{lem:TV-convergence-diverging}
  Let $\X_{\varepsilon}$ be a family of discrete random variables with values in
  the finite set $\Phi$. Let \[ \Proba(\X_\varepsilon=X)=\frac{f_\varepsilon(X)}{\sum_{Y \in
        \Phi}f_\varepsilon(Y)}, \]
  where the following Laurent series holds for $f$:
  \[ f_\varepsilon(X) = \varepsilon^{\eta_X}(f_0(X) + \O(\varepsilon)).\]
  for some $\eta_X \in \mathbb{Z}$ which may be negative. Let $\eta_{min} = \min_{X \in \Phi}
  \eta_X$ and $\Phi_{min} = \left\{ X | \eta_X = \eta_{min} \right\}$.
  Then $\X_{\varepsilon} \in \Phi_{min}$ almost surely as $\flatlim$. Moreover, $\X_\varepsilon \to \X_\star$, where $\X_\star$ is the random variable with support in $\Phi_{min}$, with
  $\Proba(\X_\star = X)=\frac{f_0(X)}{\sum_{Y \in \Phi_{min}}
    f_0(Y)}$. 
\end{lemma}

\subsection{Some matrix perturbation theory}
\label{sec:basic-matrix-perturbation}

In what follows we will be concerned with perturbed matrices. Matrix
perturbation theory is often used in statistics, but unfortunately the
perturbation problems that appear here are singular (they feature matrices that
become non-invertible at $\varepsilon=0$), and the theoretical tools we need are
a bit more exotic. In this section we introduce some basic results, a full
treatment can be found in \cite{kato1995perturbation}.

We are interested in asymptotic expansions for the eigenvalues and eigenvectors
of matrices of the form
\[ \bA(\varepsilon)=\bA_0+\varepsilon\bA_1+\varepsilon^2\bA_2+\ldots\]
Each entry in $\bA(\varepsilon)$ is analytic in $\varepsilon$, and this is
therefore known as an ``analytic perturbation'' (of $\bA_0$). The simplest case is just the
linear perturbation, also called a ``matrix pencil'':
\[ \bA(\varepsilon)=\bA_0+\varepsilon\bA_1\]
The difficulty comes from the fact that $\bA_0$ may be singular, in which case
some of the eigenvalues will be 0 at $\varepsilon=0$.

Rellich's perturbation theorem is very useful here (\cite{rellich1969perturbation}, th.
I.1.1):

\begin{lemma}
  \label{lem:rellichs-theorem}
  Let $\bA(\varepsilon)=\bA_0+\varepsilon\bA_1+\varepsilon^2\bA_2+\ldots$, with
  $\bA(\varepsilon)$ Hermitian for real $\varepsilon$ in a neighbourhood of 0. The eigenvalues $\lambda_1(\varepsilon) \ldots \lambda_n(\varepsilon)$ and
  corresponding eigenvectors $\bu_1(\varepsilon) \ldots \bu_n(\varepsilon)$ may
  be chosen analytic in a (complex) neighbourhood of 0.
\end{lemma}

Armed with Rellich's theorem, it is easy to prove some results on (singular)
linear perturbations by matching orders in series.

\begin{lemma}
  \label{lem:pert-eigenvalues}
  Let \[ \bA(\varepsilon)=\bA_0+\varepsilon\bA_1\ldots\] be an $n\times n$ positive  semi-definite
  matrix, and $\rank(\bA_0)=p < n$. Then $p$ eigenvalues of $\bA$ are $\O(1)$
  (but not $\O(\varepsilon)$),
  and the remaining  $n-p$ are $\O(\varepsilon)$. 
\end{lemma}
\begin{proof}
  This result may also be proved using the Courant-Weyl minimum-maximum principle, as in
  \cite{wathen2015eigenvalues}. Here we rely on a series expansion instead.
  Let $\lambda,\bu$ designate an eigenvalue/eigenvector pair of $\bA$. It verifies:
  \begin{equation}
    \label{eq:eigenvector-eq}
    \bA(\varepsilon)\bu(\varepsilon) = \lambda(\varepsilon)\bu(\varepsilon)
  \end{equation}
  which we may expand as:
  \begin{equation}
    \label{eq:eigenvector-eq-expanded}
    (\bA_0 + \varepsilon \bA_1 + \ldots)(\bu_0 + \varepsilon \bu_1 + \varepsilon^2 \bu_2 + \ldots) = (\lambda_0 + \lambda_1 \varepsilon + \lambda_2 \varepsilon^2 + \ldots)(\bu_0 + \varepsilon \bu_1 + \varepsilon^2 \bu_2 + \ldots)
  \end{equation}
  by Rellich's theorem.
Matching orders in $\varepsilon$, eq. \eqref{eq:eigenvector-eq-expanded}
  implies at constant order:
  \begin{equation}
    \label{eq:eigenvector-pert-order-0}
    \bA_0\bu_0 = \lambda_0 \bu_0
  \end{equation}
  implying that the first order pair $(\lambda_0,\bu_0)$ is an eigenpair of
  $\bA_0$. By hypothesis, since $\bA_0$ has rank $p$, there are $p$ eigenvalues
  of order $\O(1)$ (but not $\O(\varepsilon)$), and the rest are $\O(\varepsilon)$ or less. 
\end{proof}

Continuing the process further, we have: 
\begin{lemma}
  \label{lem:linear-pert-eigenvectors}
  Under the same condition as in lemma \ref{lem:pert-eigenvalues}, a limiting
  basis of eigenvectors can be written as $\begin{bmatrix} \bU_{0}, & \widetilde{\bU_{1}}  \end{bmatrix}$, where
  $\bU_{0}$ is an $n \times p$ matrix concatenating the $p$ eigenvectors of
  $\bA_{0}$ associated with its non-null eigenvalues, and $ \widetilde{\bU_{1}}$ concatenating
  the $(n-p)$ eigenvectors associated with the non-null eigenvalues of $\widetilde{\bA_1}=(\bI - \bU_{0}\bU_{0}^\top) \bA_1 (\bI - \bU_{0}\bU_{0}^\top)$. 
\end{lemma}
\begin{proof}
  Let $(\lambda,\bu)$ denote an eigenpair as before. If $\lambda_0$ is non-null,
  then $\bu_0$ is an non-null eigenvector of $\bA_0$. There are $p$ such
  eigenvectors, which we collect as $\bU_{0}$. If $\lambda_0=0$, then eq.
  \eqref{eq:eigenvector-pert-order-0} implies that $\bu_0$ belongs to the kernel
  of $\bA_0$. Define $\matr{P} = \bI - \bU_{0}\bU_{0}^\top$ the projector on
  $\orth(\bA_0)$ : then $\matr{P}\bu_0=\bu_0$. The
  eigenvalue equation (eq. \eqref{eq:eigenvector-eq-expanded}) implies at order
  $\varepsilon$ that:
  \[    \bA_0 \bu_1 + \bA_1 \matr{P} \bu_0 =  \lambda_1 \bu_0 \] 
  Multiplying by $\matr{P}$ on the left, we have:
  \[     \matr{P} \bA_0 \bu_1 + \matr{P} \bA_1 \matr{P} \bu_0 =  \lambda_1 \bu_0
  \]
  which from the definition of $\matr{P}$ implies
  \[ \matr{P} \bA_1 \matr{P} \bu_0 =  \lambda_1 \bu_0\]
  This last expression is an eigenvalue equation for the matrix $\widetilde{\bA_1}=\matr{P} \bA_1 \matr{P}$,
  which has at most $n-p$ non-null eigenvalues.
\end{proof}

\begin{example}
  As an example, we take the matrix
  \begin{equation}
    \label{eq:example-matrix-pert}
    \bA(\varepsilon) = 
    \begin{pmatrix}
      1 + \varepsilon& 1 \\
      1 & 1+ \frac{\varepsilon}{2}
    \end{pmatrix} = \mathbf{11}^t + \varepsilon   \begin{pmatrix}
      1 &  \\
      & \frac{1}{2}
    \end{pmatrix}.
  \end{equation}
  The results above imply that as $\flatlim$ the eigenvalues should be $\O(1)$
  and $\O(\varepsilon)$, and the associated eigenvectors proportional to $
  \begin{pmatrix}
    1 \\
    1
  \end{pmatrix}
  $ and $  \begin{pmatrix}
    -1 \\
    1
  \end{pmatrix}
  $ in the limit. Indeed, in this case the computations can be done by hand, and
  we find that
  \[ \lambda_1(\varepsilon) =  \frac{1}{4} (\sqrt{\varepsilon^2 + 16}
    + 3\varepsilon + 4) = 4 + \O(\varepsilon)\]
  and \[ \lambda_2(\varepsilon) =
    \frac{1}{4}(-\sqrt{\varepsilon^2 + 16}  + 3\varepsilon + 4) =
    0+\O(\varepsilon).\]
  The associated eigenvectors are
  \[ \begin{pmatrix}
      \frac{1}{4}(\varepsilon+\sqrt{\varepsilon^2 + 16}) \\
      1
    \end{pmatrix} = \begin{pmatrix}
      1 \\
      1
    \end{pmatrix} + \O(\varepsilon) \]
  and 
  \[ \begin{pmatrix}
      \frac{1}{4}(\varepsilon-\sqrt{\varepsilon^2 + 16}) \\
      1
    \end{pmatrix} = \begin{pmatrix}
      -1 \\
      1
    \end{pmatrix} + \O(\varepsilon)  \]
  Note that since the square-root terms can be expanded in a power series around
  16 the eigenvalues and eigenvectors are indeed analytic at 0.
  
\end{example}

\subsection{Polynomials}
\label{sec:polynomial_backgrounds}

\noindent\textbf{Multivariate polynomials.} 
Polynomials will play an important role in the paper, especially when we study
the flat limit of DPPs in section \ref{sec:univariate-results} and beyond. We
recall here the essential facts on multivariate polynomials. 

Let $\vect{x} =
\begin{pmatrix}
  x_1 & x_2 & \ldots & x_d
\end{pmatrix}^\top
 \in \R^d $. A monomial in $\vect{x}$ is a function of
the form:
\[ \vect{x}^{\vect{\alpha}} = \prod_{i=1}^d x_i^{\alpha_i}  \]
for $\vect{\alpha} \in \mathbb{N}^d$ (a multi-index). Its total degree (or degree for short) is defined as
$|\vect{\alpha}|=\sum_{i=1}^d \alpha_i$. For instance:
\[ \vect{x}^{(2,1)} = x_1^2x_2 \]
and it has degree 3. A multivariate polynomial in $\vect{x}$ is a weighted sum
of monomials in $\vect{x}$, and its degree is equal to the maximum of the
degrees of its component monomials. For instance, the following is a
multivariate polynomial of degree 2 in $\R^3$:
\[ \vect{x}^{(0,1,1)}-\vect{x}^{(1,0,1)}+2.2\vect{x}^{(1,0,0)}-1.\]

One salient difference between the univariate and the multivariate case is that
when $d>1$, there are several monomials of any given degree, instead of just
one. For instance, with $d=2$, the first few monomials are (by increasing
degree):
\begin{align*}
  \vect{x}^{(0,0)} \\
  \vect{x}^{(1,0)},\vect{x}^{(0,1)} \\
  \vect{x}^{(2,0)},\vect{x}^{(1,1)},\vect{x}^{(0,2)}\\
\end{align*}
There is a well-known formula for counting monomials of degree $k$ in dimension $d$:
\begin{equation}
  \label{eq:number-monomials}
  \HH_{k,d} = {k+d-1 \choose d-1}.
\end{equation}
The notation $\HH_{k,d}$ comes from the notion of \emph{homogeneous polynomials}. A
homogeneous polynomial is a polynomial made up of monomials with equal degree.
Therefore, the set of homogeneous polynomials of degree $k$ has dimension
$\HH_{k,d}$. The set of polynomials of degree $k$ is spanned by the sets of
homogenous polynomials up to $k$, and has dimension:
\begin{equation}
  \label{eq:dim-polynomials}
  \PP_{k,d} = \HH_{0,d} + \HH_{1,d} + \ldots + \HH_{k,d} = {k+d \choose d}.
\end{equation}
Note for instance that $\PP_{0,d}=1$ and $\PP_{1,d}=d+1$. By convention, we will also set $\PP_{-1,d}$ to be equal to $0$.\\

\noindent\textbf{Multivariate Vandermonde matrices.} We  now  define the multivariate generalisation of Vandermonde
matrices. Monomials are naturally ordered by degree, but monomials of the same
degree have no natural ordering. To properly define our matrices, we require
(formally) an ordering. For the purposes of this paper which ordering is used is
entirely arbitrary. For more on orderings, see
\cite{BarthelmeUsevich:KernelsFlatLimit} and references therein.

For an ordered set of points $\Omega = \{\vect{x}_1, \ldots, \vect{x}_n\}$, all in $\RR^d$, we define the multivariate Vandermonde matrix as:
\begin{equation}\matr{V}_{\le k} = 
  \begin{bmatrix}
    \matr{V}_{0} & \matr{V}_{1}  & \cdots & \matr{V}_{k}  
  \end{bmatrix} \in \mathbb{R}^{n\times \PP_{k,d}},
\end{equation}
where each block $\matr{V}_i \in\mathbb{R}^{n\times \HH_{i,d}}$ contains the monomials of degree $i$ evaluated on
the points in $\Omega$. As an example, consider $n=3$, $d=2$ and the ground set
\[
  \Omega = \{\left[\begin{smallmatrix}y_1 \\z_1 \end{smallmatrix}\right], \left[\begin{smallmatrix}y_2 \\z_2 \end{smallmatrix}\right], \left[\begin{smallmatrix}y_3 \\z_3 \end{smallmatrix}\right]\}.
\]
One has, for instance for $k=2$:
\[
  \matr{V}_{\le 2} =
  \left[\begin{array}{c|cc|ccc}
          1 & y_1 & z_1 & y_1^2 & y_1z_1 & z_1^2 \\
          1 & y_2 & z_2 & y_2^2 & y_2z_2 & z_2^2 \\
          1 & y_3 & z_3 & y_3^2 & y_3z_3 & z_3^2 \\
        \end{array}\right],
    \]
    where the ordering within each block is arbitrary. 

We will use $\bV_{\le k}(\X)$ to denote the matrix $\bV_{\le k}$ reduced to its lines indexed by the elements in $\mathcal{X}$. As such, $\bV_{\le k}(\X)$ has $|\X|$ rows and $\PP_{k,d} $ columns.

\section{Extended L-ensembles}
\label{sec:ppDPPs}
The goal of this section is to introduce \emph{extended L-ensembles}, a novel way of representing the class of DPPs. This representation has the advantage of giving explicit expressions for the
joint probability distribution of \emph{all} varying and fixed-size DPPs (not only varying and fixed-size L-ensembles). 

In particular, the extended L-ensemble viewpoint will provide easy-to-use,
explicit formulas for the joint probability of DPPs in cases where the spectrum
of the DPP's marginal kernel contains eigenvalues equal to $1$ (that is, in
cases where the DPP at hand is not an L-ensemble) \footnote{A formula due to
  \cite{Macchi:CoincidenceApproach} exists in this case but it is unwieldy}. 
According to Lemma~\ref{lemma:bernoulli}, those are the cases where the size of the DPP is the sum of a deterministic part (the number of such eigenvalues equal to 1) and a random part.
Such DPPs, that we will call \emph{partial projection DPPs} for reasons that will become clear when we
study their mixture representation, arise as limits of certain L-ensembles, as we will see in later sections.

\subsection{Conditionally positive (semi-)definite matrices}
\label{sec:cpd-matrices}

L-ensembles are naturally formed from positive semi-definite matrices, because
$\bL$ being positive semi-definite is a sufficient condition for $\det \bL_\X$
being non-negative. Extended L-ensembles, defined below, can accomodate a broader set
of matrices called conditionally positive semi-definite (CPD) matrices.
\begin{definition}
  A matrix $\bL \in \R^{n \times n}$ is called conditionally positive
  (semi-)definite with respect to a rank $p\geq 0$ matrix $\bV \in \R^{n \times p}$ if
  $\vect{x}^\top\bL\vect{x} > 0$ (resp., $\vect{x}^\top\bL\vect{x} \geq 0$) for all $\vect{x}$ such that $\bV^\top \vect{x}=0$.
\end{definition}
\begin{remark}
	Note that we authorize $p=0$ in the definition: in this case, the definition simply boils down to that of positive semi-definite matrices. 
\end{remark}
The set of vectors such that $\bV^\top \vect{x}=0$ is the space orthogonal to the
span of $\bV$, which we note $\orth \bV$. The conditionally positive definite
requirement may be read as a requirement for $\bL$ to be positive definite
within $\orth \bV$. Positive-definite matrices are therefore also conditionally
positive-definite, but matrices with negative eigenvalues may also be
conditionally positive-definite.
\begin{proposition}
  \label{prop:positivity-cond-eigenval}
  Let $\bL$ be conditionally positive (semi-)definite with respect to $\bV \in \R^{n \times p}$, that we suppose full column rank. Let
  $\bQ \in \R^{n \times p}$ designate an orthonormal basis for $\mspan \bV$, so that $\bI - \bQ\bQ^\top$
  is a projection on $\orth \bV$. Let $\widetilde{\bL} = (\bI - \bQ\bQ^\top)\bL(\bI -
  \bQ\bQ^\top)$. Then the eigenvalues of $\widetilde{\bL}$ are all non-negative.
\end{proposition}
\begin{proof}
  Follows directly from the definition: for all $\vect{x}$, $\vect{x}^\top (\bI - \bQ\bQ^\top)\bL(\bI -
  \bQ\bQ^\top)\vect{x} \geq 0$. 
\end{proof}

The above remark will become important when we define extended L-ensembles. The
following example of a conditionally positive definite is classical (but
surprising), and is a special case of a class of conditionally positive definite
kernels studied in \cite{micchelli1986interpolation}. We take this example
because it arises in section \ref{sec:univariate-results}:
\begin{example}[\cite{micchelli1986interpolation}]
  Let
  \[ \bD^{(1)} = [\norm{\vect{x}_i-\vect{x}_j}]_{i,j}\] the distance matrix
  between $n$ points in $\R^d$. Then $-\bD^{(1)}$ is conditionally positive
  definite with respect to the all-ones vector $\vect{1}_n$.
\end{example}
Some extensions of this example can be found in section \ref{sec:ppdpp-examples-cpdef}. 
\subsection{Nonnegative Pairs}
\label{sec:ppDPP-def}

The central object when defining extended L-ensembles is what we call a Nonnegative Pair (NNP for short). 

\begin{definition}
	\label{def:ext_Lens}
  A Nonnegative Pair, noted $\ELE{\bL}{\bV}$ is a pair $\bL \in \R^{n \times n}$, $\bV \in \R^{n
    \times p}$, $0\leq p\leq n$, such that $\bL$ is symmetric and conditionally positive semi-definite with respect to
  $\bV$, and $\bV$ has full column rank. Wherever a NNP $\ELE{\bL}{\bV}$ appears below, we consistently use the following notation: 
  \begin{itemize}
  		\item $\bQ \in \R^{n \times p}$ is an
  		orthonormal basis of  $\mspan \bV$, such that $\bI - \bQ\bQ^\top$
  		is a projector on $\orth \bV$
  		\item $\widetilde{\bL} = (\bI - \bQ\bQ^\top)\bL(\bI -
  		\bQ\bQ^\top)\in \R^{n \times n}$ is also symmetric and thus diagonalisable. From Proposition ~\ref{prop:positivity-cond-eigenval}, we know that all eigenvalues are non-negative. We will denote by $q$ the rank of $\widetilde{\bL}$. Note that $q \le n-p$ as the $p$ columns of $\bQ$ are trivially eigenvectors of $\widetilde{\bL}$ associated to $0$. We write
  		\[
  		\widetilde{\bL} = \tbU\tbLam\tbU^\top
  		\] 
  		its truncated spectral decomposition; where $\tbLam=\text{diag}(\widetilde{\lambda}_1,\ldots,\widetilde{\lambda}_q)\in \mathbb{R}^{q\times q}$ and $\tbU \in \mathbb{R}^{n\times q}$ are the diagonal matrix of nonzero eigenvalues and the matrix of the corresponding eigenvectors of $\widetilde{\bL}$, respectively.
  \end{itemize}
\end{definition}
\begin{remark}
		Again, note that we authorize $p=0$ in the definition: in this case, $\bQ=0$ and $\tbL=\bL$.
\end{remark}
Let us first formulate the following lemma, useful for the next section. 
\begin{lemma}\label{lem:equivalence_ppdpp_ensemble}
	Let $\ELE{\bL}{\bV}$ be a NNP. Then, for any subset $X \subseteq \{1, \ldots, n\}$:
	\begin{align*}
	(-1)^{p}\det \begin{pmatrix}
	{\bL}_{X} & \bV_{X,:} \\
	(\bV_{X,:})^\top & \matr{0} 
	\end{pmatrix} = (-1)^{p} \det \begin{pmatrix}
	{\tbL}_{X} & \bV_{X,:} \\
	(\bV_{X,:})^\top & \matr{0} 
	\end{pmatrix}  \ge 0.
	\end{align*}
\end{lemma}
\begin{proof} 
	Let us write $m=|X|$ the size of $X$. 
	The case $\rank \bV_{X,:} < p$ is trivial as both sides of the equality are zero. Next, assume that $\bV_{X,:}\in\mathbb{R}^{m\times p}$ is full column rank.
	If $m = p$, then $\bV_{X,:}$ is square and both sides are equal to $(\det \bV_{X,:})^2$.
	Now consider the case $m>p$.
	Let $\bQ$ be as in Definition~\ref{def:ext_Lens}, so that $\bV  = \bQ \matr{R}$ (with $\matr{R}$ nonsingular).  
	Let $\bB(X) \in \R^{m \times (m-p)}$ be the basis of $\orth(\bV_{X,:}) = \orth(\bQ_{X,:}) $. 
	Then, using lemma~\ref{lem:det-saddlepoint}, we have that
	\begin{align*}
	(-1)^p  \det \begin{pmatrix}
	\bL_{X} & \bV_{X,:} \\
	(\bV_{X,:})^\top & \matr{0} 
	\end{pmatrix}  &= \det ( (\bV_{X,:})^{\top}  \bV_{X,:})  \det ((\bB^\top(X) {\bL}_{X}\bB(X))\\
	&= \det ( (\bV_{X,:})^{\top}  \bV_{X,:})  \det ((\bB^\top(X) \tbL_{X}\bB(X) )
	=  (-1)^p\det
	\begin{pmatrix}
	\widetilde{\bL}_{X} & \bV_{X,:} \\
	(\bV_{X,:})^\top & \matr{0} 
	\end{pmatrix}, 
	\end{align*}
	where the last but one equality is from   $\tbL=(\bI - \bQ\bQ^\top)\bL(\bI - \bQ\bQ^\top)$ and the fact that  $\bB^{\T}(X)\bQ_{X,:} = 0$ and hence $(\bI - \bQ\bQ^\top)_X \bB(X) = \bB(X)$. 
	Finally, $\det (\bB^\top(X) \tbL_{X}\bB(X) ) \ge 0$ due to positive semidefiniteness of $\tbL$, which completes the proof.
\end{proof}

\subsection{DPPs via extended L-ensembles}

\begin{definition}[Extended L-ensemble] Let $\ELE{\bL}{\bV}$ be any NNP. An extended L-ensemble $\X$ based on $\ELE{\bL}{\bV}$ is a point process verifying:
	\begin{align}
	\label{eq:proba_mass}
	\forall X\subseteq\Omega,\qquad \Proba(\X=X) \propto (-1)^p \det
	\begin{pmatrix}
	\bL_{X} & \bV_{X,:} \\
	(\bV_{X,:})^\top & \matr{0}
	\end{pmatrix}.
	\end{align} 
\end{definition}
	\begin{remark}
	We stress that an extended L-ensemble reduces to an L-ensemble only in the case $p=0$. If $p\geq 1$, an extended L-ensemble is not an L-ensemble, since the probability mass function of $\X$ is not expressed as a principal minor of a larger matrix. Also, 
	the right-hand  side in eq. \eqref{eq:proba_mass} is non-negative by
	Lemma~\ref{lem:equivalence_ppdpp_ensemble}, and thus defines a valid
	probability distribution. The normalisation constant is tractable and given
	later (see section \ref{sec:ppdpp-marginal-kernel}). 
	On a more minor note, the factor $(-1)^p$ arises because of the peculiar properties of saddle-point matrices, see Lemma~\ref{lem:det-saddlepoint}. 
\end{remark}

One shows in fact that the class of extended L-ensembles is identical to the class of DPPs, as the two following theorems demonstrate. 
	\begin{theorem}
		\label{thm:L-ens_to_K}
		Let $\ELE{\bL}{\bV}$ be any NNP, and $\X$ be an extended L-ensemble based on $\ELE{\bL}{\bV}$. 
		Then, $\X$ is a DPP with marginal kernel
		\begin{equation}
		\label{eq:dpDPPmarginalkernel}
		\bK = \bQ \bQ^\top + \tbL(\bI+ \tbL)^{-1}.\end{equation}
	\end{theorem}
	\begin{proof}
		See Appendix~\ref{sec:thm:K_to_L-ens_and_back}.

\end{proof}
Thus, an extended L-ensembles is a DPP. Importantly, the converse is also true: any DPP (not only L-ensembles) is an extended L-ensemble.

	\begin{theorem}
		\label{thm:K_to_L-ens}
		Let $\bm{0}\preceq\bK\preceq \bI$ be any marginal kernel and $\X$ its associated DPP. Denote by $\bV\in\mathbb{R}^{n\times p}$ the matrix concatenating the $p\geq 0$ orthonormal eigenvectors of $\bK$ associated to eigenvalue $1$ and $\bL=\bK\left(\bI-\bK\right)^{\dagger}$ with $\dagger$ representing the Moore-Penrose pseudo-inverse. Then, $\X$ is an extended L-ensemble based on the NNP $\ELE{\bL}{\bV}$.
	\end{theorem}
	\begin{proof}
		See Appendix~\ref{sec:thm:K_to_L-ens_and_back}.
	\end{proof}
Recall that, as per definition~\ref{def:fsDPP}, a fixed-size DPP is simply a DPP conditioned on size. As a consequence of the equivalence between extended L-ensembles and DPPs, one thus obtains the following explicit expression of the probability mass function of any fixed-size DPP:
	\begin{corollary}
		Let $\bm{0}\preceq\bK\preceq \bI$ be any marginal kernel and $\X$ its associated fixed-size DPP of size $m$. Let $\ELE{\bL}{\bV}$ be the NNP as defined in theorem~\ref{thm:K_to_L-ens}. Then 
		\begin{align}
		\label{eq:proba_mass_fs}
		\forall X\subseteq\Omega,\qquad \Proba(\X=X) \propto (-1)^p \det
		\begin{pmatrix}
		\bL_{X} & \bV_{X,:} \\
		(\bV_{X,:})^\top & \matr{0}
		\end{pmatrix}\Ind(|X|=m).
		\end{align}
	\end{corollary}
	\begin{remark}
		Fixed-size DPPs of size $m$ with marginal kernel $\bK$ cannot be defined for
    $m$ smaller the multiplicity of $1$ in the spectrum of $\bK$. In other
    words, one cannot condition the DPP based on $\bK$ having fewer samples than
    its number of eigenvalues equal to one (by lemma~\ref{lemma:bernoulli}).
    Consequently, from the extended L-ensemble viewpoint, $m$ should always be
    larger than or equal to $p$.
	\end{remark}

\subsection{Partial projection DPPs}
The previous section made clear that 
	\begin{itemize}
		\item any DPP in the class of DPPs may be defined equivalently either via a marginal kernel  $\bm{0}\preceq\bK\preceq\bI$ from the marginal point of view, or via a NNP $\ELE{\bL}{\bV}$ from the point of view of the explicit probability mass function.
		\item the class of fixed-size DPPs, being in all generality defined as DPPs conditioned on size, are in fact best described with extended L-ensembles. Their probability mass function are given by Eq.~\eqref{eq:proba_mass_fs}. Apart from the special case where $m=p$ that implies a projection DPP~\footnote{If $m=p$, 
			$\bV_{X,:}$ is square in Eq.~\eqref{eq:proba_mass_fs} and by Lemma~\ref{lem:det-saddlepoint}, $\Proba(\X = X) \propto \det(\bV_{X,:})^2$, which is
			the probability mass function of a projection DPP (see lemma
			\ref{lem:max-rank-dpp}).}, fixed-size DPPs do not have marginal kernels.
	\end{itemize}
	In the following, for the purpose of this work, we differentiate DPPs (both varying-size and fixed-size) defined by NNPs $\ELE{\bL}{\bV}$ for which
	\begin{itemize}
		\item $p=0$:  Eq.~\ref{eq:proba_mass} (resp. Eq.~\ref{eq:proba_mass_fs}) boils down to Eq.~\ref{eq:def-dpp_via_L} (resp. Eq.~\ref{eq:def-dpp}): we recover the L-ensembles $\X\sim DPP(\bL)$ (resp. fixed-size L-ensembles $\X\sim|DPP|_m(\bL)$).
		\item $p\geq1$: in this case, the associated DPPs are not L-ensembles; and we will call them \emph{partial-projection DPPs} (pp-DPPs) for reasons that will become clear in section~\ref{sec:mixture_rep_ppDPP}. We will denote them $\X\sim \ppDPP \ELE{\bL}{\bV}$ and $\X\sim \mppDPP{m}\ELE{\bL}{\bV}$ for the varying-size and the fixed-size cases respectively.
\end{itemize}

\subsection{A generalisation of the Cauchy-Binet Formula}

The cornerstone of the mixture representation of L-ensembles, discussed in Section~\ref{sec:mixture-representation}, is in fact the Cauchy-Binet formula, recalled in Lemma~\ref{lem:cauchy-binet} (see for instance~\cite{Hough:DPPandIndep, kulesza2012determinantal}). In order to provide a similar spectral understanding of extended L-ensembles, we need the following generalisation of the Cauchy-Binet formula. 
\begin{theorem}
	\label{thm:equivalence-extended-spectral}
	Let $\ELE{\bL}{\bV}$ be a NNP, and $\bQ$, $\tbU$, $\tbLam$ and $q$ be as in Definition~\ref{def:ext_Lens}. Then for any subset $X \subseteq \{1, \ldots, n\}$ of size $|X| = m$, $p\leq m \leq p+q$, 
	it holds that
	\begin{equation}
	\label{eq:equivalence-spectral-extended-pDPP}
	(-1)^p\det
	\begin{pmatrix}
	\bL_{X} & \bV_{X,:} \\
	(\bV_{X,:})^\top & \matr{0} 
	\end{pmatrix} = \det(\bV^\top\bV)\sum_{Y,|Y|=m-p} \det \left( 
	\begin{bmatrix}
	\bQ_{X,:} & \tbU_{X,Y}
	\end{bmatrix}  \right)^2
	\prod_{i \in Y} \widetilde{\lambda}_i
	\end{equation}
\end{theorem}
\begin{proof}
	First of all, writing the $(\bQ, \bR)$ decomposition of $\bV$ as $\bV = \bQ\bR$ one has:
	\[
	\det
	\begin{pmatrix}
	\bL_{X} & \bV_{X,:} \\
	(\bV_{X,:})^\top & \matr{0} 
	\end{pmatrix} 
	= (\det(\matr{R}))^2
	\det
	\begin{pmatrix}
	\bL_{X} & \bQ_{X,:} \\
	(\bQ_{X,:})^\top & \matr{0} 
	\end{pmatrix}.     
	\]
	Noting that $\det(\bV^\top\bV) = (\det(\matr{R}))^2$, to prove Eq.~\eqref{eq:equivalence-spectral-extended-pDPP}   it is sufficient to show that:
	\begin{equation}
	\label{eq:equivalence-spectral-extended-pDPP_Q}
	(-1)^p\det
	\begin{pmatrix}
	\bL_{X} & \bQ_{X,:} \\
	(\bQ_{X,:})^\top & \matr{0} 
	\end{pmatrix} = \sum_{Y,|Y|=m-p} \det \left( 
	\begin{bmatrix}
	\bQ_{X,:} & \tbU_{X,Y}
	\end{bmatrix}  \right)^2
	\prod_{i \in Y} \widetilde{\lambda}_i.
	\end{equation}
	Now, the case $\rank \bQ_{X,:} < p$ is trivial as both sides in \eqref{eq:equivalence-spectral-extended-pDPP_Q} are zero. Next, we assume that $\bQ_{X,:}$ is full rank.
	Using first lemma~\ref{lem:equivalence_ppdpp_ensemble} and then lemma~\ref{lem:det-coef-polynomial}, one has:
	\[
	(-1)^p\det    \begin{pmatrix}  {\bL}_{\X} & {\bQ}_{X,:} \\  ({\bQ}_{X,:})^\top & {0}  \end{pmatrix}  =
	(-1)^p\det  \begin{pmatrix}     \widetilde{\bL}_{X} & {\bQ}_{X,:} \\      ({\bQ}_{X,:})^\top & {0}    \end{pmatrix}
	= [t^p ] \det(\widetilde{\bL}_X + t {\bQ}_{X,:}({\bQ}_{X,:})^\top).
	\]
	Using the fact that $\widetilde{\bL} = {\tbU}\tbLam{\tbU}^{\top}$, the right hand side may be re-written:
	\begin{align*}
	[t^p]\det(\widetilde{\bL}_X + t {\bQ}_{X,:}({\bQ}_{X,:})^\top) & =  [t^p]\det\left([{\bQ}_{X,:} {\tilde \bU}_{X,:} ] \begin{pmatrix}t \bI_p & \\ &  \widetilde{\matr{\Lambda}}  \end{pmatrix}[{\bQ}_{X,:} {\tilde \bU}_{X,:}]^\top\right)\\ 
	&  = 
	\sum\limits_{|Y| = m-p} (\det([{\bQ}_{X,:} {\tilde\bU}_{X,Y} ]))^2 \det(\widetilde{\matr{\Lambda}}_{Y}),
	\end{align*}
	where the last equality follows from the  Cauchy-Binet lemma.
\end{proof}

\subsection{Mixture representation}
\label{sec:mixture_rep_ppDPP}

In the mixture representation of L-ensembles (see Sec.~\ref{sec:mixture-representation}), one first samples a set of orthonormal vectors, forms a
projective kernel from these eigenvectors, and then samples a projection DPP
from that kernel. In that sense, a projection DPP is the trivial mixture in
which the same set of eigenvectors is always sampled. 
In this section, we will see that in partial projection DPPs, a subset of orthogonal vectors is included deterministically (coming from $\bV$), and the rest are subject to sampling, from the
part of $\bL$ orthogonal to $\bV$, hence the name \emph{partial projection}. 

In fact, examining Eq.~\eqref{eq:equivalence-spectral-extended-pDPP}, the kinship with the mixture representation of fixed-size L-ensembles should be clear 
upon comparison with equation \eqref{eq:mixture-representation-fixed}. The left-hand side of Eq.~\eqref{eq:equivalence-spectral-extended-pDPP} is the
probability mass function, and on the right-hand side we recognise a sum (over $Y$) of
probability mass functions for projection DPPs ($\det \left( 
\begin{bmatrix}    \bQ_{X,:} & \tbU_{X,Y}  \end{bmatrix}  \right)^2$) indexed by $Y$, weighted by a
product of eigenvalues ($\prod_{i \in Y} \widetilde{\lambda}_i$). This lets us represent the partial-projection DPP as a
probabilistic mixture. Contrary to fixed-size L-ensembles, some eigenvectors
appear with probability 1: the ones that originate from $\bV$ (represented by
$\bQ_{\X,:}$ in Eq.~\eqref{eq:equivalence-spectral-extended-pDPP}). The rest are
picked randomly according to the law given by the product $\Proba(\Y = Y) \propto \prod_{i \in Y} \tilde{\lambda}_i$.

Seen as a statement about probabilistic mixtures,  theorem~\ref{thm:equivalence-extended-spectral} provides a recipe
for sampling from $\X \sim \mppDPP{m} \ELE{\bL}{\bV} $. We summarize this recipe in the following statement: 

\begin{corollary}
  \label{cor:mixture-rep-ppdpp}
  Let $\ELE{\bL}{\bV}$ be a NNP, and $\bQ$, $\tbU$, $\tbLam$ and $q$ be as in Definition~\ref{def:ext_Lens}. 
  Let  $\X \sim\mppDPP{m} \ELE{\bL}{\bV}$ with $p \leq  m \leq  p+q$. Then, equivalently,
  $\X$ may be obtained from the following mixture process:
  \begin{enumerate}
  \item Sample $m-p$ indices $\Y \sim \mDPP{m-p}(\tbLam)$
  \item Form the projection matrix $\bM = \bQ\bQ^\top + \tbU_{:,\Y}
    (\tbU_{:,\Y})^\top$ (recall that $\bQ$ and $\tbU$ are orthogonal)
  \item Sample $\X | \Y \sim  \mDPP{m}(\bM)$
  \end{enumerate}
\end{corollary}
Note that at step 1 we only sample from the \emph{optional} part, since the
eigenvectors from $\bV$ need to be included anyways. The total number of
eigenvectors to include is $m$, so $m-p$ need to be sampled randomly. 

Using theorem \ref{thm:equivalence-extended-spectral}, as in the fixed-size
case, we arrive easily at the following mixture characterisation for the varying-size case: 

\begin{corollary}
  \label{cor:mixture-rep-ppDPP-varying}
  Let $\ELE{\bL}{\bV}$ be a NNP, and $\bQ$, $\tbU$ and $\tbLam$ be as in Definition~\ref{def:ext_Lens}. 
  Let $\X \sim \ppDPP \ELE{\bL}{\bV}$. Then, equivalently,
  $\X$ may be obtained from the following mixture process:
  \begin{enumerate}
  \item Sample indices $\Y \sim DPP(\tbLam)$
  \item Form the projection matrix $\bM = \bQ\bQ^\top + \tbU_{:,\Y} (\tbU_{:,\Y})^\top$
  \item Sample $\X | \Y \sim  \mDPP{p+|\Y|}(\bM)$
  \end{enumerate}
\end{corollary}

The only difference from the fixed-size case is in step 1. Again, we include all eigenvectors from $\bV$
(they make up the $\bQ\bQ^\top$ part of the projection matrix $\bM$),
then the remaining ones are sampled from $\Y \sim
DPP(\tbLam)$, which is equivalent to including the eigenvector $\tilde{\vect{u}}_i$
with probability $\frac{\tilde{\lambda}_i}{1+\tilde{\lambda}_i}$.

\subsection{Properties}
\label{sec:ppDPP-properties}

\subsubsection{Normalisation}
\label{sec:ppdpp-marginal-kernel}

Using theorem \ref{thm:equivalence-extended-spectral}, the normalisation
constant is tractable both in the fixed-size and varying-size cases, as shown by the following corollary (see also \cite[Lemma 3.11]{BarthelmeUsevich:KernelsFlatLimit} for an alternative formulation).

\begin{corollary}
\label{cor:normalisation-ppDPP}
  Let $\ELE{\bL}{\bV}$ be a NNP, and $\widetilde{\bL}$ and $q$ as in Definition~\ref{def:ext_Lens}. For $m$ such that $p\leq m \le n$, one has:
  \begin{equation}
    \label{eq:marginalisation-ppDPP}
    (-1)^p \sum_{|X|=m} \det
    \begin{pmatrix}
      \bL_{X} & \bV_{X,:} \\
      (\bV_{X,:})^\top & \matr{0} 
    \end{pmatrix} = e_{m-p}(\tbL)\det(\bV^\top\bV)
  \end{equation}
  and
  \begin{equation}
  (-1)^p \sum_{X} \det
    \begin{pmatrix}
      \bL_{X} & \bV_{X,:} \\
      (\bV_{X,:})^\top & \matr{0} 
    \end{pmatrix} = \det(\bI + \tbL)\det(\bV^\top\bV)
    \label{eq:marginalisation-ppDPP-varying}
\end{equation}

\end{corollary}
\begin{proof}
If $m > p+q$, then the right-hand side is zero, as well as the left-hand side (by lemma~\ref{lem:det-saddlepoint}).  
 In the case $m \le p+q$,  from  theorem \ref{thm:equivalence-extended-spectral} we have:
  \begin{align*}
    (-1)^p \sum_{|X|=m} \det
    \begin{pmatrix}
      \bL_{X} & \bV_{X,:} \\
      (\bV_{X,:})^\top & \matr{0} 
    \end{pmatrix} &= 
                     \det(\bV^\top\bV) \sum_{|X|=m} \sum_{Y,|Y|=m-p} \det \left( 
    \begin{bmatrix}
      \bQ_{X,:} & \tbU_{X,Y}
    \end{bmatrix}  \right)^2
                   \prod_{i \in Y} \tilde{\lambda}_i \\
              &=  \det(\bV^\top\bV) \sum_{Y,|Y|=m-p} \quad \prod_{i \in Y} \tilde{\lambda}_i \\
              &= e_{m-p}(\tilde{\bL})\det(\bV^\top\bV),
  \end{align*}
  where the sum over $X$ is just the normalisation constant of a projection DPP
  (see remark \ref{rem:normalisation-constant-proj}). The proof for varying size
  is similar, using:
  $\sum_{Y} \prod_{i \in Y} \tilde{\lambda}_i = \prod_{i=1}^q (1 + \tilde{\lambda}_i).$
\end{proof}

Using these results, we easily obtain the distribution of the size of $|\X|$
for $\X\sim \ppDPP \ELE{\bL}{\bV}$. One may check that equivalent results
are obtained either using the mixture representation (see corollary
\ref{cor:mixture-rep-ppDPP-varying}) or the associated marginal kernel (via Eq.~\ref{eq:dpDPPmarginalkernel} and lemma~\ref{lemma:bernoulli}). 
\begin{corollary}
  Let $\X \sim \ppDPP \ELE{\bL}{\bV}$. Then
  \begin{equation}
    \label{eq:prob_size_ppdpp}
    \Proba(|\X|=m) =
    \begin{cases}
      0, & \mbox{\ if\ } m < p, \\
      \frac{e_{m-p}(\tilde{\bL})}{\det(\tbL + \bI)}, &\mbox{\ otherwise}.
    \end{cases}
  \end{equation}
\end{corollary}

\subsubsection{Complements of DPPs}
\label{sec:complements_dpps}

A known (see e.g., \cite{kulesza2012determinantal}) result about DPPs is
that the complement of a DPP in $\Omega$ is also a DPP, i.e., if $\X$ is a DPP,
$\X^c = \Omega \setminus \X$ is also a DPP. We shall give a short proof and some
extensions. 

\begin{theorem}
  \label{thm:complement-DPP}
  Let $\X$ be a DPP with marginal kernel $\bK$. Then the complement of
  $\X$, noted $\X^c$, is also a DPP, and its marginal kernel is $\bI - \bK$. 
\end{theorem}
\begin{proof}
  We first prove this for projection DPPs. Let $\cA \sim \mDPP{m}(\bU\bU^t)$ for
  orthogonal $\bU$ of rank $m$. Then
  \[ \Proba(\cA^c=A) = \Proba(\cA = A^c) \propto \det( \bU_{A^c,:})^2.\] Note that for the
  probability to be non null we need $A$ to be of size $ n-m $.
  
  Let $\bV \in \R^{n \times (n-m)}$ so that $ \bI = \bU\bU^t + \bV\bV^t$. $ \bM
  = 
  \begin{pmatrix}
    \bU & \bV
  \end{pmatrix}$ is an orthogonal basis for $\R^n$ which we may partition as
  $\begin{pmatrix}
    \bU_{A^c,:} & \bV_{A^c,:} \\
    \bU_{A,:} & \bV_{A,:}
  \end{pmatrix}$
  By lemma \ref{lem:block-det}
  \[ \det \bM  =  \det \bU_{A^c,:} \det \left( \bV_{A,:} -
      \bV_{A^c,:} (\bU_{A^c,:})^{-1}  \bV_{A^c,:} \right). \]
  This gives
  \[ \Proba(\cA = A) \propto \det \left( (\bV_{A,:} - 
      \bV_{A^c,:} (\bU_{A^c,:})^{-1}  \bV_{A^c,:})^{-1} \right).\]
  By the inversion formula for block matrices this is equal to the lower-right
  block in $\bM^{-1} =\bM^t$, and so:
  \[ \Proba(\cA = A) \propto \det \left( (\bV_{A,:} \right)^2\] where we recognise a
  projection DPP ($\cA \sim \mDPP{n-m}(\bV\bV^t)$, as claimed). 
  We now use the mixture property to show the general case.
  In the general case,
  \[ \Proba(\X = X) = \sum_{\Y} \Proba(\Y) \det (\bU_{X,\Y})^2\]
  so that:
  \[ \Proba(\X^c = A) = \sum_{\Y} \Proba(\Y) \det (\bU_{A^c,\Y})^2
    =  \sum_{\Y^c} \Proba(\Y^c) \det (\bV_{A,\Y^c})^2 \]
  Since each eigenvector is picked independently in $\Proba(\Y)$ with probability
  $\pi_i$, picking each eigenvector independently with probability $1-\pi_i$
  produces a draw from $\Proba(\Y^c)$. $\Proba(\X^c = A)$ is therefore a DPP, and its kernel is $\bI-\bK$.
\end{proof}

Applying the theorem to L-ensembles we obtain: 
\begin{corollary}
  Let $\X \sim DPP(\bL)$, with $\bL$ a rank $p$ matrix and $p \leq n$. Then
  $\X^c \sim DPP(\bL^\dag, \bV)$ with $\bV$ a basis for $\orth \bL$. In
  particular, if $p = n$
  ($\bL$ is full rank), we have $\X^c \sim DPP(\bL^{-1})$.
\end{corollary}

For extended L-ensembles this generalises to:
\begin{corollary}
  Let $\X \sim DPP \ELE{\bL}{\bV} $, and let $\bZ$ be a basis for $\orth \tbL
  \setminus \mspan \bV $. Then 
  $\X^c \sim DPP(\tbL^\dag, \bZ)$. 
\end{corollary}

The following fixed-size variant is new: it states that the complement
of a fixed-size DPP is also a fixed-size DPP
\begin{proposition}
  Let $\X \sim \mDPP{m} \ELE{\bL}{\bV} $, and let $\bZ$ be a basis for $\orth \tbL  \setminus \mspan \bV $. Then 
  $\X^c \sim \mDPP{n-m}(\tbL^\dag, \bZ)$.
\end{proposition}
\begin{proof}
  Proof sketch: repeat the proof of th. \ref{thm:complement-DPP} up to the mixture
  representation, where we note that since $p(\Y = Y) \propto \prod_{i \in Y} \lambda_i$,
  $p(\Y^c) \propto \prod_{j \in Y^C} \frac{1}{\lambda_j}$ which is again a
  diagonal fixed-size DPP. 
\end{proof}

\subsubsection{Partial Invariance}
\label{sec:ppdpp-invariance}

We parametrise partial-projection DPPs using a pair of matrices (the NNP $\ELE{\bL}{\bV}$), but this is an over-parameterisation since all
that matters is the linear space spanned by $\bV$, as the following makes clear: 
\begin{remark}
	\label{rem:first_inv}
  Consider a NNP $\ELE{\bL}{\bV}$. Let $\bV'=\bV\bR$ with $\bR \in \R^{p
    \times p}$ invertible. We have $\mspan \bV' = \mspan \bV$. Then $\X
  \sim \ppDPP\ELE{\bL}{\bV}$ and  $\X
  \sim \ppDPP\ELE{\bL}{\bV'}$  define the same point process.
  This also holds for $\X \sim \mppDPP{m}\ELE{\bL}{\bV}$ for any $m \geq p$.
\end{remark}
\begin{proof}
  This is clear from theorem \ref{thm:equivalence-extended-spectral} or the
  mixture representation of the partial-projection DPP. Nothing on the
  right-hand side of equation \eqref{eq:equivalence-spectral-extended-pDPP}
  is affected by replacing $\bV$ with a matrix with identical span. In
  particular, the distribution is invariant to rescaling of $\bV$ by any
  non-zero scalar. 
\end{proof}

Notice that this generalises a property of projection DPPs given in the
introduction (section \ref{sec:diag-and-proj}), which is that $\X \sim
\mDPP{m}(\bL)$ and $\X \sim \mDPP{m}(\bL')$ are the same if $\bL$ and $\bL'$
have the same column span and rank $m$. 

Another source of invariance in partial projection DPPs lies in $\bL$: we can
modify $\bL$ along the subspace spanned by $\bV$ without changing the distribution.
\begin{remark}
	\label{rem:second_inv}
  Consider a NNP $\ELE{\bL}{\bV}$. Let $\bL' = \bL + \bV\bm{X}^\top + \bm{Y}\bV^\top$ for any two matrices $\bm{X}, \bm{Y}
  \in \R^{n \times p}$. Then $\X \sim \ppDPP\ELE{\bL}{\bV}$ and $\X \sim
  \ppDPP\ELE{\bL'}{\bV}$ define the same random variable.
\end{remark}
\begin{proof}
Indeed, by Definition~\ref{def:ext_Lens}, we have $\widetilde{\bL' } = \tbL$.
Therefore, by lemma~\ref{lem:equivalence_ppdpp_ensemble}, the DPPs defined by  ${\bL' }$ and $\bL$ coincide. 
\end{proof}

\subsection{Examples}
\label{sec:ppdpp-examples}

We give here a few examples of partial projection DPPs and their NNPs.

\subsubsection{Partial projection DPPs as conditional distributions}

A simple example of a partial projection DPP arises when the columns of the matrix $\bV$ come from a canonical basis (i.e., each column of $\bV$ is a standard unit vector).
In this case, partial projection DPPs can be interpreted as a particular conditional of a DPP.
For simplicity, assume that $\matr{V}  =  \begin{bmatrix} \matr{I}_p & 0 \end{bmatrix}^{\top}$,
so that the projected $\matr{L}$ matrix  becomes
\[
\tbL =  \begin{bmatrix} 0& 0  \\ 0 & \bL_{\{\vect{x}_{p+1}, \ldots,\vect{x}_n\}}\end{bmatrix}.
\]

In this case, the mixture representation for pp-DPPs (resp. fixed-sized pp-DPPs) implies that:
\begin{itemize}
\item all the points  $\vect{x}_1, \ldots,\vect{x}_p$ are always sampled; 
\item  the remaining points are sampled according to the L-ensemble (resp. fixed-size L-ensemble) based on   $\matr{L}_{\{\vect{x}_{p+1}, \ldots,\vect{x}_n\}}$.
\end{itemize}
For example, in the  varying-size case $\X \sim \ppDPP \ELE{\bL}{\bV}$, the probability of sampling the remaining points is
\begin{equation}\label{eq:pp-dpp-conditional}
\Proba(\X \cap \{\vect{x}_{p+1}, \ldots,\vect{x}_n\} = X') \propto \det  \bL_{X'},
\end{equation}
which is linked  to a certain conditional distribution of the ordinary L-ensemble based on $\bL$
(see \cite[\S 2.4.3]{kulesza2012determinantal} for more details).

\subsubsection{Partial projection DPPs and conditional positive definite functions}
\label{sec:ppdpp-examples-cpdef}

An important generalisation of positive definite kernels is the notion of
conditional positive definite kernels (see for example
\cite{micchelli1986interpolation},\cite{wendland2004scattered}), especially in
interpolation problems with polynomial regularisation. Conditional positive
definite kernels generate conditionally positive definite matrices when
evaluated at a finite set of locations, just like positive definite kernels
generate positive definite matrices. We will show here that extended L-ensembles
let us construct DPPs based on conditional positive definite functions.

\begin{definition}
A function $f:\R^d \longrightarrow \R$ is conditionally positive definite of order $\ell$ if and only if, for any $n\in \mathbb{N}$, any $X= (\vect{x}_1,\ldots,\vect{x}_n) \in (\R^d)^n$, any $\vect{\alpha}\in \R^n$ satisfying $\sum_i \alpha_i \vect{x}_i^{\vect{\beta}}=0$ for all multi-indices $\vect{\beta}$  s.t. $|\vect{\beta}|<\ell$, the quadratic form
\[
 \sum_{i,j} \alpha_i \alpha_j f(\vect{x}_i-\vect{x}_j)
 \]
 is non-negative.
\end{definition}

Suppose now that we introduce ``Gram" matrices $\vect{L}_{X} = [f(\vect{x}_i-\vect{x}_j)]_{i,j}$, and the multivariate Vandermonde matrix   $\vect{V}_{\leq \ell-1}( X)$. Then, an equivalent definition is
\begin{definition}
A function $f:\R^d \longrightarrow \R$ is conditionally positive definite of order $\ell$ if and only if, for any $n\in \mathbb{N}$, any $X= (\vect{x}_1,\ldots,\vect{x}_n) \in (\R^d)^n$, the matrix $\vect{L}_{X}$ is conditionally positive definite with respect to $\vect{V}_{\leq \ell-1}(X)$.
\end{definition}
This extends the possible functions used to measure diversity in DPP sampling. For example, it can be shown that 
$f( \vect{x} ) = \phi (  \| \vect{ x} \|^2_2)$ 
where $\phi : \R^+\rightarrow \R$ is the so-called multiquadrics 
$(-1)^{\lceil \beta \rceil} (c^2 + r^2 )^\beta ; c,\beta >0, \beta \not\in
\mathbb{N}$ is conditional positive definite of order $\lceil \beta \rceil$. To
be explicit, we may for instance define  a valid extended L-ensemble based on a NNP $\ELE{\bL}{\ones}$ with $L_{ij}= - \sqrt{c^2 +
  \norm{\vect{x}_i-\vect{x}_j}^2}  $. Likewise, the so-called "thin-plate spline" $\phi(r)= (-1)^{k+1} r^{2k} \log(r)$ makes $f( \vect{x} ) = \phi (  \| \vect{ x} \|^2_2)$ a conditional positive definite function of order $k+1$  on $\R^d$.   

A last example of great interest for this paper is the case of $\phi(r)=(-1)^{\lceil \beta/2 \rceil}  r^\beta ; \beta >0, \beta \not\in 2\mathbb{N}$ which makes $f(\vect{x}) = (-1)^{\lceil \beta/2 \rceil}  \|\vect{x}\|^{2\beta} $ a conditional positive function of order $\lceil \beta/2 \rceil$. Indeed, we will encounter in sections \ref{sec:univariate-results} to \ref{sec:varying-size}  extended $L$-ensembles of the form $((-1)^r\bD^{(2r-1)} ; \vect{V}_{\leq r-1})$ where $\bD^{(2r-1)} = [\norm{\vect{x}_i-\vect{x}_j}]^{2r-1}_{i,j}$, for $r$ a positive integer, corresponding to $\beta=r-1/2$.

We stated above that a link exists to interpolation. To illustrate the link, suppose we want to interpolate points $((\vect{x}_1,y_1), \ldots,(\vect{x}_n,y_n)\in (\R^d \times \R)^n$ using the function $s(x)=\sum_i \alpha_i f(\vect{x}_i-\vect{x}_j) + \sum_{k=1}^{\PP_{\ell-1,d}} \beta_k p_k(\vect{x})$
where $f$ is a conditionally positive function of order $\ell$, and $p_k$, $k=1,\ldots, \PP_{\ell-1,d}$ is a basis for the set of polynomials of degree less  or equal than $\ell-1$.  The solution of this interpolation problem is then equivalent to the solution of the linear system
\[
  \begin{pmatrix}
      \bL_{X} & \bV_{\leq \ell-1}\\
     (\bV_{\leq \ell-1})^\top & \matr{0} 
    \end{pmatrix}
      \begin{pmatrix}
      \vect{\alpha} \\
     \vect{\beta} 
    \end{pmatrix}
    =
      \begin{pmatrix}
     \vect{y}\\
    \vect{0} 
    \end{pmatrix}
  \]
where we recover the matrix defining the $L$-ensemble in partial projection
DPPs. A DPP based on the conditional positive definite kernel $f$ will sample a
good design for interpolation, since the interpolation points are
selected such that the interpolation matrix is well-conditioned. This link
between DPP sampling and interpolation theory deserves to be further studied, but is beyond the scope of the paper.

\subsubsection{Roots of trees in uniform spanning random forests are partial projection DPPs}
\label{sec:roots-forests}

\begin{figure}[pt]
\begin{center}
\begin{picture}(0,0)\includegraphics{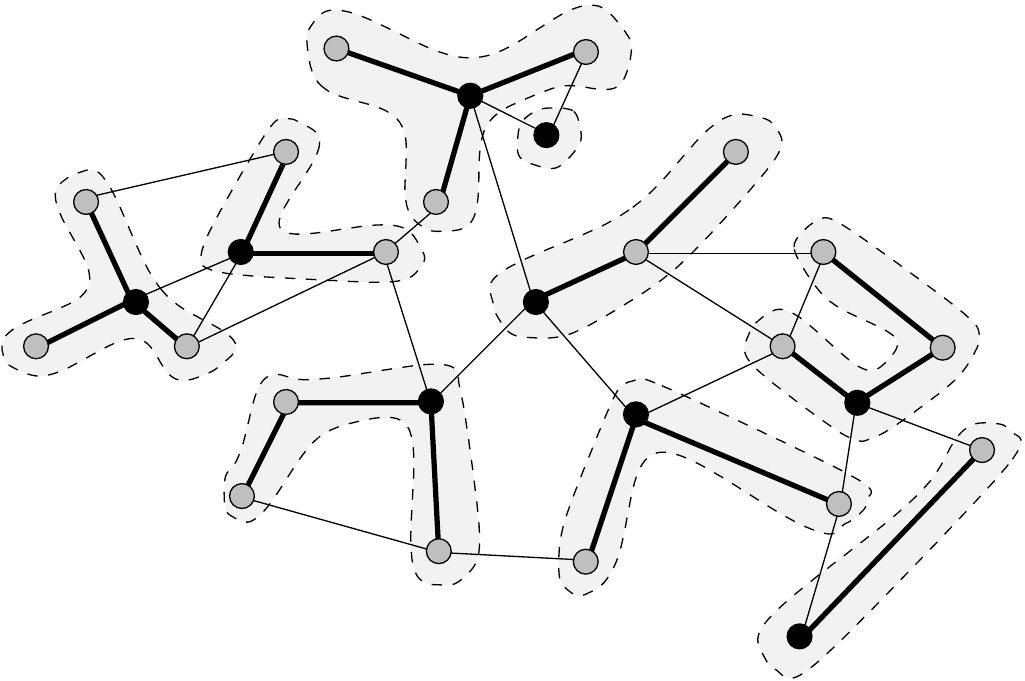}\end{picture}\setlength{\unitlength}{3315sp}\begingroup\makeatletter\ifx\SetFigFont\undefined \gdef\SetFigFont#1#2#3#4#5{\reset@font\fontsize{#1}{#2pt}\fontfamily{#3}\fontseries{#4}\fontshape{#5}\selectfont}\fi\endgroup \begin{picture}(5850,3872)(2937,-4753)
\end{picture} \end{center}
\caption{Roots of uniform random forests over a graph are distributed according to a partial projection DPP. Vertices or nodes are depicted in gray; edges as thin lines. A random forest is depicted : its trees are surrounded by light gray zones; edges of the trees are thicks black lines; roots are the black nodes. The forest is spanning the graph as each nodes of the graph appears once in a tree of the forest.}
\label{fig:URForest}
\end{figure}

It is known ({\it e.g. } \cite{avena2013some}) that the roots of the trees in a uniform random spanning forest over a graph with $n$ nodes and Laplacian $\vect{{\mathcal{L}}}$ are distributed according to a DPP with  marginal kernel  $\vect{K} = q (q \vect{I} +\vect{{\mathcal{L}}})^{-1} $ for some real parameter $q>0$. 
Figure \ref{fig:URForest} illustrates what a spanning forest over a graph is. 
Let us denote as $\lambda_1\geq \ldots \geq \lambda_n = 0$ the eigenvalues of the Laplacian, and $\{ \vect{u}_i \}_i$ the associated set of orthonormal  eigenvectors. It is well known that $\lambda_n=0$ for any graph:  $\vect{K}$ thus has at least one eigenvalue equal to $1$ and, as such, the associated DPP is not an L-ensemble. It can however be described by an extended L-ensemble:
\begin{proposition}
The set of roots in a uniform random spanning forest over a connected graph with Laplacian $\vect{{\mathcal{L}}}$ is distributed according to a partial projection DPP with NNP $( q \vect{{\mathcal{L}}}^{\dagger} ; \vect{1} )$, where $\dagger$ stands for the Moore-Penrose inverse.
\end{proposition}

\begin{proof}
Applying theorem~\ref{thm:K_to_L-ens}, a DPP with marginal kernel $\bK$
  can be described by an extended L-ensemble based on the NNP $(\bL, \bV)$ with $\bV$ and $\bL$ verifying:
	\begin{itemize}
		\item the matrix $\bV$ concatenates all eigenvectors of $\bK$ associated to eigenvalue 1: in a connected graph, there is only one such eigenvalue and it is associated to eigenvector  $\vect{u}_n=n^{-1/2} \vect{1}$
		\item the matrix $\bL$ is equal to $\bK(\bI-\bK)^\dagger$, which is equal to $q\vect{{\mathcal{L}}}^{\dagger}$
	\end{itemize}
\end{proof}

\begin{remark}
  This example also provides a nice illustration for the properties of
  complements of DPPs (section \ref{sec:complements_dpps}). Since $\mathcal{L}$ is a positive-definite matrix, we may define  $\mathcal{C} \sim
  DPP(\mathcal{L})$. The complement of $\mathcal{C}$ is a DPP $\mathcal{C}^c
  \sim DPP \ELE{\mathcal{L}^\dag}{\mathbf{1}}$, which  from the result above 
  corresponds to the roots process. 
  $\mathcal{C}$ therefore samples every node except the roots of a random forest
  on the graph. 
\end{remark}

\section{Partial projection DPPs as limits}
\label{sec:ppDPP-as-limits}
The main goal of this section is to serve as a warm-up for the study of flat
limits, and illustrate on a simple case the mathematical
tools used later in the paper, as well as some of the peculiarities of limits of L-ensembles (such as dependence on scaling).

As stated above, pp-DPPs arise as limits of certain L-ensembles, and in this
section we exhibit one such limit: the L-ensemble based on the linear perturbation of a (low-rank)
positive semi-definite matrix; i.e., we consider L-ensembles based on matrices of the form:
\begin{equation}
  \label{eq:unbalanced-L-ensemble}
  \bL_\varepsilon \triangleq \varepsilon \bA + \bV \bV^\top
\end{equation}
where $\bA$ has full rank\footnote{The case where $\bA$ is not full rank can also be studied, but it is
	more burdensome and not much more informative} $n$ and $\bV$ has full column rank $p < n$. 

Thus $\bL_\varepsilon$ defined in  \eqref{eq:unbalanced-L-ensemble} is a regular matrix pencil.  One should think about this
scenario as constructing a kernel as a sum of (a) a few important features
contained in $\bV\bV^{\top}$ and (b) a generic kernel in $\bA$.

\subsection{Limit of fixed-size L-ensembles based on \texorpdfstring{$\varepsilon \bA + \bV \bV^\top$}{}}
We begin with the more straightforward fixed-size case.
We seek the
limiting process $\Xs$ of $\X_\varepsilon \sim \mDPP{m}(\bL_\varepsilon)$ as $\flatlim$.
The following theorem establishes the limiting distribution  using asymptotic expansions of the  determinants. 
\subsubsection{Limiting process}
\label{sec:linear-pert-fixed-size}
\begin{theorem}
  \label{thm:partial-proj-limit}
  Let $\Xe \sim \mDPP{m}(\bL_\varepsilon)$, with $\bL_\varepsilon$ as in Eq.~\eqref{eq:unbalanced-L-ensemble}. Then the limiting process is: 
\[
\Xe \rightarrow \Xs \sim 
\begin{cases}
\mDPP{m}(\bV\bV^{\top}),  & m \le p \\
\mppDPP{m} \ELE{\bA}{\bV}, & m > p. \\
\end{cases}  
\]
 \end{theorem}
\begin{proof}

First, we consider the case $m \le p$.
Note that the unnormalized probability mass function for the $L$-ensemble based on $\bL_\varepsilon$ is 
\[
f_{\varepsilon} (X)  = \det((\varepsilon \bA +  \bV\bV^\top)_{X}) =  \det( \varepsilon  \bA _{X} +  \bV_{X,:}( \bV_{X,:})^\top)=  \det( \bV_{X,:}(\bV_{X,:})^\top)  + \O(\varepsilon).
\]
Since $\rank \bV = p \ge m$, there exists a subset of rows $X_0$ such that 
\begin{equation}\label{eq:VVT_nonzero_minor}
\det( \bV_{X_0,:}(\bV_{X_0,:})^\top) \neq 0.
\end{equation}
Therefore, by lemma~\ref{lem:TV-convergence-diverging}, we get that $\Xe  \rightarrow \mDPP{m}(\bV\bV^{\top})$.

The case $m > p$ is more delicate, as eq. \eqref{eq:VVT_nonzero_minor} no longer holds true, and we need to determine the  order of $\varepsilon$ in the expansion of $f_{\varepsilon} (X)$.
For this, we can invoke lemma~\ref{lem:det-coef-polynomial}  and remark~\ref{rem:det-coef-degree} to get

\begin{align*}
f_{\varepsilon} (X)  &   =\det( \varepsilon  \bA _{X} +  \bV_{X,:}( \bV_{X,:})^\top) = \varepsilon^{m} \det(  \bA _{X} + \varepsilon^{-1}   \bV_{X,:}( \bV_{X,:})^\top).\\
 &= 
 \varepsilon^{m} \left(\varepsilon^{-p}  (-1)^p\det
    \begin{pmatrix}
      \bA_{X} & \bV_{X,:} \\
     (\bV_{X,:})^\top & \matr{0} 
    \end{pmatrix} + \varepsilon^{-(p-1)} \ldots   \right)   \\ 
& = \varepsilon^{m-p} \left(  (-1)^p\det
    \begin{pmatrix}
      \bA_{X} & \bV_{X,:} \\
     (\bV_{X,:})^\top & \matr{0} 
    \end{pmatrix} +\O(\varepsilon) \right).
\end{align*}
By applying lemma \ref{lem:TV-convergence}, we get
 \begin{equation}
    \label{eq:extended-L-ensemble_bis}
    \Proba(\Xs = X) \propto  (-1)^p\det
    \begin{pmatrix}
      \bA_{X} & \bV_{X,:} \\
      (\bV_{X,:})^\top & \matr{0} 
    \end{pmatrix},
  \end{equation}
and hence $\Xe  \rightarrow \mppDPP{m} \ELE{\bA}{\bV}$.  
\end{proof}

\begin{remark}
  Note that if $m=p$  the limiting process is a projection DPP by lemma \ref{lem:max-rank-dpp}
\end{remark}

\subsection{A spectral view}
\label{sec:partial-proj-spectral}

As we show in this section, the limiting distribution in theorem~\ref{thm:partial-proj-limit} can be obtained using a completely different, and, in our opinion, more interpretable approach.

Recall the mixture representation of L-ensembles and fixed-size L-ensembles described in section
\ref{sec:mixture-representation}. Given a positive semi-definite matrix $\bL = \bL_\varepsilon$, one
first samples some eigenvectors of $\bL$, then builds a projection matrix $\bU_{:,\Y}(\bU_{:,\Y})^\top$ from
these eigenvectors, then samples a projection DPP from $\bU_{:,\Y}(\bU_{:,\Y})^\top$. We shall now study the asymptotic distribution of $\Xe$ from the
mixture point of view, using the spectral results of section
\ref{sec:basic-matrix-perturbation}.

Lemma \ref{lem:pert-eigenvalues} implies that the spectrum of $\bL_\varepsilon$ 
contains $p$ eigenvalues  $\lambda_{1}(\varepsilon), \ldots, \lambda_{p}(\varepsilon)$ of order $\O(1)$, and $n-p$ eigenvalues $\lambda_{p+1}(\varepsilon), \ldots, \lambda_{n}(\varepsilon)$ of order
$\O(\varepsilon)$.
In other words, their   expansion reads
\begin{equation}\label{eq:ev_2nd_order}
\lambda_i(\varepsilon) = \lambda_{i,0} +
\varepsilon\lambda_{i,1}+\O(\varepsilon^2),
\end{equation}
 where $\lambda_{i,0} \neq 0$ for $i \le p$,   and $\lambda_{i,0}$ is null otherwise.

In the case of fixed-size L-ensembles, in the mixture representation, the eigenvectors
are sampled according to the following law ($\Y_\varepsilon$ indexes the sampled eigenvectors): 
\begin{equation}\label{eq:diag-mdpp-asymptotics}
\Proba(\Y_\varepsilon = Y) \propto  \prod_{i\in Y} \lambda_i(\varepsilon) \cdot \Ind( |Y| = m),
\end{equation}
where $\lambda_i(\varepsilon)$ are as in \eqref{eq:ev_2nd_order}. 
Intuitively: if $m \leq p$, then all the sets $Y \subseteq \{ 1, \ldots, p\} $ have probability mass $\O(1)$. All other sets $Y$
have probability mass $\O(\varepsilon)$ or smaller. As $\flatlim$, the limiting
process must then only select $Y \subseteq \{ 1, \ldots, p\}$. If $m > p$, then the process
is forced to select some of the small eigenvalues, but then as few as possible:
the lowest possible order in $\varepsilon$ of the probability mass function is
$\O(\varepsilon^{m-p})$, which is obtained by having $\{ 1, \ldots, p \} \subset
Y$, and selecting the $m-p$ remaining ones at random. 
This discussion can be summarized as follows. 
\begin{proposition}\label{prop:kDPP_pencil_limit_spectral}
  If $m \leq p$, the limiting distribution of $\Y_\varepsilon$ is:
  \[
    \Proba(\Y_\star = Y) \propto  \prod_{i\in Y} \lambda_{i,0}
 \cdot \Ind\big(|Y| = m \text{ and } Y \subseteq \{1, \ldots, p\}\big)  
 \]
  As a special case, if $m=p$ then $\Y_\star=\{1, \ldots, p\}$ with probability 1. 
  
  If $m>p$ the limiting distribution of $\Y_\varepsilon$ is 
\[
\Proba(\Y_\star = Y) \propto   \prod_{i\in Y \cap \{p+1, \ldots, n\}} \lambda_{i,1}\cdot
\Ind\left(|Y| = m \text{ and }  \{1, \ldots, p\} \subset Y \right).
\]
\end{proposition}
\begin{proof}
  Let $\Z = \Y \cap \{1, \ldots, p \}$. We first characterise the limiting
  distribution of $|\Z|$, then the conditional $ \Y \big\vert |\Z|$. If $m \leq p$, we see that  
  $\Proba(|\Z| = m) = 1 +
  \O(\varepsilon)$, hence the conditional distribution is
\[
\Proba(\Y = Y \big\vert |\Z|=m) \propto\prod_{i\in Y \cap \{1,\ldots, m\} }\lambda_i(\varepsilon).
\]  
If $m>p$, we see that $\Proba(|\Z| = p) = 1 +  \O(\varepsilon)$,  the conditional distribution is 
\[
\Proba(\Y = Y \big\vert |\Z|=p) \propto  \prod_{i\in Y \cap \{p+1, \ldots, n\} }  \lambda_i(\varepsilon).
\]
In both cases, we may invoke lemma \ref{lem:TV-convergence} to complete the proof. 
\end{proof}

We now know how $\Y_\varepsilon$ is sampled in the limit. In parallel,
we have conditional distributions $\Xe | \Y_\varepsilon$ that are projection-DPPs. By lemma,
\ref{lem:linear-pert-eigenvectors} the
eigenvectors of $\bL_\varepsilon$ converge to  $\begin{bmatrix} \bQ& \tbU\end{bmatrix}$, where  $\bQ$ and  $\tbU$ are as in  Definition~\ref{def:ext_Lens} for the extended L-ensemble $\ELE{\bA}{\bV}$. 
This establishes the following: 
\begin{proposition}
  Let $\Xe \sim \mDPP{m}(\bL_\varepsilon)$, $\bL_\varepsilon$ as in \eqref{eq:unbalanced-L-ensemble}.
  Note
  $\bL_\varepsilon = \bU(\varepsilon)\bm{\Lambda}(\varepsilon)\bU(\varepsilon)^\top$ the
  eigendecomposition of $\bL_\varepsilon$. Then the mixture representation of $\Xe$, i.e.
  \begin{enumerate}
  \item $\Y_\varepsilon \sim \mDPP{m}(\bm{\Lambda}(\varepsilon))$,
  \item $\X_\varepsilon | \Y_\varepsilon \sim \mDPP{m}(\bM(\Y_{\varepsilon},\varepsilon))$ with
    $\bM(Y,\varepsilon)= \bU_{:,Y}(\varepsilon)(\bU_{:,Y}(\varepsilon))^\top$,
  \end{enumerate}
  has the limit:
  \begin{enumerate}
  \item $\Y'_\star \sim \mDPP{m-p}(\widetilde{\bm{\Lambda}})$,
  \item $\X_\star | \Y'_\star \sim \mDPP{m}(\bM(\Y'_{\star},\varepsilon))$ with
    $\bM(Y')= \bQ\bQ^\top + \tbU_{:,Y'}(\tbU_{:,Y'})^\top$.
  \end{enumerate}
which is equivalent to corollary \ref{cor:mixture-rep-ppdpp}.    
\end{proposition}

Put more plainly, if $m \geq p$ the limiting fixed-size L-ensembles is a partial projection DPP: the top $p$
eigenvectors are included with probability 1, and the  $m-p$ others are picked
according to the law of a diagonal L-ensemble with diagonal entries equal to the
(non-zero) eigenvalues of $(\bI - \bQ\bQ^{\top})\bA (\bI - \bQ\bQ^\top)$, by lemma
\ref{lem:linear-pert-eigenvectors}.

\subsection{Limits of variable-size L-ensembles based on \texorpdfstring{$\bA  + \varepsilon^{-1}\bV\bV^t$}{}}
\label{sec:variable-size-pdpps-as-limits}

The variable-size version of the results requires a bit more care. In fixed-size
L-ensembles, the law of $\X$ is invariant to a rescaling of the positive semi-definite matrix it is based on: $\X \sim
\mDPP{m}(\bL)$ is equivalent to $\mDPP{m}(\alpha\bL)$ for any $\alpha > 0$. For
regular (variable-size) DPPs this is not true. That feature both enriches and
complicates a little the asymptotic analysis. 

\subsubsection{A trivial limit}
\label{sec:trivial-limit}

Let us start with a straightforward limit, namely $\Xe \sim
DPP(\bL_\varepsilon)$ based on the matrix pencil defined in  \eqref{eq:unbalanced-L-ensemble}.  There are several equivalent ways of obtaining the limiting process, but let us
use the mixture representation, to contrast with the fixed-size case. In the mixture representation, the
only difference between L-ensembles and fixed-size L-ensembles is in how one samples the
eigenvectors. In variable-size L-ensembles,  by lemma~\ref{lem:mixture_dpp},  these are sampled from a Bernoulli process with inclusion probability 
\[ \pi_i(\varepsilon)= \frac{\lambda_i(\varepsilon)}{1+\lambda_i(\varepsilon)} \]
Inserting   expansions of $\lambda_i(\varepsilon)$ from \eqref{eq:ev_2nd_order}, we can directly compute the limit of the inclusion probabilities:
\[ \pi_i(\varepsilon) = \frac{\lambda_{i,0}}{1+ \lambda_{i,0}} + \O(\varepsilon).
\]
Thus, the probability to sample each of the eigenvectors goes to $\frac{  \lambda_{i,0}}{1+ \lambda_{i,0}}$,   which is equal to $0$ for the last  $n-p$ eigenvectors. 
Since these events are independent, this implies that in $\flatlim$ (with
probability $1$) we only sample from the top $p$ eigenvectors of $\bL(\varepsilon)$.
By lemma \ref{lem:linear-pert-eigenvectors}, these top $p$ eigenvectors
themselves tend to the eigenvectors of $\bV \bV^\top$, which is enough to show:

\begin{proposition}
  \label{prop:standard-dpp-limit}
  Let $\Xe \sim DPP(\varepsilon\bA + \bV\bV^\top)$. Then the limiting process $\Xs$ is $\Xs \sim
  DPP( \bV \bV^\top)$.
\end{proposition}The result is not very surprising. It has a noteworthy consequence, which is
that as $\flatlim$, the expected sample size will  be bounded by $p$ from above:  
\[ \E(|\Xe|) = \sum_{i=1}^n \pi_i(\varepsilon) = \sum_{i=1}^p \frac{  \lambda_{i,0}}{1+ \lambda_{i,0}} +
  \O(\varepsilon) \le p + \O(\varepsilon).\] 
If we wish to sample a larger number of points on average, then it appears that we
are out of luck.

\subsubsection{A more interesting limit}

We may instead look at a very similar limit: instead of taking $\bL_\varepsilon$, we will now take 
\[
\bL'_\varepsilon
= \varepsilon^{-1} \bL_\varepsilon = \bA + \varepsilon^{-1}\bV\bV^\top,
\]
which
carries the same intuition of giving more importance to $\bV\bV^\top$ than $\bA$. Since
we know the limiting eigenvalues and eigenvectors of $\bL_\varepsilon$, we know those of
$\bL'_\varepsilon$: the eigenvectors are unaffected, but the eigenvalues are scaled by
$\varepsilon^{-1}$. 

The scaling affects the probabilities of including eigenvectors, since we now have:
\[ \pi'_i(\varepsilon) = \frac{\varepsilon^{-1}  \lambda_i(\varepsilon)}{1+ \varepsilon^{-1} \lambda_i(\varepsilon)}
  =
  \begin{cases}
    1 + \O(\varepsilon), &  \mbox{\ if \ } i \leq p, \\
    \frac{ \lambda_{i,1}}{1+ \lambda_{i,1}}(1+ \O(\varepsilon)) , & \mbox{\ otherwise}. \\
  \end{cases}
\]
With the new scaling, the probability  of being included goes to $1$  for  the $p$ first eigenvectors, and  tends  to  
$\frac{\lambda_{i,1}}{1+ \lambda_{i,1}}$ for the remaining eigenvectors. We have
a partial-projection DPP, i.e., we obtain:
\begin{proposition}
  \label{prop:standard-dpp-limit-partial}
  Let $\Xe \sim DPP(\bA +
  \varepsilon^{-1} \bV\bV^\top)$. Then the limiting process is $\Xs \sim
  \ppDPP\ELE{\bA}{\bV}$
\end{proposition}
 Importantly, the expected sample size goes to:
\[ \E(|\Xs|) = \sum_i^n \pi_i=1 = p + \sum_{i=p+1}^{n} \frac{ \lambda_{i,1}}{1+ \lambda_{i,1}} \geq p, \]
so the rescaled L-ensemble allows for a larger sample size.

\subsection{Scaling L-ensembles to control sample size}
\label{sec:scaling-functions}

To sum up, partial-projection DPPs also arise as limits of L-ensembles. The
types of limits we obtain are analogous to the fixed-size case, but some
attention has to be paid to scaling, so that $|\X|$ is controlled in
expectation. 
The goal of this section is to motivate rescalings of the form
$\alpha\varepsilon^{-p}\bL_\varepsilon$. It is technical and may be skipped on
a first reading. Here we shall consider general kernels at an abstract level,
and not just the matrix pencils studied in the rest of the section.
 
In L-ensembles, the natural way of controlling the expected sample size is to multiply the
positive semi-definite matrix $\bL$ it is based on by a scalar. In other words, we need to rescale $\bL$ to  $\beta \bL$, with $\beta$
such that 
\begin{equation*}
  \E |\Xe| = \Tr \left( \beta \bL_\varepsilon ( \beta \bL_\varepsilon + \bI)^{-1}\right) = m,
\end{equation*}
where $m$ is the average sample size we would like to obtain. 
Rescaling by a scalar is a natural process if one thinks of the elements of $\bL$ as
representing similarity, which is defined on a ratio scale (i.e. the similarity
between $i$ and $j$ is actually $\frac{L_{ij}}{\sqrt{L_{ii}L_{jj}}}$, which is
invariant to rescaling by a scalar). 
The effect of rescaling is best seen from the point of view of the inclusion
probabilities of the eigenvectors (that we noted $\pi_i$ above). For $\Xe \sim DPP(\beta\bL_\varepsilon)$,
we have
\begin{equation}
  \label{eq:sum-of-probs-rescaled}
  \E |\Xe| = \sum_{i=1}^n \pi_i = \sum_{i=1}^n \frac{\beta\lambda_i(\varepsilon)}{1+\beta\lambda_i(\varepsilon)} = s_\varepsilon(\beta).
\end{equation}
It is not too hard to see that $s_\varepsilon(\beta)$ is a continous, monotonic
function of $\beta$ and that: 
\[0 = s_\varepsilon(0) \leq s_\varepsilon(\beta) < \rank \bL_\varepsilon \]
Because $s$ is monotonic, for every $\varepsilon$ there exists a unique $\beta$
such that $s_\varepsilon(\beta) = m$ for $m < \rank \bL_\varepsilon$. This
value of $\beta$ is an implicit function of $m$ and $\varepsilon$, which we
note $\beta^\star_m(\varepsilon)$. One may verify using the implicit function theorem
that $\beta^\star_m(\varepsilon)$ is continuous and differentiable.
In addition, it has an expansion in $\varepsilon$ as a
Puiseux series. To see why, note that $s_\varepsilon(\beta)=m$ may be rewritten
as a polynomial equation:
\begin{align*}
  \sum_{i=1}^n
  \frac{\beta\lambda_i(\varepsilon)}{1+\beta\lambda_i(\varepsilon)} = m \iff
  \sum_{i=1}^n
  \beta\lambda_i(\varepsilon) \prod_{j \neq i} (1+\beta\lambda_j(\varepsilon)) = m \prod_{j=1}^n (1+\beta\lambda_j(\varepsilon)),
\end{align*}
which is a polynomial in $\beta$, with coefficients that depend analytically on
$\varepsilon$ (via the $\lambda_i$'s). We call the solution $\beta^\star_m(\varepsilon)$ a
scaling function because it specifies how to rescale the matrix $\bL$ (as a
function of $\varepsilon$) so that $\E(|\Xe|)=m$ for all $\varepsilon$. 

Because $\beta^\star_m(\varepsilon)$ is the solution of a polynomial equation
with analytical coefficients, the Newton-Puiseux theorem 
states that the solution can be written (in an non-empty, punctured neighbourhood of 0, see
\cite{moro2002first}) as:
\begin{equation}
  \label{eq:puiseux-series}
  \beta^\star_m(\varepsilon) = \sum_{i=-s}^{\infty} \alpha_i \varepsilon^{i/c},
\end{equation}
where $c$ is some positive integer and $s$ determines the order of the
divergence at 0. This Puiseux series is simply a
Laurent series in $\varepsilon^{1/c}$. While we could go deeper in the study of
scaling functions, it would require introducing quite a bit of
background on Newton diagrams (which enable us to show for instance that $c=1$
in most cases). Instead, for the purposes of this article, we are content to note that scaling functions are
asymptotically of the form $\alpha \varepsilon^{-p}$ for some $\alpha$ and $p$
that depend on $m$. In the theorems below (section \ref{sec:varying-size}), we
study limits of L-ensembles rescaled by $\alpha \varepsilon^{-p}$, and describe what
happens as $p$ varies.

\subsection{A summary}
\label{sec:projection-DPPs-a-summary}
It may be helpful to take a step back and look broadly at the space of DPPs,
fixed-size DPPs, partial-projection DPPs and their relationships. Recall figure \ref{fig:venn-diagram}. Partial projection
DPPs can be thought of as forming part of the boundary of the space of DPPs.
Seen from the point of view of marginal kernels, they are on the boundary
of the set $\mathcal{K}$ of positive semi-definite matrices with eigenvalues
between 0 and 1 (since in a partial projection DPP, at least one of the
eigenvalues equals 1). Seen from the point of view of L-ensembles, partial
projection DPPs can be obtained by taking certain limits.
The following facts are useful to keep in mind:
\begin{itemize}
\item A projection DPP may be obtained by taking the limit in $\flatlim$ of the
   L-ensemble $\bL(\varepsilon) = \varepsilon^{-1} \bV \bV^\top$. The
  limiting DPP is a projection DPP, $\Xs \sim \mDPP{\rank \bV}(\bV\bV^\top)$. It
  has an L-ensemble as a fixed-size DPP, but not as a DPP (the L-ensemble
  diverges in the limit). 
\item A partial projection DPP may be obtained by taking the limit in $\flatlim$ of the
  L-ensemble $\bL(\varepsilon) = \bA + \varepsilon^{-1} \bV \bV^\top$.
  This is proposition \ref{prop:standard-dpp-limit-partial}.
\item A partial projection DPP with fixed-size $m$ may be  obtained by taking the
  limit in $\flatlim$ of a $\mDPP{m}$ with  $\bL(\varepsilon) = \varepsilon\bA + \bV \bV^\top$, if
  $m\geq \rank \bV$. This is theorem \ref{thm:partial-proj-limit}. If $m = \rank
  \bV$, then the limit is a projection DPP. 
\end{itemize}

\section{The flat limit of fixed-size L-ensembles (univariate case)}
\label{sec:univariate-results}

Now that we have introduced partial-projection DPPs, and seen how they arise
as limits in the specific case of pencil matrices, we have the requisite tools to deal with flat limits of L-ensembles in general. In this
section and the two following ones, we study L-ensembles based on kernel matrices taken in the flat limit. More specifically, Section~\ref{sec:univariate-results} starts gently with \emph{fixed-size L-ensembles} in the \emph{univariate} (the ground set $\Omega$ is a subset of the real line) case. Then, Section~\ref{sec:results-multivariate} extends these results to the \emph{multivariate} case ($\Omega\subseteq\mathbb{R}^d$, $d\geq1$), but still in the fixed-size context. Finally, Section~\ref{sec:varying-size} deals with the more involved limits of \emph{varying-size L-ensembles}, again first in the univariate case before extending to the multivariate case.

We begin by defining our objects of study, and summarise a few properties of
determinants in the flat limit, taken from
\cite{lee2015flatkernel,BarthelmeUsevich:KernelsFlatLimit}. We then apply these
results to study the flat limit of fixed-size L-ensembles, which as we will see depends
mostly on $r$, the smoothness parameter of the kernel. The section concludes
with some numerical results.

\subsection{Introduction}
\label{sec:1d-intro}

We focus on stationary kernels, as defined in section
\ref{sec:intro-kernels}, where $\varepsilon$ plays the role of an inverse scale
parameter. Thus, we consider L-ensembles based on matrices of the form
\[ \bL(\varepsilon) = [\kappa_\varepsilon(x_i,x_j)]_{i=1,j=1}^n\]
for a set of points $\Omega = \{x_1, \ldots, x_n\}$, all on the real line and all different from one another. 
From stationarity, the kernel function $\kappa_\varepsilon$ may be
written as:
\[ \kappa_\varepsilon(x_i,x_j) = f(\varepsilon|x_i-x_j|)\]
and we further assume that $f$ is analytic in a neighbourhood of 0. As in
equation \eqref{eq:kernel-expansion}, we expand the kernel in powers of $\varepsilon$ as:
\[ \kappa_\varepsilon(x_i,x_j) = f_0 +  \varepsilon f_1 |x_i-x_j| +
  \varepsilon^2 f_2 |x_i-x_j|^2  + \varepsilon^3 f_3|x_i-x_j|^3 +  \ldots\]
The expansion for individual entries may be represented in a more compact and
familiar manner in a matrix form:
\begin{equation}
  \label{eq:kernel-matrix-expansion}
  \bL(\varepsilon) = f_0 \bD^{(0)} + \varepsilon f_1 \bD^{(1)} + \varepsilon^2 f_2 \bD^{(2)} + \ldots
\end{equation}
where 
\[ \bD^{(p)} = [|x_i-x_j|^p]_{i,j}\]
Our goal is to characterise the limiting processes that arise from varying-size and fixed-size L-ensembles based on $\bL(\varepsilon)$ as $\flatlim$. One may recognise in Eq.~\eqref{eq:kernel-matrix-expansion} a more complex version of the linearly
perturbed matrix studied in section \ref{sec:ppDPP-as-limits}. It is indeed
useful to think of the terms $\varepsilon^i f_i \bD^{(i)}$ as containing
features that are increasingly down-weighted as $\flatlim$. The analysis is more
complicated than in the simple case above, notably because the matrices
$\bD^{(i)}$ are rank-deficient for even $i$ (up to some index depending on $n$) 
but invertible for odd $i$ \cite{BarthelmeUsevich:KernelsFlatLimit}. 
The smoothness order of the kernel (see section
\ref{sec:intro-kernels}) defines how soon in the decomposition the first
invertible matrix appears. For instance, if $r=2$ then $f_1 = 0$ and we get:
\[   \bL(\varepsilon) = f_0 \bD^{(0)} +
  \varepsilon^2 f_2 \bD^{(2)} + \varepsilon^3 f_3 \bD^{(3)} + \ldots  \]
If $n>2$, the first invertible matrix to appear in the expansion in
$\varepsilon$ is $\bD^{(3)}$, and it will lead to different asymptotic
behaviour than if the first invertible matrix had been $\bD^{(1)}$ ($r=1$) or
$\bD^{(5)}$ ($r=3$). If the kernel is completely smooth, then:
\[   \bL(\varepsilon) = \sum_{i=0}^\infty \varepsilon^{2i}f_{2i}\bD^{(2i)}  \]
and odd terms never appear. This again has its own asymptotic behaviour. A subtle issue is that if the matrix under
consideration is small enough compared to the regularity order, then the
asymptotics are the same than in the completely smooth case. We invite the
reader to pay attention to the interplay between $m$ (the size of the L-ensemble) and
$r$ (the regularity order) in our theorems. For more on the flat asymptotics of
kernel matrices, we refer again to \cite{BarthelmeUsevich:KernelsFlatLimit}. 

\subsubsection{Univariate polynomials and Vandermonde matrices}
\label{sec:univar-poly}

 Recall that we define the Vandermonde matrix of order $k$ as:
\begin{equation}\label{eq:Vandermonde1D}
  \matr{V}_{\le k} = \begin{bmatrix}
    1 & x_1 & \cdots & x^{k}_1\\
    \vdots &  & \vdots \\
    1 & x_n & \cdots & x^{k}_n 
  \end{bmatrix},
\end{equation}
where $x_1, \ldots, x_n$ are the $n$ points of the ground set $\Omega$ (We may sometimes use
the notation $\bV_{<k} = \bV_{\leq k-1}$ as well).
Note that $\matr{V}_{\le k}$ has $k+1$ columns. 
The ``classical'' Vandermonde matrix has $k=n-1$, which makes it square.
$\bV_{\le n-1}$ is invertible if and only if the points in $\Omega$ are distinct, which can
be established from the following well-known determinantal formula:
\begin{equation}
  \label{eq:VdM-determinant}
  \det \matr{V}_{\le n-1} = \prod_{i<j} (x_i-x_j)
\end{equation}
As short-hand, we shall use $\vect{v}_l=
\begin{pmatrix}
  x_1^{l-1}, \ldots, x_n^{l-1}
\end{pmatrix}^\top
$ to denote the $l$-th column of $\matr{V}$. 
Submatrices of $\bV_{\leq k}$ corresponding to a subset of points $X$ will be
denoted $\bV_{\leq k}(X) \in \R^{|X| \times (k+1) }$.

\subsubsection{Some results on limiting determinants and spectra}
\label{sec:flat-limit-summary}

In this section we summarise  some of the main results from
\cite{BarthelmeUsevich:KernelsFlatLimit}. These concern the limiting
determinants and spectra of kernel matrices. All we need for the proofs are the
results on the limiting determinants, but the results on asymptotic spectra may
help understand how the limiting process arises.

The statements involve the Wronskian matrix of the kernel, which we now define. The
Wronskian is a matrix of derivatives of the kernel at 0, specifically:
\begin{equation}\label{eq:Wronskian}
  \matr{W}_{\leq k} \eqdef
  \begin{bmatrix}
    \frac{\kappa^{(0,0)}(0,0)}{0!0!}  & \cdots & \frac{\kappa^{(0,k)}(0,0)}{0!k!} \\
    \vdots &  & \vdots \\
    \frac{\kappa^{(k,0)}(0,0)}{k!0!}  & \cdots & \frac{\kappa^{(k,k)}(0,0)}{k!k!} 
  \end{bmatrix}.
\end{equation}
Thus, $\bW_{\leq k}$ contains derivatives up to order $k$. It is important to realise that $\matr{W}$ depends \emph{only} on the kernel,
and is independent of the locations $\Omega$ at which the kernel is evaluated.

The first theorem concerns the limiting determinants in the smooth case, which tie
in directly to Vandermonde determinants: 
\begin{theorem}\label{thm:det_1d_smooth}
Let $\kappa$ be a kernel function and $X$ a set of $m$ points. 
If the smoothness order $r$ satisfies $r\geq m$ then, for small $\varepsilon$, the determinant of $\bL_{X}(\varepsilon) =
  [\kappa(\varepsilon x_i,\varepsilon x_j)]_{i,j=1}^m$ has the  expansion
      \begin{equation}
      \det (\matr{L}_{X}(\varepsilon)) = \varepsilon^{m(m-1)} (\det(\matr{V}_{\le m-1}(X))^2 \det \matr{W}_{\leq m-1} + \O(\varepsilon)).\label{eq:det_1d_smooth}
    \end{equation}
\end{theorem}
We have made explicit in the notation the quantities that depend on the points
$X$ versus those that do not. 

This result appeared originally in \cite{lee2015flatkernel}, and can be found in this form in theorem 4.1 of~\cite{BarthelmeUsevich:KernelsFlatLimit}. It can be
generalised to cases with lower order of smoothness, leading to:
\begin{theorem}\label{thm:det_1d_finite_smoothness}
Let $\kappa$ be a kernel function and $X$ a set of $m$ points. 
	If the smoothness order $r$ satisfies $r\leq m$ then, for small $\varepsilon$, the determinant of $\bL_{X}(\varepsilon) =
	[\kappa(\varepsilon x_i,\varepsilon x_j)]_{i,j=1}^m$ has the  expansion
\begin{equation}
\det (\matr{L}_{X}(\varepsilon)) =\varepsilon^{m(2r-1)-r^2} \left( \widetilde{l}(X) + \O(\varepsilon)\right),
\label{eq:det_finite_smoothness}
\end{equation}
where the main term is given by
\begin{equation}
\widetilde{l}(X) =
(-1)^{r} 
\det \matr{W}_{\leq r-1}  \det
\begin{bmatrix}
f_{2r-1}  \bD^{(2r-1)}(X) &  \matr{V}_{\le r-1}(X) \\
 \matr{V}_{\le r-1}(X) ^\top& 0 
\end{bmatrix}\\
\end{equation}
\end{theorem}
\begin{remark}
  Note that for $r=m$,  equations \eqref{eq:det_1d_smooth}  and \eqref{eq:det_finite_smoothness}  coincide,  since $ \matr{V}_{\le m-1}(X)$ is square and the determinant in (\ref{eq:det_finite_smoothness}) reduces to $ (-1)^m \det(\matr{V}_{\le m-1}(X)^2$.
\end{remark}

\begin{remark}
  In the introduction (see fig.
  \ref{fig:illus-dpp}), we stated that while determinants of kernel matrices go to
  0 in the flat limit, ratios of determinants go to a finite value. The statement
  follows as a direct consequence of thm. \ref{thm:det_1d_smooth} and
  \ref{thm:det_1d_finite_smoothness}:. For instance, under the conditions of
  Theorem 
  \ref{thm:det_1d_smooth}, we have:
  \[ \frac{\det
      (\matr{L}_{X'}(\varepsilon))}{\det (\matr{L}_{X}(\varepsilon)) } =
    \frac{\det(\matr{V}_{\le m-1}(X'))^2}{\det(\matr{V}_{\le m-1}(X))^2} +
    \O(\varepsilon) \]
  By itself this observation is almost enough to prove convergence. 

\end{remark}

\subsection{Flat limit in the fixed-size case}
\label{sec:univariate-fixed-size}

Consider $\X_\varepsilon \sim \mDPP{m}(\bL(\varepsilon))$ with $m\leq n$ and $m$ and $n=|\Omega|$ fixed (no large $n$ asymptotics are involved here). We are interested in
the limiting distribution of $\X_\varepsilon$ as $\flatlim$.

It is not at first blush obvious that the limiting point process exists
and is non-trivial. Indeed, as $\varepsilon\rightarrow0$, every entry of the
matrix $\bL(\varepsilon)$ goes to 1, and so $\det(\bL(\varepsilon)_X)$ goes to 0 for all
subsets $X$. What makes the limit non-trivial is, as we shall see in the
proofs, that these quantities go to 0 at different speeds.

The first result characterises the smooth case, where the smoothness order of
the kernel is larger than $m$.
\begin{theorem}
  \label{thm:cs-case-fixed-size-1d}
  Let $\bL_\varepsilon = [\kappa_\varepsilon(x_i,x_j)]_{i,j}$ with $\kappa$ a
  stationary kernel of smoothness order $r \geq m$. Then $\Xe \sim
  \mDPP{m}(\bL_\varepsilon)$ converges to $\Xs \sim \mDPP{m}(\bV_{\leq m-1}
  \bV_{\leq m-1}^\top)$.   
\end{theorem}
\begin{proof}
  The result follows directly from theorem \ref{thm:det_1d_smooth}, applied to
  minors of $\bL(\varepsilon)$ of size $m \times m$, and lemma \ref{lem:TV-convergence}.  
  To be   more explicit, let $\bL^\star = \bV_{\leq m-1}\bV_{\leq m-1}^\top$. Theorem \ref{thm:det_1d_smooth} implies that:
  \[ \Proba(\X_\varepsilon = X) =  \frac{  \varepsilon^{m(m-1)} \left( \det
        \matr{W}_{\leq m-1}  \det
        \bL^\star_{X} + \O(\varepsilon)  \right) }{ \varepsilon^{m(m-1)} \left(
        \det \matr{W}_{\leq m-1} \sum_{Y,|Y|=m} \det
        \bL^\star_{Y} + \O(\varepsilon)  \right)}
  \]
We  may apply lemma \ref{lem:TV-convergence} directly. $\X_\varepsilon$ tends to
  $\X_\star$, a fixed-size DPP with law: 
  \[ \Proba(\X_\star = X) =  \frac{\det
      \bL^\star_{X} }{\sum_{Y,|Y|=m} \det
      \bL^\star_{Y} } \]
\end{proof}
\begin{remark}
  The result says that as $\flatlim$ the limiting point process is (a) a fixed-size L-ensemble (and even a projection DPP as $\bV_{\leq m-1} \bV_{\leq m-1}^\top$ is of rank $m$)
  and (b) the positive semi-definite matrix it is based on is a Vandermonde matrix of
  $\Omega$. It is worth studying this matrix in greater detail. Let
  $\matr{M} = \bV_{\leq m-1} \bV_{\leq m-1}^\top$. Then for any subset $X \subset
  \Omega$ of size $m$, $\det \matr{M}_X=\det^2(\bV_{\leq m-1}(X))$, because
  $\bV_{\leq m-1}(X)$ is a square matrix. From the Vandermonde determinant
  formula (eq. \eqref{eq:VdM-determinant}), this means that if $\X \sim
  \mDPP{m}(\matr{M})$,
  \begin{equation}
    \label{eq:cs-case-joint-prob}
    \Proba\left(X\right) = \frac{1}{Z} \prod_{(x,y)\in X^2} (x - y)^2
\end{equation}
  
\end{remark}
\begin{remark}
  The conditional law $\Proba(\Xe = \{x\} \cup Y | Y)$ (the conditional law of one of the
  points when the rest are fixed) tends to:
  \[ \Proba\left(\Xs = \{x\} \cup Y | Y\right) \propto \prod_{y \in Y} (x-y)^2\]
  which is evidently a repulsive point process (since small distances between
  points are unlikely). 
\end{remark}
To summarise: if we sample a fixed-size L-ensemble of size $m$, and the kernel is
regular enough compared to $m$ (\emph{i.e.}, $r\geq m$), then \emph{whatever} the kernel the
limiting process exists and is the same\footnote{The ``whatever the kernel''
  part becomes more complicated in the multidimensional case, as we shall see.}. The probability of sampling a set $X$
is just proportional to a squared Vandermonde determinant, and that defines a projection DPP.  

The next theorem describes what happens when the kernel is less smooth. We obtain a partial projection
kernel, where the projective part comes from polynomials, and the non-projective
part comes from the first nonzero odd term in the kernel expansion (see Eq.~\eqref{eq:kernel-matrix-expansion}).

\begin{theorem}
  \label{thm:fs-case-fixed-size-1d}
  Let $\bL_\varepsilon = [\kappa_\varepsilon(x_i,x_j)]_{i,j}$ with $\kappa$ a
  stationary kernel of smoothness order $r \leq m$. 
  Then $\Xe \sim
  \mDPP{m}(\bL_\varepsilon)$ converges to $\Xs \sim \mppDPP{m} \ELE{\bD^{(2r-1)}} {\bV_{\leq r-1}}$.   
\end{theorem}
\begin{proof}
  The argument is exactly the same as in theorem
  \ref{thm:cs-case-fixed-size-1d}, this time using the limiting form of the
  determinant given by theorem \ref{thm:det_1d_finite_smoothness}. 
\end{proof}
\begin{example}
  In the case of the exponential kernel $\kappa_{\varepsilon}(x,y) = e^{ -\varepsilon |x-y| }$, $r=1$, and the theorem states that
  \begin{equation}
    \label{eq:extended-L-ensemble-exp-kernel}
    \Proba(\Xs = X) \propto \det
    \begin{pmatrix}
      -\bD_{X}^{(1)} & \ones \\
      \ones^\top & 0
    \end{pmatrix}
  \end{equation}
  Equivalently, from a mixture point of view, the constant eigenvector
  $\vect{q}_0=\frac{1}{\sqrt{n}}\ones$ is sampled with probability 1, and the remaining
  $m-1$ eigenvectors are sampled from a (diagonal) fixed-size L-ensemble with diagonal entries equal
  to the eigenvalues of $\tilde{\bD}=-(\bI-\vect{q}_0\vect{q}_0^\top)\bD^{(1)}(\bI-\vect{q}_0\vect{q}_0^\top)$
\end{example}

\begin{remark}
  Some algebra reveals that
  \begin{equation}
    \label{eq:det-exp-kernel}
    \det
    \begin{pmatrix}
      -\bD_{X}^{(1)} & \ones \\
      \ones^t & 0
    \end{pmatrix}
    = (2)^{m-1} \prod_{i = 1}^m (x_{i+1} - x_i)
  \end{equation}
  where in the last expression we have sorted the points in $X$ so that $x_1
  \leq x_2 \leq \ldots \leq x_m$. 
  As in  \eqref{eq:cs-case-joint-prob} above, the repulsive nature of the limit
  point process is immediately apparent from eq. \eqref{eq:det-exp-kernel}.
  Unlike \eqref{eq:cs-case-joint-prob}, which involves all distances, eq.
  \eqref{eq:det-exp-kernel} only involves distances between direct neighbours.
  We speculate that similar expressions exist for $r>1$ but we unfortunately
  have not been able to derive them.
\end{remark}
\begin{proof}
  Eq. \eqref{eq:extended-L-ensemble-exp-kernel} may be derived by using a finite
  difference operator of the form:
  \[
    \matr{F} =
    \begin{pmatrix}
      1 & 0 & \ldots \\
      \frac{-1}{\delta_1} & \frac{1}{\delta_1} & 0 & \ldots \\
      0 & \frac{-1}{\delta_2} & \frac{1}{\delta_2} & 0 & \ldots \\
      \vdots & \vdots & \vdots & \vdots
    \end{pmatrix}
  \]
  where $\delta_i = x_{i+1}-x_i$. Since $\matr{F}$ is lower-triangular, $\det
  \matr{F}= \prod_{i=1}^{m-1} \delta_i^{-1}$.  Then applying
  lemma \ref{lem:det-saddlepoint} to
  \[ \det(
    \begin{bmatrix}
      \matr{F} & 0 \\
      0 & 1
    \end{bmatrix}
    \begin{bmatrix}
      -\bD_{X}^{(1)} & \ones \\
      \ones^t & 0
    \end{bmatrix}
    \begin{bmatrix}
      \matr{F}^t & 0 \\
      0 & 1
    \end{bmatrix}
    ) \] and simplifying yields the result.
\end{proof}

\subsection{Some numerical illustrations}
\label{sec:numerics-1d}

To illustrate the convergence theorems above, a good visual tool is to examine
the convergence of conditional distributions of the form:
\begin{equation}
  \label{eq:cond-law-dpp}
  \Proba(\X = \{x\} \cup Y | Y) \propto \det \bL_{\{x\} \cup Y} \propto  (\bL_{x,x}-\bL_{x,Y}\bL_{Y}^{-1} \bL_{Y,x})
\end{equation}
This should be interpreted as the conditional probability of the $m$-th item
fixing the first $m-1$. The conditional law $\Proba(\Xe = \{x\} \cup Y | Y)$ tends
to that of $\Proba(\Xs = \{x\} \cup Y | Y)$, and in dimension 1 we can depict this,
 as a function of $x$.

We do so in figure \ref{fig:convergence-cond-1d}, where we assume $\X$ is a
$m=5$ fixed-size L-ensemble, and the ground set is a finite subset of $[0,1]$. The conditioning subset $Y$
is chosen to be of size 4, 
and for the sake of illustration, we let $x$ vary as a continuous parameter in $[0,1]$.
 The four panels correspond to four different kernel
functions. The conditional probability is plotted for different values of $\varepsilon$. In all plots we observe a rapid convergence
with $\varepsilon$. In the top panel, the difference between the asymptoptics obtained for $r=1$ and
$r = \infty$ are quite striking. In the bottom panel, we have two different
kernels with identical smoothness index, and as predicted by Theorem
\ref{thm:fs-case-fixed-size-1d} the $\flatlim$ limits are identical.
 
\begin{figure}
  \centering
  \begin{subfigure}[b]{0.48\textwidth}
    \includegraphics[width=\textwidth]{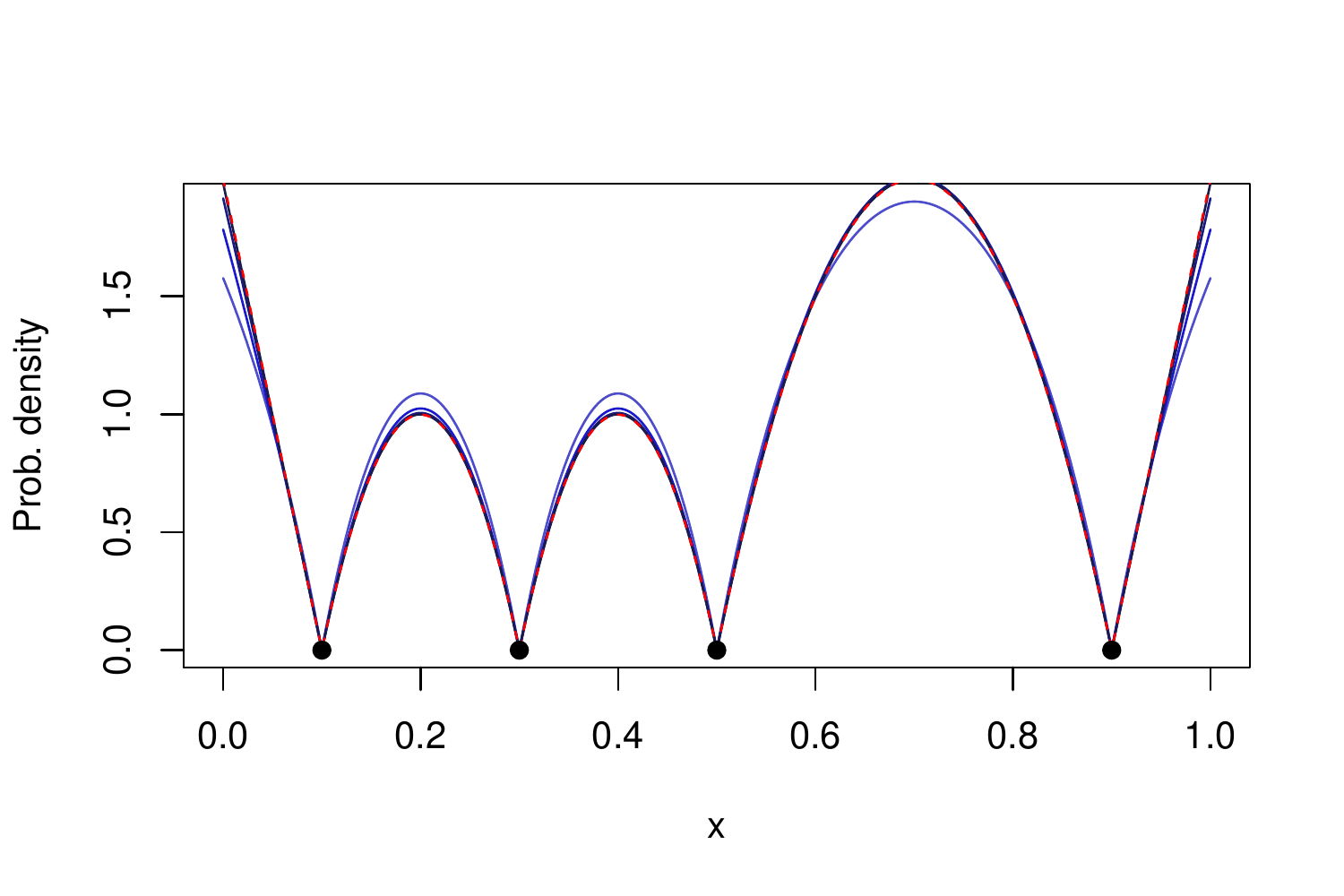}
    \caption{$k(x,y)=\exp(-|x-y|)$, a kernel with $r=1$}
  \end{subfigure}
  \hfill
  \begin{subfigure}[b]{0.48\textwidth}
    \includegraphics[width=\textwidth]{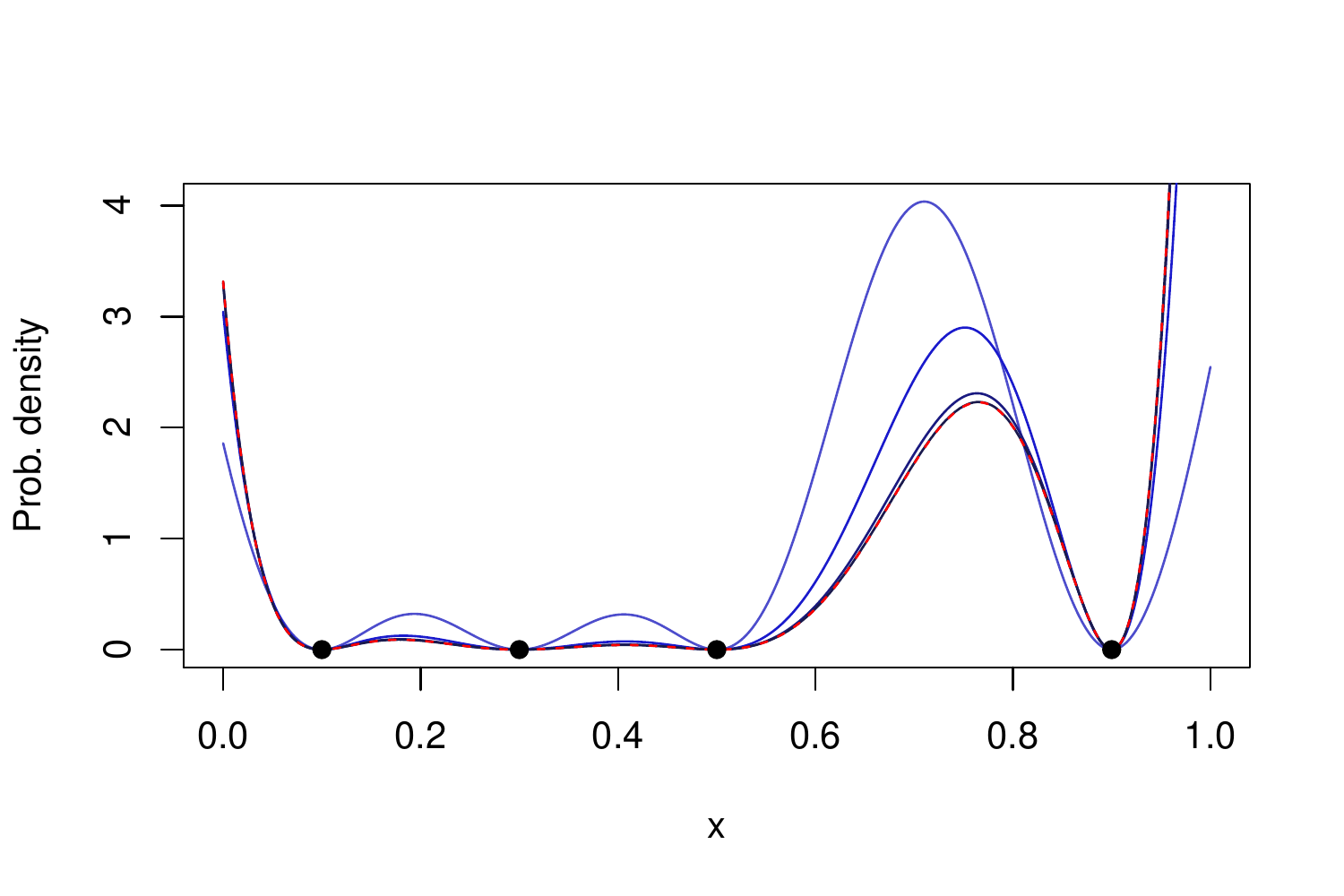}
    \caption{$k(x,y)=\exp(-(x-y)^2)$, a kernel with $r=\infty$}
  \end{subfigure}
  \\
  \begin{subfigure}[b]{0.48\textwidth}
    \includegraphics[width=\textwidth]{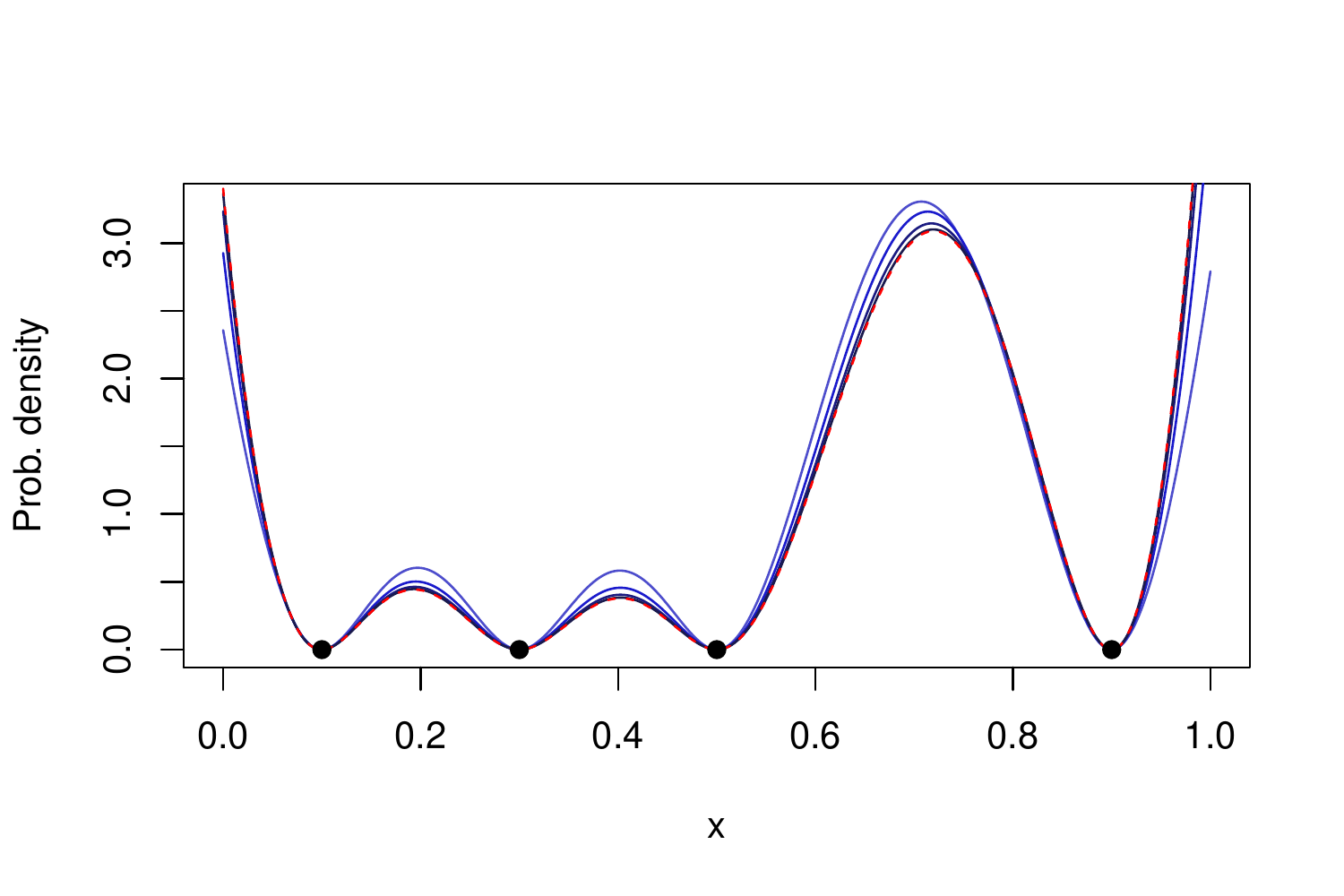}
    \caption{$k(x,y)=(1+|x-y|)\exp(-|x-y|)$, a kernel with $r=2$}
  \end{subfigure}
  \hfill
  \begin{subfigure}[b]{0.48\textwidth}
    \includegraphics[width=\textwidth]{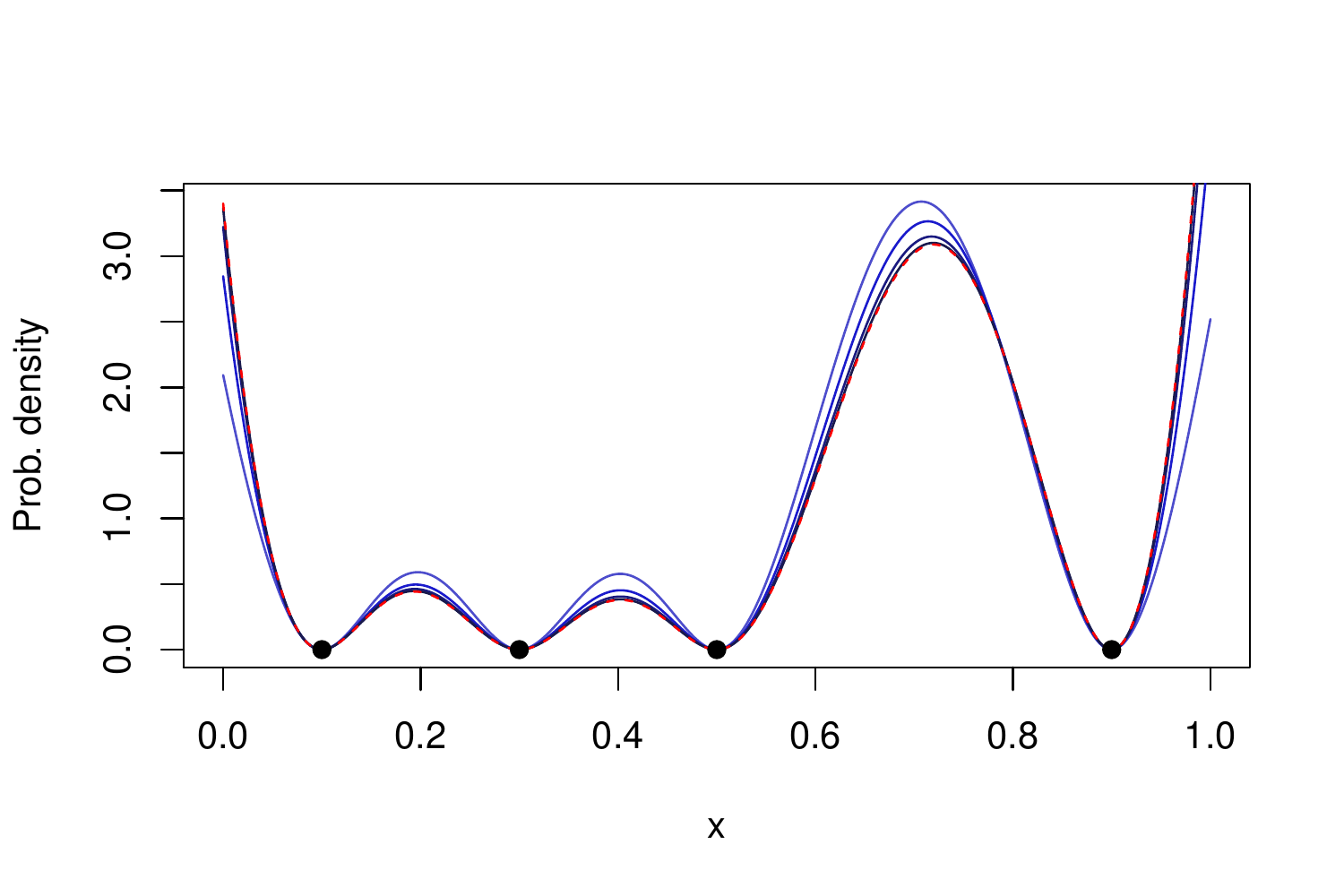}
    \caption{$k(x,y)=\sin(|x-y|+\frac{\pi}{4})\exp(-|x-y|)$, another kernel with $r=2$}
  \end{subfigure}

  \caption{Asymptotics of conditional densities of L-ensembles based on four different
    kernels. Here we plot $\Proba(\Xe = \{x\} \cup Y | Y)$, the conditional density of a fixed size L-ensemble
    (with $m=5$) where four of the points are fixed ($Y$) and the last is varying
    $(x)$. The points in $Y$ are at $0.1,0.3,0.5,0.9$. The curves in blue are
    the conditional densities for different values of $\varepsilon$:
    $4,1.5,.5,.1$, in blue. The dotted red line is
  the asymptotic limit in $\flatlim$. Note that the two kernels in the bottom row have
the same regularity coefficient $r=2$, and as predicted by the results the
limiting densities are equal.}
\label{fig:convergence-cond-1d}

\end{figure}

Another set of quantities that are easy to examine visually are the first order
inclusion probabilities ($\Proba(x \in \X)$). We refer to
\cite{Barthelme:AsEqFixedSizeDPP} for how to compute these quantities in
fixed-size L-ensembles. Since $\Xe$ converges to $\Xs$, so must the inclusion
probabilities, and this is shown in figure \ref{fig:convergence-inclusion-1d} for
three kernels with increasing values of $r$.  For these plots, the ground set consists in 20 points drawn at random in the unit interval.
We depict the first order inclusion probabilities for four different values of $\varepsilon$. Rapid convergence with $\varepsilon$ is also oberved. 
\begin{figure}
  \includegraphics[width=\textwidth]{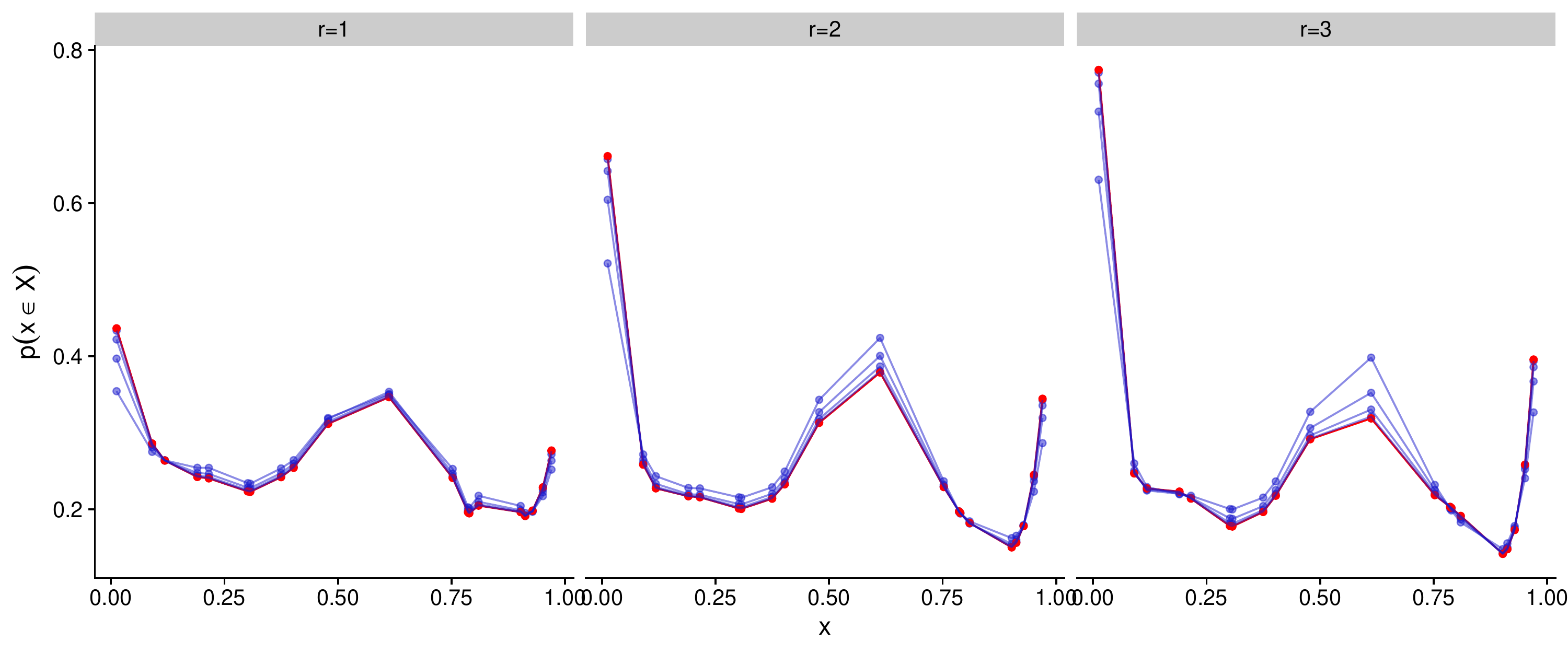}
  \caption{Flat limit of inclusion probabilities of (fixed-size) L-ensembles for three different
    kernels. Here we plot $\Proba(x \in \Xe)$, the inclusion probabilities for a fixed size L-ensemble
    (with $m=5$), where the ground set $\Omega$ consists in 20 points drawn at
    random from the unit interval. The dots in blue (joined by lines for
    clarity) are inclusion probabilities for 
    $\varepsilon = 4,1.5,.5,.1$. The dots in red correspond to
  the asymptotic limit in $\flatlim$. The three kernels are, from left-to-right,
$\exp(-\delta), (1+\delta)\exp(-\delta),(3+3\delta+\delta^2)\exp(-\delta)$,
where $\delta = |x - y|$. These kernels have $r=1$,$2$ and $3$, respectively.}

\label{fig:convergence-inclusion-1d}
\end{figure}

\section{The flat limit of fixed-size L-ensembles (multivariate case)}
\label{sec:results-multivariate}

The univariate results we stated above have a multivariate generalisation, and
in some cases they are almost the same. The only major difference is that in the
univariate case, the \emph{only} aspect of the kernel function that plays a role
in determining the limiting process is the smoothness order $r$. Two kernels may
look different, but if they have the same smoothness order they have the same
limiting DPP. 

When $d>1$ this is no longer always true. The limiting process may \emph{sometimes}
depend on the specific values of the derivatives of the kernel at 0 (not just whether they exist). 
Sometimes, but not always: for
instance, all kernels with $r=1$ give the same limiting fixed-size DPP. All kernels
with $r=2$ give the same limiting fixed-size ($m$) L-ensemble, as long as $m>d$. The case of
infinitely smooth kernels is particularly intriguing: there is a universal
limiting process, but only for $m$ in a set of ``magic'' values $\magicn{d}$ to be defined below. 
 When $m$ falls in between these values, then the
limiting process depends on the kernel (although perhaps not strongly).

To build a picture of what the final results look like, we state the easiest first:
\begin{example}
  Let $\bL_\varepsilon = [\kappa_\varepsilon(\vect{x}_i,\vect{x}_j)]_{i,j}$ with
  $\kappa$ a stationary kernel of smoothness order $r = 1$. 
  Then $\Xe \sim  \mDPP{m}(\bL_\varepsilon)$ converges to $\Xs \sim
  \mppDPP{m}  \ELE{-\bD}{\ones} $. 
\end{example}
A more general statement is given later, but this one has the advantage of being
identical to the univariate result. 

As the more general statements are also more complicated, we present our
results in increasing order of complexity. The general theorem is found at the
end of the section, and all results we state first (including the above) are
special cases. But before delving into this, we need to recall some aspects of Vandermonde matrices and introduce the magic numbers $\magicn{d}$. Furthermore, we will give in section \ref{sec:non-universal-limits} the spectral interpretation for the universal/non universal limits. We will then present the technical results. 

\subsection{Multivariate Vandermonde matrices}
\label{sec:multivar-poly} 

We recall for the sake of readability the appropriate generalisations for multivariate Vandermonde matrices presented
in the background section \ref{sec:polynomial_backgrounds} on polynomials.
For an ordered set of points $\Omega = \{\vect{x}_1, \ldots, \vect{x}_n\}$, all in $\RR^d$,  the multivariate Vandermonde matrix 
is defined as:
\begin{equation}\label{eq:VandermondeND}
  \matr{V}_{\le k} = 
  \begin{bmatrix}
    \matr{V}_{0} & \matr{V}_{1}  & \cdots & \matr{V}_{k}  
  \end{bmatrix} \in \mathbb{R}^{n\times \PP_{k,d}}
\end{equation}
where each block $\matr{V}_i \in\mathbb{R}^{n\times \HH_{i,d}}$ contains the monomials of degree $i$ evaluated on
the points in $\Omega$. 
    
As in the previous section, we use $\bV_{\le k}(X)$ to denote the matrix $\bV_{\le k}$ reduced to its lines indexed by the elements in $X$. As such, $\bV_{\le k}(X)$ has $|X| = m$ rows and $\PP_{k,d} $ columns. For some values of $m$ and $k$ it is square and (potentially) invertible. For instance, consider $\bV_{\leq k}$ as in Eq.~\eqref{eq:VandermondeND}, with $k=1$ and $d=2$. Choosing a subset $X$ of size $m=3$, the matrix $\bV_{\leq 1}(X)$ is square. In dimension 2, there exists a square Vandermonde matrix for sets $X$ of size $m=1$, $3$, $6$, $10$, $15$, $21$, etc.
	
In fact, for an arbitrary dimension $d$, there exists a square Vandermonde matrix for any size $m$ such that there exists $k\in\mathbb{N}$ verifying $\PP_{k,d} =m$, that is, any $m$ included in the set of integers:
\begin{equation}
\label{eq:magic-numbers}
\magicn{d} = \left\{ \PP_{k,d} | k \in \mathbb{N}\right\}.
\end{equation}
We will see that these values of $m$ are in some sense natural sizes for L-ensembles, because they
lead to universal limits, and that is the reason for calling them magic numbers. 

We note in passing that while we may easily determine whether $\bV_{\leq k}(X)$ is
square, whether it is invertible is a complicated question that depends on the
geometry of the points $X$, as there are some non-trivial configurations for
which it is not \cite{gasca2000polynomial}. The results below show that such configurations
 have probability 0 in the flat limit under any L-ensemble with $r$ sufficiently large compared to $m$.

\subsection{Universal and non-universal limits, a spectral view}
\label{sec:non-universal-limits}

To understand why universal limits sometimes arise and sometimes not, it is
worth making a small detour to examine the behaviour of the eigenvalues in the flat limit.

Schaback in \cite[Theorem~6]{schaback2005multivariate} showed that eigenvalues of completely smooth
kernels have different orders in $\varepsilon$. All but the first go to 0 as
$\flatlim$, but they do so at different rates. When $d=2$, the top eigenvalue is
$\O(1)$, the next two are $\O(\varepsilon^2)$, the next three are
$\O(\varepsilon^4)$, the next four are $\O(\varepsilon^6)$, etc. The reader may
notice that there are as many eigenvalues of order $\O(\varepsilon^{2i})$ as
$\HH_{i,d}$ the number of monomials of degree $i$ in dimension $d=2$. This is indeed the general case for
smooth kernels in any dimension $d$. In \cite{BarthelmeUsevich:KernelsFlatLimit} the result is
extended to finitely smooth kernels, and the main term in the expansion of the eigenvalues
as $\flatlim$ is given. In finitely smooth kernels of smoothness order $r$, the first $r$
groups of eigenvalues behave as in the completely smooth case, meaning that the first group
(of size $\HH_{0,d}=1$) has order $\O(1)$, the second of size $\HH_{1,d}=d$\ has order $O(\varepsilon^2)$, etc. up to
the group of order $\O(\varepsilon^{2(r-1)})$ with size $\HH_{r-1,d}$. Then all the remaining eigenvalues form a single
group of order $\O(\varepsilon^{2r-1})$ and of size $n-\PP_{r-1,d}$. For instance, if $r=2$, and $d=2$, the
top eigenvalue is $\O(1)$, the next two are $\O(\varepsilon^2)$, and the
remaining $n-3$ eigenvalues are all $\O(\varepsilon^3)$.
Let us examine this case more closely, in light of the spectral mixture
viewpoint on L-ensembles. The asymptotic expansion of the eigenvalues for $r=2$, and
$d=2$ are as follows:
\begin{align*}
  \left. \lambda_0(\varepsilon) = \lt_0 + \O(\varepsilon) \right\}\textrm{\textcolor{blue}{Group 1}}\\
  \left.
  \begin{array}{c}
\lambda_1(\varepsilon) = \varepsilon^2\left(\lt_1 + \O(\varepsilon)\right) \\
  \lambda_2(\varepsilon) = \varepsilon^2\left(\lt_2 + \O(\varepsilon)\right)
  \end{array}
  \right\}\textrm{\textcolor{blue}{Group 2}}\\
  \left.
  \begin{array}{c}
  \lambda_3(\varepsilon) = \varepsilon^3\left(\lt_3 + \O(\varepsilon)\right) \\
  \vdots \\
  \lambda_{n-1}(\varepsilon) = \varepsilon^3\left(\lt_{n-1} + \O(\varepsilon)\right)
  \end{array}
  \right\}\textrm{Group 3}
\end{align*}
We highlight the first two groups in blue because they correspond to the smooth
part of the spectrum, i.e. the part that behaves in the same way in the
completely smooth case. The rest is the non-smooth part. 
What the precise values of $\lt_0,\lt_1,\ldots$ are does not matter here (see
Theorem 6.3 in \cite{BarthelmeUsevich:KernelsFlatLimit} for the expression), but
what matters to this explanation is the following: in the smooth part, the
eigenvalues depend non-trivially on the Taylor expansion of the kernel
at 0. Different kernels with equal order of regularity may have different
asymptotic eigenvalues, but they will appear in groups with the same structure.
In the non-smooth part, that is not the case, apart from a trivial global
scaling that does not matter here. To sum up: in our example of $r=2$ and
$d=2$, as $\flatlim$,
$\frac{\lambda_2}{\lambda_1}$ depends on the kernel, while e.g.
$\frac{\lambda_5}{\lambda_4}$ does not. 
Now consider what happens when we sample a fixed-size L-ensemble, going into the
limit $\flatlim$, and bearing in mind lemma \ref{lem:TV-convergence-diverging}. 

With $m=1$, only the top eigenvector will ever be sampled
(its eigenvalue is $\O(1)$, all the rest are asymptotically smaller). The result
is a projection DPP and the limit is universal. With $m=2$, the top one is always sampled, then either of the
next two. We have a partial-projection DPP again. The relative probability of
sampling the second or third eigenvector depends on
$\frac{\lambda_2}{\lambda_1}$, which in turn depends on the kernel. The limit is
here non-universal. With $m=3$, the top three eigenvectors are necessarily
sampled, the ratio $\frac{\lambda_2}{\lambda_1}$ is irrelevant. Again, we find a
projection DPP as the universal limit. Finally, with $m>4$, we start hitting the
non-smooth part. The first three eigenvectors are necessarily sampled, and then
$m-3$ eigenvectors from the remaining ones. In that part of the spectrum the
ratios $\frac{\lambda_i}{\lambda_j}$ do not depend on the kernel, and so the
limit is universal (and a partial-projection DPP). 
In conclusion, with $r=2$ and $d=2$, there is a universal limit for every value
of $m$ except $m=2$. With $r=2$ and $d=3$, and repeating the same reasoning, we
find a universal limit for every $m$ except $m=2$ and $m=3$. 

Theorem \ref{thm:general-case-smooth-fixed-size} below will describe the general pattern for $m\leq \PP_{r-1,d}$, gives the asymptotic process for non-universal limits ($m$ non magic) and universal ($m$ magic).  Before presenting it, we will present separately the case of universal limits alone given for the cases $m> \PP_{r-1,d}$ and $m\in \magicn{d}$.

The statement of the theorems involves derivatives of the kernel.
A convenient short-hand notation for higher-order derivatives uses
multi-indices:
\[
  f^{(\va)}
  (\vect{x}) =\frac{\partial f^{|\va|}}{\partial x_1^{\alpha_1} \cdots\partial x_d^{\alpha_d}} (\vect{x})
\]
The Wronskian matrix of the kernel is defined as:
\begin{equation}\label{eq:wronskian_nd}
  \matr{W}_{\leq k}  = 
  \left[
    \frac{k^{(\va,\vect{\beta})} (\vect{0},\vect{0})}{\va!\vect{\beta}!}
  \right]_{|\va| \leq k, |\vb| \leq k} \in\mathbb{R}^{\PP_{k,d}\times \PP_{k,d}}.
\end{equation}
Here we index the matrix using multi-indices (equivalently, monomials), so that an element of
$\matr{W}_{\leq k}$ is e.g., $\matr{W}_{(0,2),(2,1)}$ which is a scaled
derivative of $k(\vect{x},\vect{y})$ of order $(0,2)$ in $\vect{x}$ and $(2,1)$
in $\vect{y}$.
For example, for $d=2$ and $k=2$ we may write 
\[
\matr{W}_{\leq 2} =\begin{bmatrix}
k^{((0,0),(0,0))} & k^{((0,0),(1,0))} & k^{((0,0),(0,1))} & \frac{k^{((0,0),(2,0))}}{2} & {k^{((0,0),(1,1))}} & \frac{k^{((0,0),(0,2))}}{2} \\
k^{((1,0),(0,0))} & k^{((1,0),(1,0))} & k^{((1,0),(0,1))} & \frac{k^{((1,0),(2,0))}}{2} & {k^{((1,0),(1,1))}} & \frac{k^{((1,0),(0,2))}}{2} \\
k^{((0,1),(0,0))} & k^{((0,1),(1,0))} & k^{((0,1),(0,1))} & \frac{k^{((0,1),(2,0))}}{2} & {k^{((0,1),(1,1))}} & \frac{k^{((0,1),(0,2))}}{2} \\
\frac{k^{((2,0),(0,0))}}{2} & \frac{k^{((2,0),(1,0))}}{2} & \frac{k^{((2,0),(0,1))}}{2} & \frac{k^{((2,0),(2,0))}}{4} & {\frac{k^{((2,0),(1,1))}}{2}} & \frac{k^{((2,0),(0,2))}}{4} \\
k^{((1,1),(0,0))} & k^{((1,1),(1,0))} & k^{((1,1),(0,1))} & \frac{k^{((1,1),(2,0))}}{2} & {k^{((1,1),(1,1))}} & \frac{k^{((1,1),(0,2))}}{2} \\
\frac{k^{((0,2),(0,0))}}{2} & \frac{k^{((0,2),(1,0))}}{2} & \frac{k^{((0,2),(0,1))}}{2} & \frac{k^{((0,2),(2,0))}}{4} & {\frac{k^{((0,2),(1,1))}}{2}} & \frac{k^{((0,2),(0,2))}}{4} \\
\end{bmatrix}\in\mathbb{R}^{\PP_{2,2}\times \PP_{2,2}}
\]
for a given ordering of the monomials, and where all the derivatives are taken
at $\vect{x}=0,\vect{y}=0$.

\subsection{Universal (easy) limits}
\label{sec:universal-limits}

The following result applies when the kernel is sufficiently smooth and the L-ensemble has
fixed size $m \in \magicn{d}$.
\begin{theorem}
  \label{thm:nd-smooth-magic-case}
  Let $d\in \mathbb{N}^*$ and $\bL_\varepsilon =
  [\kappa_\varepsilon(\vect{x}_i,\vect{x}_j)]_{i,j}$ for $\kappa$ a stationary
  kernel of smoothness order $r$ and $\vect{x}_1, \ldots, \vect{x}_n$
  vectors in $\R^d$. Then for all $m\in \{\PP_{k,d}\}_{k\leq r-1} \subset\magicn{d}$, the fixed-size L-ensemble $\X_\varepsilon \sim
  |DPP|_m(\bL(\varepsilon))$ has the limiting distribution:
  \[ \Xs \sim |DPP|_m(\bV_{\leq k}\bV_{\leq k}^\top)\]
  Equivalently, if $\bQ$ is an orthonormal basis for $\bV_{\leq k}$, then:
  \[ \Xs \sim |DPP|_m(\bQ\bQ^\top)\]
\end{theorem}
\begin{proof}
  Case 1 of theorem 6.1 in \cite{BarthelmeUsevich:KernelsFlatLimit} states the behavior in $\varepsilon$ of the determinant in this case:
  \begin{align*}
    \forall X \text{ s.t. } | X|=m,\qquad \det(\bL_{\varepsilon,X}) = \varepsilon^{M} \left(\det \bW_{\leq k}(\det\bV_{\leq k}(X))^2 + \O(\varepsilon)\right) 
  \end{align*} 
 for some $M\in\mathbb{N}$ that we do not need to specify in this proof. $\bW_{\leq
    k}$  is the Wronskian matrix. It is irrelevant here as it does
  not depend on $X$. Similarly to the univariate proof (of theorem~\ref{thm:cs-case-fixed-size-1d}), one obtains that the limiting distribution is indeed $\Xs \sim |DPP|_m(\bV_{\leq k}\bV_{\leq k}^\top)$.
  The equivalence between the two formulations of the limiting process comes
  from the fact that $\bV_{\leq k}$ has dimension $n \times m$, and we may
  apply lemma \ref{lem:max-rank-dpp}. Any orthonormal basis will do. 
\end{proof}
\begin{remark}
  Since $\bV_{\leq k}$ is a polynomial basis, $\bQ$ is a basis of orthogonal
  polynomials. The limiting process we see appearing here is the same as the one
  studied in \cite{tremblay2019determinantal} in the discrete case. A similar
  theorem can be proved for continuous DPPs, essentially by tediously changing
  the notation, and leads to the multivariate orthogonal ensembles studied in
  \cite{bardenet2016monte}. What this means is that the properties proved in
   \cite{bardenet2016monte} (good properties for integration)
    and
 \cite{tremblay2019determinantal}  (asymptotic rebalancing) also hold for any
  sufficiently smooth kernel in the flat limit, at least for DPPs of size $m \in \magicn{d}$.
\end{remark}

The case of kernels with finite smoothness is simple if $m$ is greater than 
$\PP_{r-1,d}$.  We then obtain another universal limiting process:
\begin{theorem}
  \label{thm:nd-finite-smooth-magic-case}
  Let $d\in \mathbb{N}^*$ and $\bL_\varepsilon =
  [\kappa_\varepsilon(\vect{x}_i,\vect{x}_j)]_{i,j}$ for $\kappa$ a stationary
  kernel of smoothness order $r$ and $\vect{x}_1, \ldots, \vect{x}_n$
  vectors in $\R^d$. Then, for all $m \geq \PP_{r-1,d}$, the limiting distribution
  of $\Xe \sim |DPP|_m(\bL(\varepsilon))$ is:
  \[ \Xs \sim  \mppDPP{m} \ELE{(-1)^r \bD^{(2r-1)}} {\bV_{\leq r-1}} \]
\end{theorem}
\begin{proof}
	Case 1 of theorem 6.3 in \cite{BarthelmeUsevich:KernelsFlatLimit} states the behavior in $\varepsilon$ of the determinant in this case:
		\begin{align*}
		\forall X \text{ s.t. } | X|=m\geq\PP_{r-1,d},\qquad \det (\matr{L}_{\varepsilon,  X}) =\varepsilon^{M} \left( \widetilde{l}( X) + \O(\varepsilon)\right),
		\end{align*} 
		with $\widetilde{l}(X)$ as in Eq.~\eqref{eq:det_finite_smoothness} (with
    $\bD^{(2r-1)}(X)$, $\bW_{\le r-1}$ and  $\matr{V}_{\le r-1}(X)$ replaced
    by their multivariate equivalent -- see section~\ref{sec:non-universal-limits} to see how this is done), and $M\in\mathbb{N}$ that we do not need to specify in this proof neither. Similarly to the univariate proof (of theorem~\ref{thm:fs-case-fixed-size-1d}), one obtains that the limiting distribution is indeed $\Xs \sim  \mppDPP{m} \ELE{\bD^{(2r-1)}} {\bV_{\leq r-1}}$.
\end{proof}

With these two theorems in hand, we can go back to the teaser (figure
\ref{fig:teaser}) we gave in the introduction. In figure
\ref{fig:teaser}, the points 1 to 6 are on a parabolic curve: $x_2 = x_1^2$,
while point 7 ($x_1 = 0.5,x_2 = 0.6$) is not. For now let $X = \{1,2,3,4,5,6 \}$
and $X' = \{2,3,4,5,6,7 \}$. 
Applying theorem
\ref{thm:nd-smooth-magic-case} for a  $\mDPP{6}$ with a Gaussian kernel, we see
that $p(\Xs = X) \propto \det V_{\leq 2}(X)^2 = 0$
(the matrix is square and has two identical columns). On the other hand, one may check
numerically that $\det V_{\leq 2}(X')$ is non-zero, even though $X'$ is less
spread-out than $X$. 
For the case of the exponential kernel, we apply theorem
\ref{thm:nd-finite-smooth-magic-case}, and we can verify numerically that
$X$ is much more likely than $X'$.
In fact, the two theorems tell us more: the case of the Gaussian kernel holds in
fact for all kernels with $r>1$, which all give zero probability to set $X$. The
more general phenomenon this illustrates is that DPPs defined from smooth
kernels avoid non-unisolvent sets, even though they may be acceptably
spread-out.

\subsection{The general case.}
\label{sec:non-universal-limits}

Up to here, we have covered all the easy cases which lead to universal limits. To be precise, for a fixed $d\in\mathbb{N}^*$ and $r\in\mathbb{N}^*$:
\begin{itemize}
	\item Thm.~\ref{thm:nd-finite-smooth-magic-case} covers the case $m\geq \PP_{r-1,d}$
	\item Out of the remaining cases where $m\leq \PP_{r-1,d}$, Thm.~\ref{thm:nd-smooth-magic-case} covers the special cases where $m\in\magicn{d}$: $m=\PP_{0,d}$, $m=\PP_{1,d}$, $\ldots$, $m=\PP_{r-1,d}$.
\end{itemize}
 What remains is to cover the not-so-easy cases where $m\leq \PP_{r-1,d}$ and $m\notin\magicn{d}$, as provided by:
\begin{theorem}
  \label{thm:general-case-smooth-fixed-size}
  Let $d\in \mathbb{N}^*$ and $\bL_\varepsilon =
  [\kappa_\varepsilon(\vect{x}_i,\vect{x}_j)]_{i,j}$ for $\kappa$ a stationary
  kernel of smoothness order $r$, and $\vect{x}_1, \ldots, \vect{x}_n$
  vectors in $\R^d$. Let $m\leq \PP_{r-1,d}$ and $k\leq r-1$ the integer such
  that $\PP_{k-1,d} < m \le  \PP_{k,d}$. Let us partition the Wronskian
  $\bW_{< k}$ as: 
  \[
  \matr{W}_{< k} =  \begin{bmatrix}\matr{W}_{< k-1} & \matr{W}_{\left\urcorner\right.} \\ \matr{W}_{\llcorner}  &\matr{W}_{\lrcorner}  \end{bmatrix} .
  \]
  Then, the limiting distribution of $\Xe \sim |DPP|_m(\bL_\varepsilon)$ is:
  \[\Xs \sim \mppDPP{m} \ELE{\bV_k\bar{\bW}\bV_k^\top}{\bV_{\leq k-1}}\]
  where $\bar{\bW}\in\mathbb{R}^{\HH_{k,d}\times \HH_{k,d}}$ is the  Schur complement:
  \[ \bar{\bW}= \matr{W}_{\lrcorner} - \matr{W}_{\llcorner}
    (\matr{W}_{< k - 1})^{-1}\matr{W}_{\urcorner} \]
\end{theorem}
\begin{proof}
	Let $X \subset\Omega$ be a subset of size $m$. 
  Case 2 of theorem 6.1 in \cite{BarthelmeUsevich:KernelsFlatLimit} states the behavior in $\varepsilon$ of the determinant in this case:
  \begin{equation}
    \label{eq:limiting-det-general-case}
    \det{\matr{L}_{\varepsilon,X}} = \varepsilon^{2 s(k,d)}  (\det(\matr{Y}\matr{W}_{\leq k} \matr{Y}^{\T}) \det(\matr{V}_{\le k-1}(X)^{\T} \matr{V}_{\le k-1}(X)) + \O(\varepsilon))
  \end{equation}
  with $s(k,d) = d {k+d \choose d+1} - k (\PP_{k,d} - m)$ and $\matr{Y} \in \RR^{m \times \PP_{k,d}}$ defined as:
  \[
    \matr{Y} = 
    \begin{bmatrix}
      \matr{I}_{\PP_{k-1,d}} &  \\ & \Qort(X)^{\T} \matr{V}_k(X)
    \end{bmatrix},
  \]
  $\matr{I}_{\PP_{k-1,d}}$ being the identity matrix of dimension $\PP_{k-1,d}$, $\Qort(X)\in\mathbb{R}^{m\times (m-\PP_{k-1,d})}$ is an orthonormal basis for the space orthogonal to $\text{span } \matr{V}_{\leq k-1}(X)$. 
  
  Expanding the expression:
  \[
    \det(\matr{Y}\matr{W}_{\leq k} \matr{Y}^{\T}) =
    \det
    \begin{pmatrix}
      \matr{W}_{\leq k-1} & \matr{W}_{\urcorner} \bV_k(X)^\top \Qort(X) \\
      \Qort(X)^\top \bV_k(X)  \matr{W}_{\llcorner} & \Qort(X)^\top \bV_k(X)  \matr{W}_{\lrcorner}
      \bV_k(X)^\top \Qort(X) \\
    \end{pmatrix}
  \]
  Applying lemma \ref{lem:block-det}:
  \begin{align}
  \label{eq:def_Y}
    \det(\matr{Y}\matr{W}_{\leq k} \matr{Y}^{\T}) &=
    \det(\matr{W}_{\leq k-1})\det\left( \Qort(X)^\top \bV_k(X) \left( \matr{W}_{\urcorner}
                                                 - \matr{W}_{\llcorner}\matr{W}_{\leq k-1}^{-1}\matr{W}_{\urcorner} \right)\bV_k(X)^\top \Qort(X) \right) \\
                                               &= \det(\matr{W}_{\leq k-1})\det\left( \Qort(X)^\top \bV_k(X) \bar{\bW} \bV_k(X)^\top \Qort(X) \right) \\
  \end{align}
  Injecting into \eqref{eq:limiting-det-general-case} and applying lemma
  \ref{lem:det-saddlepoint}, we obtain:
  \[ \det{\matr{L}_{\varepsilon,\X}} = \varepsilon^{2s(k,d)}
    \left( \det(\matr{W}_{\leq k-1})\det
      \begin{pmatrix}
        \bV_k(X) \bar{\bW} \bV_k(X)^\top & \bV_{\leq k-1}(X) \\
        \bV_{\leq k-1}(X)^\top  & \matr{0}
      \end{pmatrix} + \O(\varepsilon)
\right)
\]
The rest of the proof is identical to the univariate case. 
\end{proof}

\subsection{Numerical illustrations}
\label{sec:numerics-nd}

We show here some numerical results analoguous to those of section~
\ref{sec:numerics-1d}.  In figures \ref{fig:cond-dens-2d-exp} and
\ref{fig:cond-dens-2d-r2}, we show the convergence of conditional densities for
two different kernels.  We illustrate the conditional probabilities of $\vect{x} \cup Y  \large| Y $ where $Y$ comprises seven points already sampled. Even if the ground set is finite and for the sake of illustration,  $\vect{x}$ varies continuously in the unit square. 
Figure \ref{fig:convergence-inclusion-2d} shows the convergence of inclusion probabilities in an example.

\label{sec:conditionals-numerical}
\begin{figure}
  \centering
  \includegraphics[width=\textwidth]{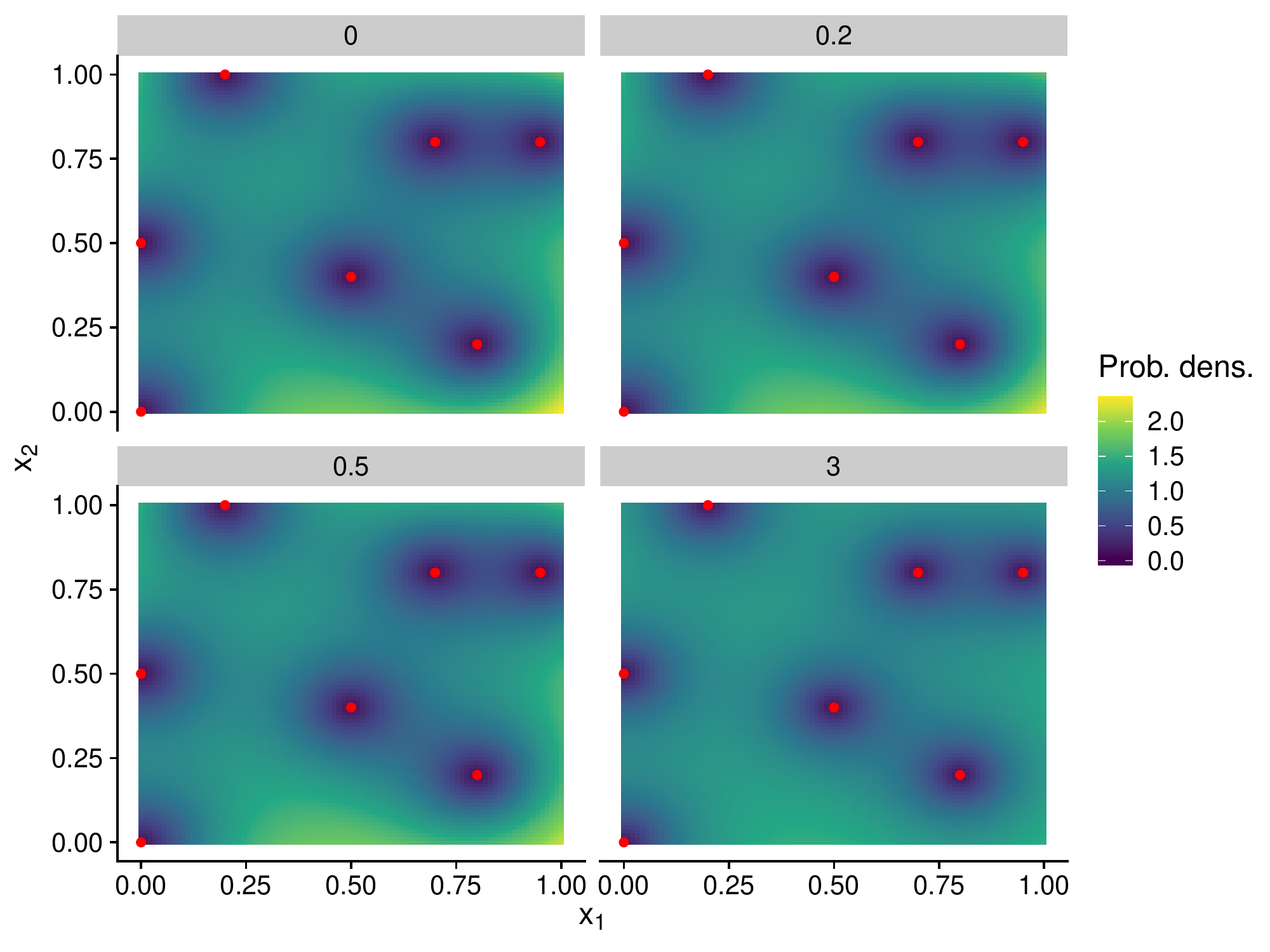}
  \caption{Conditional probability density for $\vect{x}\in [0,1]^2$ conditional on the 7
    nodes in red, for the exponential kernel $\exp(-\norm{\vect{x} -
      \vect{y}})$. The four panels represent the density for different values of
    $\varepsilon$ (panels are labelled with the value). The top-left panel is the theoretical limit.
  }
  \label{fig:cond-dens-2d-exp}
\end{figure}

\begin{figure}
  \centering
  \includegraphics[width=\textwidth]{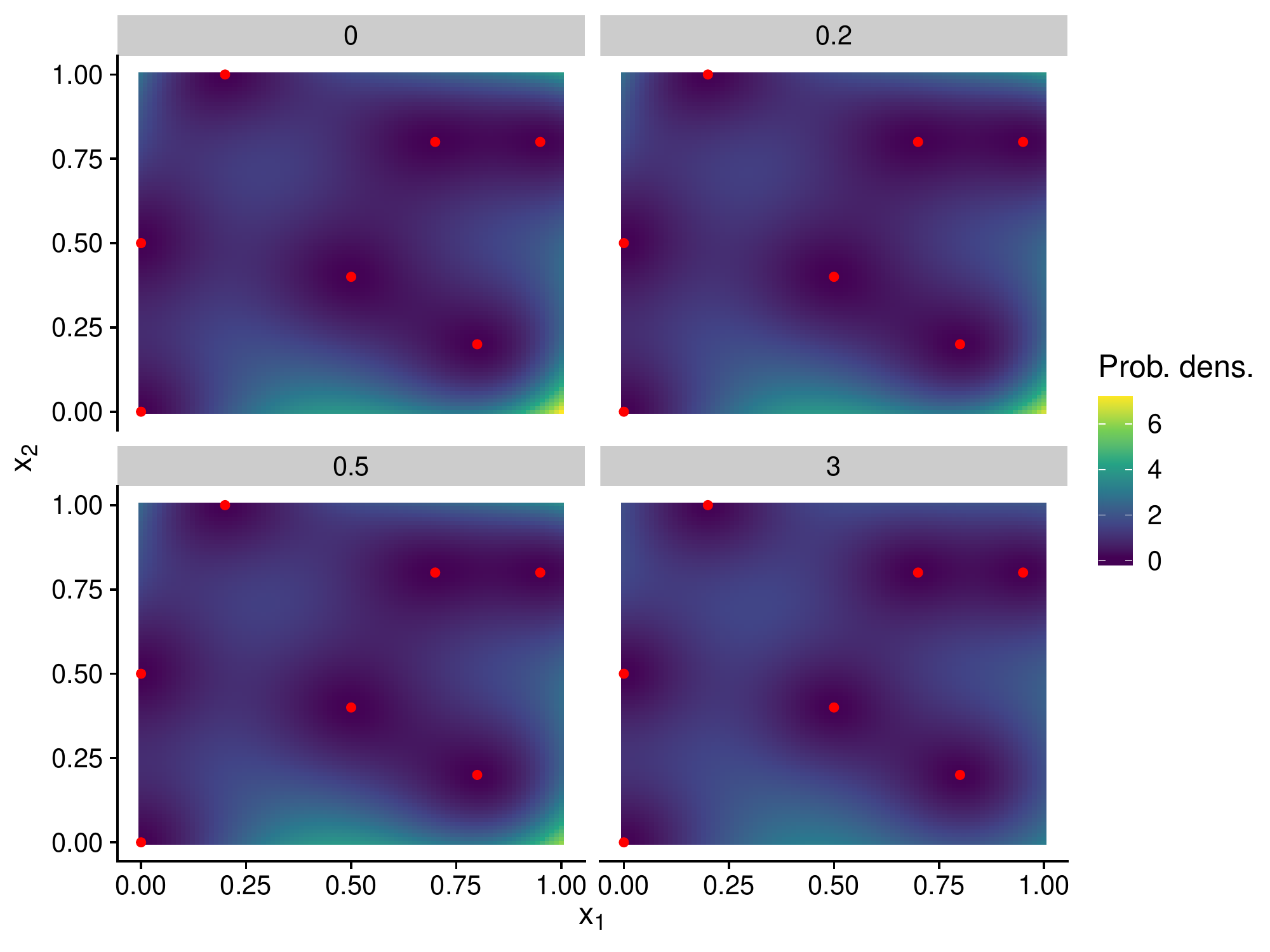}
  \caption{Same as in figure \ref{fig:cond-dens-2d-exp}, but for the kernel $(1+\norm{\vect{x} - \vect{y}})\exp(-\norm{\vect{x} - \vect{y}})$
  }
  \label{fig:cond-dens-2d-r2}
\end{figure}

\begin{figure}

  \begin{center}
    \includegraphics[width=\textwidth]{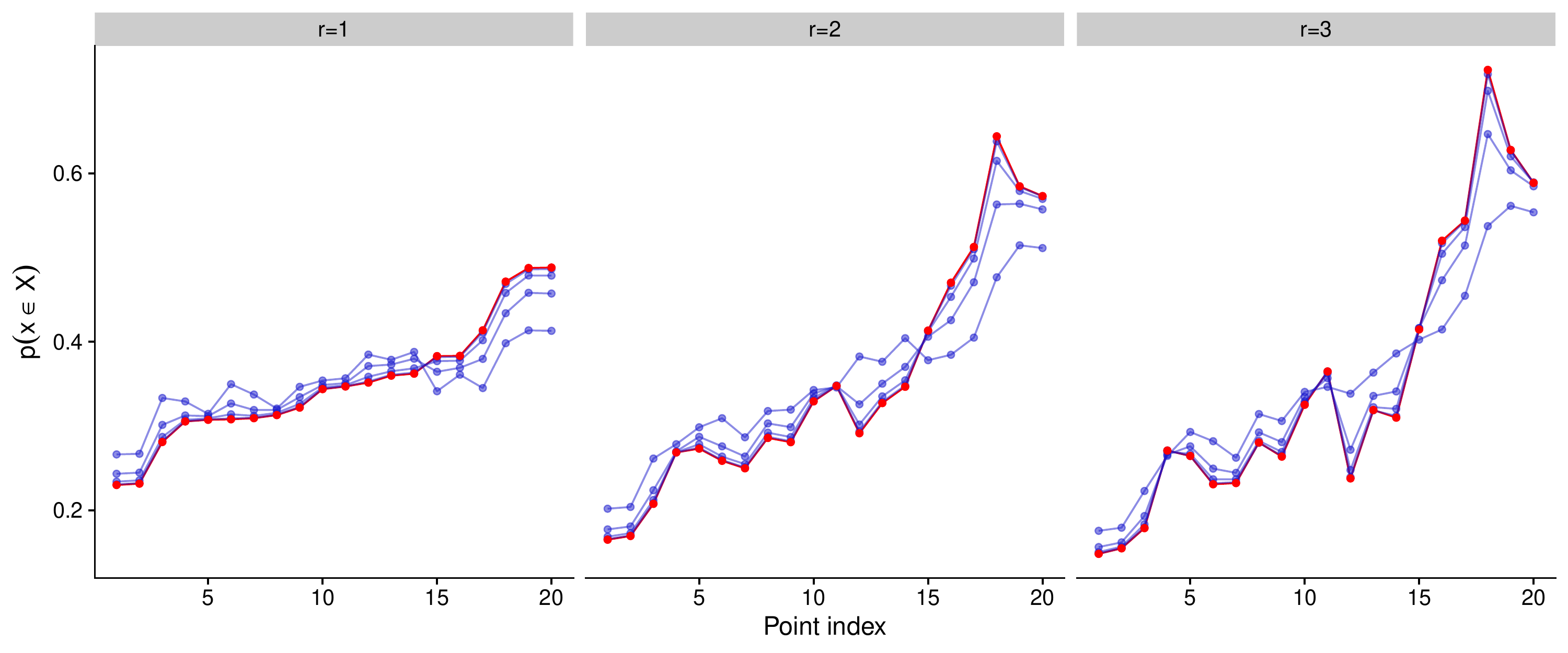}
  \end{center}

  \caption{Flat limit of inclusion probabilities of (fixed-size) L-ensembles for three different
    kernels, multivariate case. Here we plot $\Proba(x \in \Xe)$, the inclusion probabilities for a fixed size L-ensemble 
    (with $m=7$), where the ground set $\Omega$ consists in 20 points drawn at
    random from the unit square. To better visualise the convergence, we plot
    $\Proba(x_i \in \Xe)$ as a function of the index $i$, and we have ordered the
    points according to their inclusion probability for the first kernel.
    Everything else is analoguous to fig. \ref{fig:convergence-inclusion-1d}. 
    The dots in blue (joined by lines for
    clarity) are inclusion probabilities for 
    $\varepsilon = 4,1.5,.5,.1$. The dots in red represent
  the limit in $\flatlim$. The three kernels are, from left-to-right,
$\exp(-\delta), (1+\delta)\exp(-\delta),(3+3\delta+\delta^2)\exp(-\delta)$,
where $\delta = \norm{\vect{x} - \vect{y}}$. These kernels have $r=1$,$2$ and $3$, respectively.}

\label{fig:convergence-inclusion-2d}
\end{figure}

\section{The flat limit of varying-size L-ensembles }
\label{sec:varying-size}

As we saw in section \ref{sec:variable-size-pdpps-as-limits} a  difficulty in studying limits of varying-size L-ensembles 
is the control of the sample size. Using the interesting fact that  L-ensembles are not
invariant to a rescaling of the matrix it is based on, we showed how to control the sample size by
using appropriate scaling functions.
We restrict ourselves to scaling functions that are asymptotically of the form
$\alpha\varepsilon^{-p}$, and we study the limiting process as a function of
$p$ (and $\alpha$, but $p$ plays the more important role).

Studying the flat limit of rescaled L-ensembles reveals an intricate interplay
between the scaling parameter $p$ and the smoothness order $r$ of the kernel. This will
be summarized by  pictures analogous to phase diagrams featuring phase transitions. Once again, we begin the study with the $d=1$ case before delving into the multivariate case.

\subsection{The univariate case}
\label{sec:univariate-variable-size}

In the simple case examined in section
\ref{sec:variable-size-pdpps-as-limits}, we had to rescale 
$\bL(\varepsilon)$ by $\varepsilon^{-1}$ in order to have $\E(|\Xe|) > p$ in the
limit. Here we  generalise the scaling to
$\alpha\varepsilon^{-p}\bL$, and the limiting
size of the L-ensemble will depend on $p$. Interestingly, we will see that in some cases, if $p$ is
odd then the limit is a projection DPP, whereas if $p$ is even the limit is a
partial projection DPP. 
As in the fixed-size case, finitely smooth kernels are indistinguishable from completely smooth
kernels if $|\X|$ is small enough, so that a subtle interplay between $p$ and
$r$ is at work in our result given in theorem \ref{thm:Xe_varying_univariate}. 

This interplay is summed up in figure (\ref{fig:VariableSizeFL}). In a $(p,r)$ plot,
we distinguish three different zones in which the limiting behavior is different. 
If $p\geq 2n-1$ (where we recall that $n$ is the size of the ground set
$\Omega$) or $p > 2r -1$, the limit is trivial since the process converges with
probability 1 to  the ground set. If $ p< 2n-1$ and $p < 2r -1$, the limiting
process depends on the parity of $p$ as announced above. An odd $p$ gives a
fixed-size L-ensemble as a limit, whereas an even $p$ leads to a partial projection DPP
with two possible sample size. Finally, on the line defined by $r=(p+1)/2$ for
$p$ varying from 0 to $2n-1$, the limit process is a partial projection DPP with a wide range of possible sample size, the probability mass of which is explicitely given in Lemma \ref{lem:size_Xe_varying_univariate}. The definition of the limit processes are given in Theorem \ref{thm:Xe_varying_univariate}.

We will prove these results in two steps. The first step is to characterise the
distribution of the size of $|\Xe|$ in the limit. Once we know how $|\Xe|$ is
distributed, we use the fact that the conditional law $\Xe | |\Xe| = m$ is a
fixed size L-ensemble and use the results derived in the previous section to
work out the limit of the point process.

\begin{figure}
\begin{center}
\begin{picture}(0,0)\includegraphics{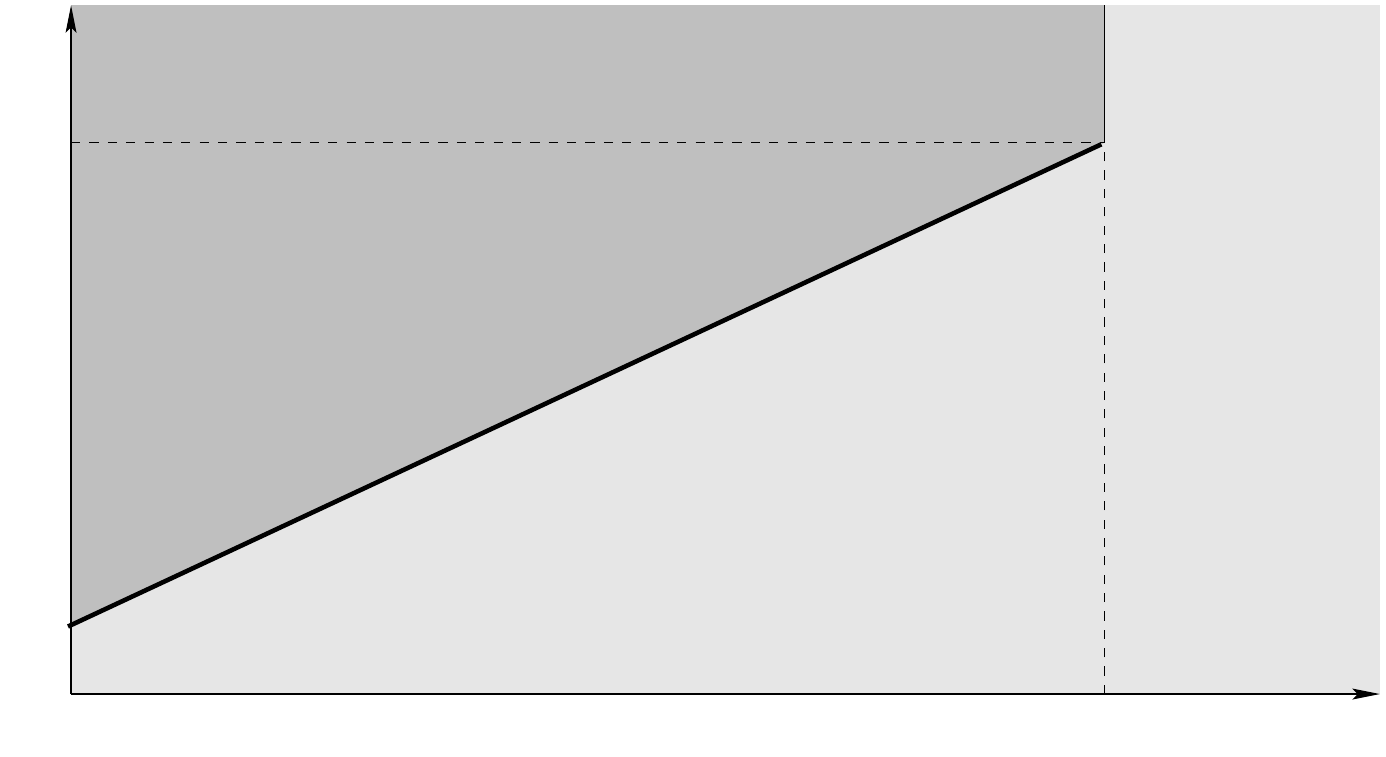}\end{picture}\setlength{\unitlength}{2901sp}\begingroup\makeatletter\ifx\SetFigFont\undefined \gdef\SetFigFont#1#2#3#4#5{\reset@font\fontsize{#1}{#2pt}\fontfamily{#3}\fontseries{#4}\fontshape{#5}\selectfont}\fi\endgroup \begin{picture}(9027,4923)(1336,-4972)
\put(9991,-4876){\makebox(0,0)[lb]{\smash{{\SetFigFont{11}{13.2}{\rmdefault}{\mddefault}{\updefault}{\color[rgb]{0,0,0}$p$}}}}}
\put(1351,-286){\makebox(0,0)[lb]{\smash{{\SetFigFont{11}{13.2}{\rmdefault}{\mddefault}{\updefault}{\color[rgb]{0,0,0}$r$}}}}}
\put(1846,-1231){\makebox(0,0)[lb]{\smash{{\SetFigFont{8}{9.6}{\rmdefault}{\mddefault}{\updefault}{\color[rgb]{0,0,0}$p$ even: ${\cal X}_\star=\mbox{ppDPP}$ and $|{\cal X}_\star |=\frac{p}{2}$ or $\frac{p}{2}+1$ }}}}}
\put(1846,-1816){\makebox(0,0)[lb]{\smash{{\SetFigFont{8}{9.6}{\rmdefault}{\mddefault}{\updefault}{\color[rgb]{0,0,0}$p$ odd : ${\cal X}_\star = |\mbox{DPP}|$ and $|{\cal X}_\star|=\frac{p+1}{2}$}}}}}
\put(1531,-1051){\makebox(0,0)[lb]{\smash{{\SetFigFont{8}{9.6}{\rmdefault}{\mddefault}{\updefault}{\color[rgb]{0,0,0}$n$}}}}}
\put(1576,-4156){\makebox(0,0)[lb]{\smash{{\SetFigFont{8}{9.6}{\rmdefault}{\mddefault}{\updefault}{\color[rgb]{0,0,0}$\frac{1}{2}$}}}}}
\put(8281,-4786){\makebox(0,0)[lb]{\smash{{\SetFigFont{8}{9.6}{\rmdefault}{\mddefault}{\updefault}{\color[rgb]{0,0,0}$2n-1$}}}}}
\put(7606,-3211){\makebox(0,0)[lb]{\smash{{\SetFigFont{8}{9.6}{\rmdefault}{\mddefault}{\updefault}{\color[rgb]{0,0,0}${\cal X}_\star = \Omega$ and $ |{\cal X}_\star | =n$}}}}}
\put(2206,-4201){\rotatebox{24.8}{\makebox(0,0)[lb]{\smash{{\SetFigFont{8}{9.6}{\rmdefault}{\mddefault}{\updefault}{\color[rgb]{0,0,0}$r=\frac{p+1}{2}$ : ${\cal X}_\star = \mbox{ppDPP}$ and $\frac{p+1}{2} \leq |{\cal X}_\star|\ \leq n$}}}}}}
\end{picture} \end{center}
\caption{Phase transition diagram $(p,r$) for the scaling $\varepsilon^p$ in the flat limit of varying size L-ensemble, for kernels with smoothness parameter $r$. In the light gray zone, the process converges to the whole ground set. On the diagonal line ($r=(p+1)/2$), the limit process is a partial projection DPP, with a size distributed over the integers $(p+1)/2$ up to $n$. In the dark grey zone, the limit process depends on the parity of $p$ and is either a partial projection DPP ($p$ even) of a fixed-size L-ensemble  ($p$ odd). The parameters defining the limit process are given in Lemma  \ref{lem:size_Xe_varying_univariate} and Theorem \ref{thm:Xe_varying_univariate}.
}
\label{fig:VariableSizeFL}
\end{figure}
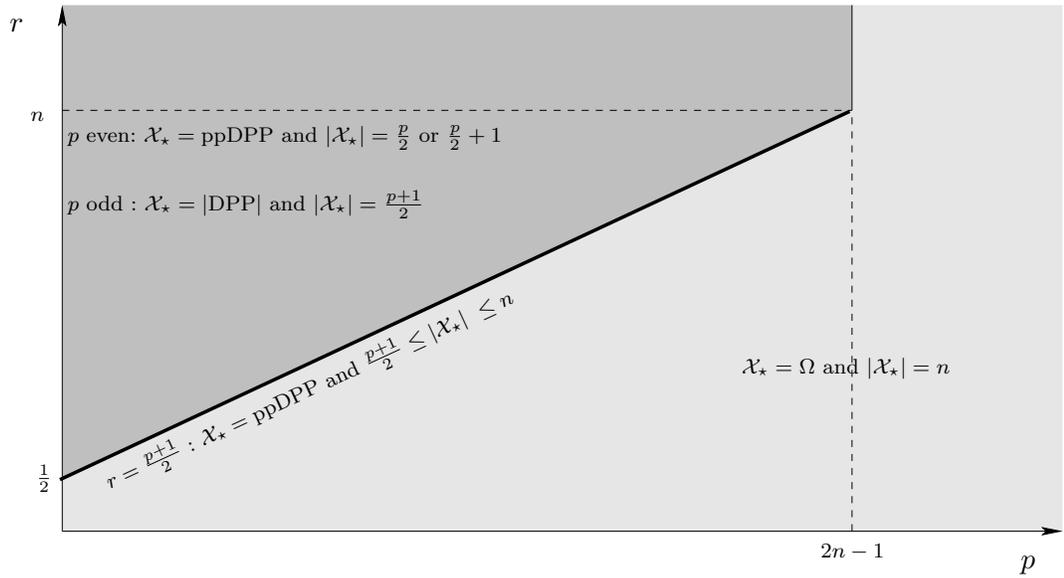

\subsubsection{Distribution of $|\X_\varepsilon|$ in the flat limit}
\label{sec:1d-distr-size}

\begin{lemma}
	\label{lem:size_Xe_varying_univariate}
	Let $p\in\mathbb{N}$, $\alpha>0$, and $\Omega=\{x_1,\ldots,x_n\}$ a set of $n$ distinct points on the real line. Let $\bL_\varepsilon = [\kappa_\varepsilon(x_i,x_j)]_{i,j}$ with $\kappa$ a
	stationary kernel of smoothness order $r\in\mathbb{N}^*$.  Let $\Xe \sim
	DPP(\alpha\varepsilon^{-p}\bL_\varepsilon)$.  In the limit $\flatlim$, the distribution of the size of $\Xe$ depends on the interplay between $p, r$ and $n$. First of all, if $\frac{p+1}{2}\geq n$ then, for any value of $r$, as $\flatlim$, $|\Xe|=n$ with probability one.
	If $\frac{p+1}{2}\leq n$, there are three scenarii depending on the value of $r$: 
		\begin{enumerate}
			\item if $r<\frac{p+1}{2}$, then, as $\flatlim$, $|\Xe|=n$ with probability one.
			\item if $r>\frac{p+1}{2}$, the size of $\Xe$ has a distribution that depends on the parity of $p$:
			\begin{enumerate}
				\item If $p$ is odd, then, as $\flatlim$, $|\Xe|=\frac{p+1}{2}$ with probability one.
				\item If $p$ is even then, as $\flatlim$, and noting $l=p/2$
				\[ |\Xe| =
				\begin{cases}
				l \text{ with probability } \frac{1}{1+\alpha\gamma} \\
				l + 1  \text{ with probability } \frac{\alpha\gamma}{1+\alpha\gamma}
				\end{cases}
				\]
				with \[ \gamma^{-1} = ((\bV_{\leq l}^\top \bV_{\leq
            l})^{-1})_{l+1,l+1}\; ((\bW_{\leq l})^{-1})_{l+1,l+1}\]
			\end{enumerate}
		\item if $r=\frac{p+1}{2}$, then, as $\flatlim$, the distribution tends to:
		\begin{align}
		\label{distrib:size_Xe_r=}
		\Proba(|\Xe| = m) =  \left\{ \begin{array}{ll} 
		0 & \mbox{if }  m< r  \mbox{ or } m>n\\ 
		 \frac{e_{m-r}\left(\alpha f_{2r-1}\widetilde{\bD^{(2r-1)}}\right)}{\det\left(\bI+\alpha f_{2r-1}\widetilde{\bD^{(2r-1)}}\right)} & \mbox{otherwise}
		\end{array} \right. 
		\end{align}
		 where $\widetilde{\bD^{(2r-1)}}=(\bI-\bQ\bQ^\top) \bD^{(2r-1)}(\bI-\bQ\bQ^\top)$, $\bQ$ being an orthonormal basis of ${\rm{span}}(\matr{V}_{\leq r -1})$.
		\end{enumerate}
\end{lemma}

\begin{proof}
In the following, $\bL_{\varepsilon,X}$ stands for the matrix $\bL_{\varepsilon}$ reduced to its lines and columns indexed by $X$. 
First, recall that if $\X \sim DPP(\bL)$, then the marginal distribution of the size $|\X|$ is given by
Eq.~\eqref{eq:distr-size-dpp}:
\begin{equation}
\label{eq:marginal-distr-size}
\Proba(|\X| = m) =  \frac{e_m(\bL)}{e_0(\bL) + e_1(\bL) + \ldots + e_n(\bL)}.
\end{equation}
where $e_m(\bL)$ is the $m$-th elementary symmetric polynomial of $\bL$ and for
consistency $e_0(\bL)=1$ for all matrices $\bL$. Here, we consider the rescaled kernel matrix $\alpha\varepsilon^{-p}\bL_\varepsilon$. 
Recall that $\det(\alpha\varepsilon^{-p} \bL_{\varepsilon,X}) =
\alpha^{|X|}\varepsilon^{-p|X|}\det(\bL_{\varepsilon,X})$. One thus has $\forall i$: $ e_i(\alpha\varepsilon^{-p} \bL_\varepsilon) = \alpha^i \varepsilon^{-ip} e_i(\bL_\varepsilon)$. 

Let $r\in\mathbb{N}^*$. In the flat limit, we can apply theorem \ref{thm:det_1d_smooth} for any set $X$ of size $i\leq r$:
\begin{align*}
\forall i\leq r,\quad e_i(\bL_\varepsilon) &= \sum_{|X|=i} \det \bL_{\varepsilon,X} = \varepsilon^{i(i-1)} \left( \sum_{|X|=i} \det(\matr{V}_{\le i-1}(X))^2 \det(\matr{W}_{\leq i-1}) + \O(\varepsilon) \right)\\
&=  \varepsilon^{i(i-1)} \left( \det(\matr{V}_{\leq i-1}^\top\matr{V}_{\leq i-1})\det( \matr{W}_{\leq i-1}) + \O(\varepsilon)  \right) \end{align*}
where we used Cauchy-Binet to write the second line. Denoting 
\begin{align}
\label{eq:etilde}
\forall i\leq r,\quad \tilde{e}_i=\det(\matr{V}_{\leq i-1}^\top\matr{V}_{\leq i-1})\det( \matr{W}_{\leq i-1})
\end{align}
one has:
\begin{align*}
\forall i\leq r,\quad e_i(\bL_\varepsilon) 
=  \varepsilon^{i(i-1)} \left( \tilde{e}_i + \O(\varepsilon)  \right).
\end{align*}
Also, we can apply theorem \ref{thm:det_1d_finite_smoothness} for any set $X$ of size $i\geq r$:
\begin{align*}
\forall i\geq r,\quad e_i(\bL_\varepsilon) &= \sum_{|X|=i} \det \bL_{\varepsilon,X} = \varepsilon^{i(2r-1)-r^2} \left( \sum_{|X|=i} \tilde{l}(X) + \O(\varepsilon) \right)\\
&=\varepsilon^{i(2r-1)-r^2} \left( \bar{e}_i + \O(\varepsilon) \right)
\end{align*}
where $\bar{e}_i$ verifies:
\begin{align}
\label{eq:ebar}
\forall i> r,\quad \bar{e}_i=\sum_{|X|=i} \tilde{l}(X)=(-1)^r\det(\matr{W}_{r-1,r-1})\sum_{|X|=i}\det\begin{bmatrix}
f_{2r-1}  \bD^{(2r-1)}(X) &  \matr{V}_{\le r-1}(X) \\
\matr{V}_{\le r-1}(X)^{\top} & 0 
\end{bmatrix}
\end{align}
Now, injecting into Eq.~\eqref{eq:marginal-distr-size}, we have:
\begin{align*}
{\rm for }\quad m\leq r: \quad \Proba(|\X| = m) &= \frac{\alpha^m\varepsilon^{-pm} \varepsilon^{m(m-1)}\left(\tilde{e}_m + \O(\varepsilon)  \right) }{1 + \sum_{i=1}^r \alpha^i\varepsilon^{-pi}\varepsilon^{i(i-1)} \left(\tilde{e}_i + \O(\varepsilon)  \right) + \sum_{i=r+1}^n \alpha^i\varepsilon^{-pi}\varepsilon^{i(2r-1)-r^2}\left(\bar{e}_i + \O(\varepsilon)  \right)}\\
{\rm for }\quad m\geq r: \quad \Proba(|\X| = m) &= \frac{\alpha^m\varepsilon^{-pm}\varepsilon^{m(2r-1)-r^2}\left(\bar{e}_m + \O(\varepsilon)  \right) }{1 + \sum_{i=1}^r \alpha^i\varepsilon^{-pi}\varepsilon^{i(i-1)} \left(\tilde{e}_i + \O(\varepsilon)  \right) + \sum_{i=r+1}^n \alpha^i\varepsilon^{-pi}\varepsilon^{i(2r-1)-r^2}\left(\bar{e}_i + \O(\varepsilon)  \right)}
\end{align*}
One can re-write these two equations as:
\begin{align*}
\forall m,\qquad \Proba(|\X| = m) &= \frac{\varepsilon^{\eta(m)}\left(f_0(m) + \O(\varepsilon)  \right) }{\sum_{i=0}^n \varepsilon^{\eta(i)}\left(f_0(i) + \O(\varepsilon)  \right) }
\end{align*}
where $\eta(\cdot)$ and $f_0(\cdot)$ are two $\varepsilon$-independent functions verifying:

\begin{equation}
  \label{eq:d_of_i_vsize_proof}
\eta(i) = \left\{ \begin{array}{ll} 
\eta_1(i)=i(i-p-1) & \mbox{if }  i\leq r \\ 
\eta_2(i)=i(2r-p-1)-r^2 & \mbox{if } i\geq r
\end{array} \right. 
\end{equation}
and
\begin{align*}
f_0(i) = \left\{ \begin{array}{ll} 
\alpha^i \tilde{e}_i & \mbox{if }  i\leq r \\ 
\alpha^i \bar{e}_i & \mbox{if } i\geq r.
\end{array} \right. 
\end{align*}
Note that we are precisely in the context of
lemma~\ref{lem:TV-convergence-diverging}, that we now apply. The question is:
what is $\argmin_{i\in[0,n]} \eta(i)$? The answer to this question depends on $p,
r$ and $n$ which explains the different cases of the theorem. Let us make first
two simple observations on $\eta_1$ and $\eta_2$ (refer to figure \ref{fig:PlotProofVariableSize} for an illustration)
\begin{itemize}
	\item $\eta_1$ is a second order polynomial, it is equal to $0$ at $i=0$ and then decreases until $i=\frac{p+1}{2}$, where it reaches its minimum and then increases again. 
	\item $\eta_2$ is a linear function with slope $2r-p-1$. The sign of that slope is equal to the sign of $r-\frac{p+1}{2}$
\end{itemize}	
We shall now explore all the cases of the theorem sequentially.

\begin{figure}
\begin{center}
\begin{picture}(0,0)\includegraphics{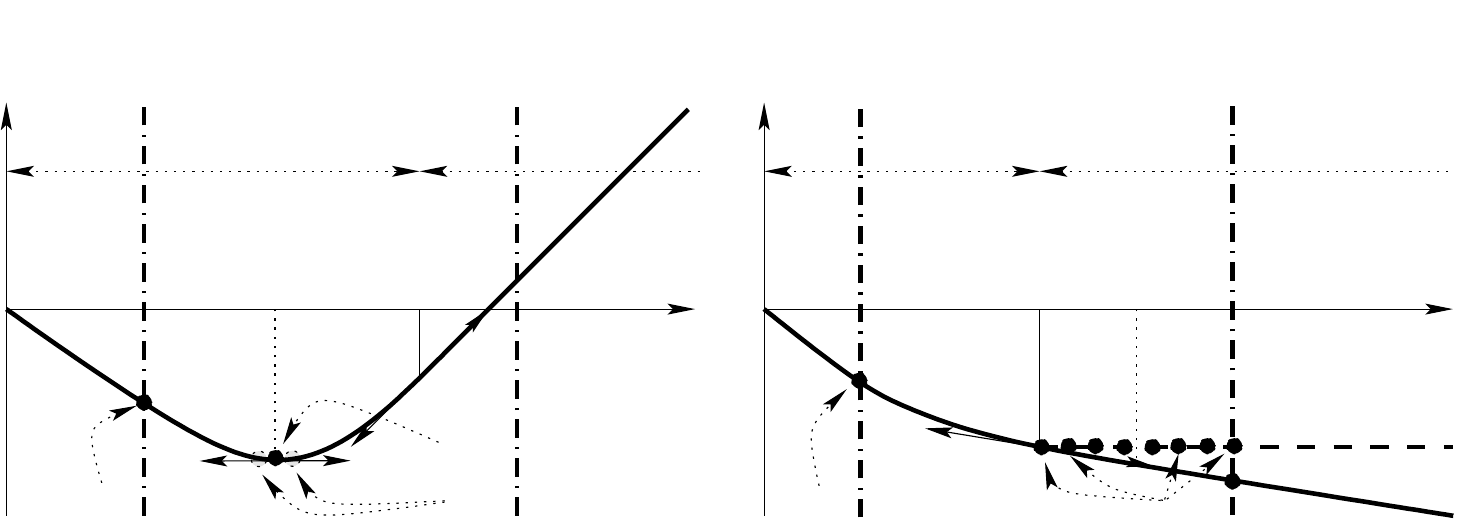}\end{picture}\setlength{\unitlength}{2901sp}\begingroup\makeatletter\ifx\SetFigFont\undefined \gdef\SetFigFont#1#2#3#4#5{\reset@font\fontsize{#1}{#2pt}\fontfamily{#3}\fontseries{#4}\fontshape{#5}\selectfont}\fi\endgroup \begin{picture}(9525,3413)(1309,-4167)
\put(9862,-1803){\makebox(0,0)[lb]{\smash{{\SetFigFont{8}{9.6}{\rmdefault}{\mddefault}{\updefault}{\color[rgb]{0,0,0}$\eta_2$}}}}}
\put(5626,-2986){\makebox(0,0)[lb]{\smash{{\SetFigFont{8}{9.6}{\rmdefault}{\mddefault}{\updefault}{\color[rgb]{0,0,0}$m$}}}}}
\put(4006,-2626){\makebox(0,0)[lb]{\smash{{\SetFigFont{8}{9.6}{\rmdefault}{\mddefault}{\updefault}{\color[rgb]{0,0,0}$r$}}}}}
\put(10486,-2986){\makebox(0,0)[lb]{\smash{{\SetFigFont{8}{9.6}{\rmdefault}{\mddefault}{\updefault}{\color[rgb]{0,0,0}$m$}}}}}
\put(8056,-2626){\makebox(0,0)[lb]{\smash{{\SetFigFont{8}{9.6}{\rmdefault}{\mddefault}{\updefault}{\color[rgb]{0,0,0}$r$}}}}}
\put(8506,-2626){\makebox(0,0)[lb]{\smash{{\SetFigFont{8}{9.6}{\rmdefault}{\mddefault}{\updefault}{\color[rgb]{0,0,0}$\frac{p+1}{2}$}}}}}
\put(2701,-961){\makebox(0,0)[lb]{\smash{{\SetFigFont{10}{12.0}{\rmdefault}{\mddefault}{\updefault}{\color[rgb]{0,0,0}Case $r>\frac{p+1}{2}$}}}}}
\put(7651,-961){\makebox(0,0)[lb]{\smash{{\SetFigFont{10}{12.0}{\rmdefault}{\mddefault}{\updefault}{\color[rgb]{0,0,0}Case $r\leq\frac{p+1}{2}$}}}}}
\put(2930,-2613){\makebox(0,0)[lb]{\smash{{\SetFigFont{8}{9.6}{\rmdefault}{\mddefault}{\updefault}{\color[rgb]{0,0,0}$\frac{p+1}{2}$}}}}}
\put(9822,-3526){\makebox(0,0)[lb]{\smash{{\SetFigFont{8}{9.6}{\rmdefault}{\mddefault}{\updefault}{\color[rgb]{0,0,0}$r=\frac{p+1}{2}$}}}}}
\put(2251,-1411){\makebox(0,0)[lb]{\smash{{\SetFigFont{10}{12.0}{\rmdefault}{\mddefault}{\updefault}{\color[rgb]{0,0,0}$n<\frac{p+1}{2}$}}}}}
\put(6931,-1411){\makebox(0,0)[lb]{\smash{{\SetFigFont{10}{12.0}{\rmdefault}{\mddefault}{\updefault}{\color[rgb]{0,0,0}$n\leq r$}}}}}
\put(9351,-1411){\makebox(0,0)[lb]{\smash{{\SetFigFont{10}{12.0}{\rmdefault}{\mddefault}{\updefault}{\color[rgb]{0,0,0}$n>r$}}}}}
\put(6533,-4094){\makebox(0,0)[lb]{\smash{{\SetFigFont{10}{12.0}{\rmdefault}{\mddefault}{\updefault}{\color[rgb]{0,0,0}min}}}}}
\put(8961,-4087){\makebox(0,0)[lb]{\smash{{\SetFigFont{10}{12.0}{\rmdefault}{\mddefault}{\updefault}{\color[rgb]{0,0,0}min}}}}}
\put(1870,-4094){\makebox(0,0)[lb]{\smash{{\SetFigFont{10}{12.0}{\rmdefault}{\mddefault}{\updefault}{\color[rgb]{0,0,0}min}}}}}
\put(4685,-1411){\makebox(0,0)[lb]{\smash{{\SetFigFont{10}{12.0}{\rmdefault}{\mddefault}{\updefault}{\color[rgb]{0,0,0}$n\geq\frac{p+1}{2}$}}}}}
\put(4311,-4080){\makebox(0,0)[lb]{\smash{{\SetFigFont{10}{12.0}{\rmdefault}{\mddefault}{\updefault}{\color[rgb]{0,0,0}min $p$ even}}}}}
\put(4312,-3739){\makebox(0,0)[lb]{\smash{{\SetFigFont{10}{12.0}{\rmdefault}{\mddefault}{\updefault}{\color[rgb]{0,0,0}min $p$ odd}}}}}
\put(1351,-1366){\makebox(0,0)[lb]{\smash{{\SetFigFont{8}{9.6}{\rmdefault}{\mddefault}{\updefault}{\color[rgb]{0,0,0}$\eta(m)$}}}}}
\put(2521,-1816){\makebox(0,0)[lb]{\smash{{\SetFigFont{8}{9.6}{\rmdefault}{\mddefault}{\updefault}{\color[rgb]{0,0,0}$\eta_1$}}}}}
\put(4816,-1816){\makebox(0,0)[lb]{\smash{{\SetFigFont{8}{9.6}{\rmdefault}{\mddefault}{\updefault}{\color[rgb]{0,0,0}$\eta_2$}}}}}
\put(6301,-1411){\makebox(0,0)[lb]{\smash{{\SetFigFont{8}{9.6}{\rmdefault}{\mddefault}{\updefault}{\color[rgb]{0,0,0}$\eta(m)$}}}}}
\put(7111,-1816){\makebox(0,0)[lb]{\smash{{\SetFigFont{8}{9.6}{\rmdefault}{\mddefault}{\updefault}{\color[rgb]{0,0,0}$\eta_1$}}}}}
\end{picture} \end{center}
\caption{The behavior of $|{\cal X}_\star|$ is governed by the argument minimizing the function $\eta(m)$ defined in the proof of Lemma \ref{lem:size_Xe_varying_univariate}. The curve of $\eta(m)$ is made of a parabola $\eta_1$ up to $m=r$ and then from a line $\eta_2$ with slope $\eta'_1(r)$ (extending integers to reals obviously!). Three cases appear depending of the relative position $(p+1)/2$ of the minimum of $\eta_1$ with respect to $r$. To study the behavior of   $|{\cal X}_\star| \leq n =|\Omega|$, we then have to locate $n$.  In the left plot, we observe that the minimum is for $m=n$ if $n<(p+1)/2$, whereas it is $m=(p+1)/2$ if $n\geq (p+1)/2$. Note in this situation that we have either one minimum if $p$ is odd, or two is $p$ is even, since we worl with integers. In the right plot, in the case $r<(p+1)/2$ (thick line),  whatever $n$ the minimum is attained at $m=n$ since $\eta$ strictly decreases. If  $r=(p+1)/2$ however (thick dashed horizontal line), the minimum is attained for the range $[(p+1)/2; n]$ if $n>(p+1)/2$, otherwise at $n$.}
\label{fig:PlotProofVariableSize}
\end{figure}
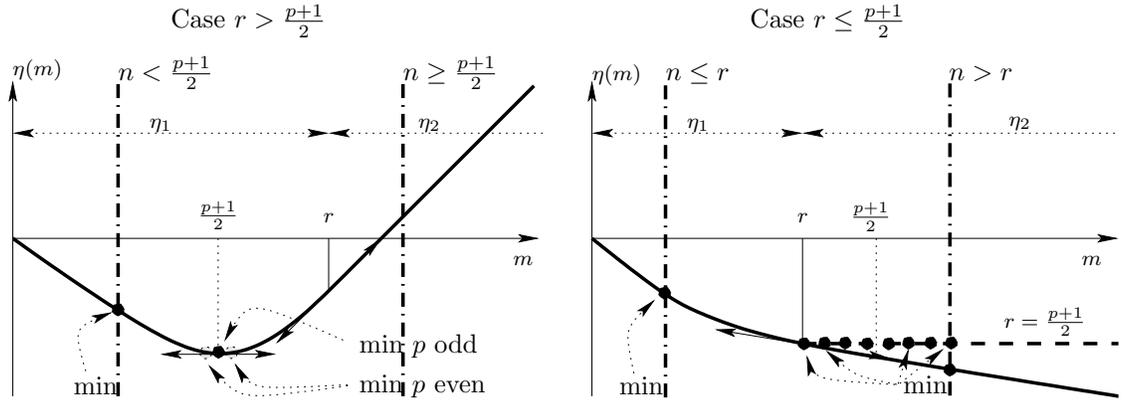

First of all, if $\frac{p+1}{2}\geq n$, there are two cases: either $r\geq n$, in which case $\eta=\eta_1$ for the whole interval $[0,n]$ and $\argmin{\eta}=n$; or $r\leq n$ in which case $\eta$ decreases up to $i=r$ and then continues to decrease (as the slope of $\eta_2$ is negative) up to $i=n$, implying $\argmin{\eta}=n$. Thus, whatever the value of $r$, $\argmin{\eta}=n$. Applying lemma~\ref{lem:TV-convergence-diverging}, for all values of $r$, as $\flatlim$, $|\Xe|=n$ with probability $1$. 

Now, if $\frac{p+1}{2}\leq n$, there are three scenarii depending on the value of $r$:
\begin{enumerate}
	\item if $r<\frac{p+1}{2}$, $d$ decreases up to $i=r$ and then continues to decrease (as the slope of $\eta_2$ is negative) up to $i=n$, implying $\argmin{d}=n$. Applying lemma~\ref{lem:TV-convergence-diverging}, as $\flatlim$, $|\Xe|=n$ with probability $1$.
	\item if $r>\frac{p+1}{2}$, $\eta$ decreases up to $\frac{p+1}{2}$, then increases up to $i=r$, and then continues to increase (as the slope of $\eta_2$ is now positive) up to $i=n$, implying $\argmin{\eta}=\frac{p+1}{2}$. Now,
	\begin{enumerate}
		\item If $p$ is odd, then $\frac{p+1}{2}$ is an integer. Applying lemma~\ref{lem:TV-convergence-diverging}, $|\Xe|=\frac{p+1}{2}$ with probability $1$, as $\flatlim$.
		\item   If $p$ is even, $\frac{p+1}{2}$ is not an integer. In that case $\argmin{\eta}$ has two integer solutions: $\frac{p}{2}$ and $\frac{p}{2}
		+ 1$. Let $l=p/2$. According to lemma \ref{lem:TV-convergence-diverging}, as $\flatlim$, $\Xe$ will be of size
		either $l$ or $l+1$ with probabilities given by
		\[ \Proba(|\Xs| = l) = \frac{\tilde{e}_{l}}{\tilde{e}_{l} +
			\alpha \tilde{e}_{l+1}}\quad \text{  and  }\quad \Proba(|\Xs| = l +1) =
		1 - \Proba(|\Xs| = l).\] 
		Injecting Eq.~\ref{eq:etilde} and simplifying, we
		find:
		\[ \Proba(|\Xs| = l) = \frac{1}{1+\alpha \gamma} \]
		with \begin{align}
		\label{eq:gamma_orig}
		\gamma &= \frac{\det (\bV_{\leq l}^\top \bV_{\leq
				l})\det \bW_{\leq l}}{\det (\bV_{\leq l-1}^\top \bV_{\leq
				l-1})\det \bW_{\leq l-1}}\\
           &= \frac{1}{((\bV_{\leq l}^\top \bV_{\leq l})^{-1})_{l+1,l+1} ((\bW_{\leq l})^{-1})_{l+1,l+1}}
		\end{align}
		where the last equality follows from Cramer's rule. Notice that $\gamma$
		depends on the Wronskian of the kernel and not just its order of regularity. 
	\end{enumerate}
 \item the last scenario, $r=\frac{p+1}{2}$, is the most involved. Indeed, in this case, the function $\eta$ decreases up to $i=r$ and then stays constant between $i=r$ and $i=n$ (as the slope of $\eta_2$ is null). Thus, $\argmin \eta$ has $n-r+1$ integer solutions: all the integers between $r$ and $n$. According to lemma \ref{lem:TV-convergence-diverging}, as $\flatlim$, the limiting distribution of $|\Xe|$ is:
 $$\forall m \mbox{  s.t.  } r\leq m\leq n,\quad \Proba(|\Xs| = m) = \frac{\alpha^m\bar{e}_m}{\sum_{i=r}^n \alpha^i\bar{e}_i}$$Now, consider the NNP $(f_{2r-1}\bD^{(2r-1)}; \matr{V}_{\leq r-1})$ as well as $f_{2r-1}\widetilde{\bD^{(2r-1)}}$ as defined in Definition~\ref{def:ext_Lens}. Note that the rank of $f_{2r-1}\widetilde{\bD^{(2r-1)}}$ is $n-r$. One may apply Corollary~\ref{cor:normalisation-ppDPP} and obtain, for all integer $i$ such that $r\leq i \leq n$:
 \begin{align*}
\sum_{|X|=i}\det\begin{bmatrix}
f_{2r-1}  \bD^{(2r-1)}(X) &  \matr{V}_{\le r-1}(X) \\
\matr{V}_{\le r-1}^{\T}(X) & 0 
\end{bmatrix} = (-1)^r e_{i-r}\left(f_{2r-1}\widetilde{\bD^{(2r-1)}}\right) \det \left(\left(\matr{V}_{\leq r-1}\right)^\top \matr{V}_{\leq r-1}\right)
\end{align*}
Injecting this in Eq.~\eqref{eq:ebar} and simplifying, one re-writes the limiting distribution of $|\Xe|$ as:
\begin{align*}
\forall m \mbox{  s.t.  } r\leq m\leq n,\quad \Proba(|\Xe| = m) = \frac{ e_{m-r}\left(\alpha f_{2r-1}\widetilde{\bD^{(2r-1)}}\right)}{\sum_{i=r}^n e_{i-r}\left(\alpha f_{2r-1}\widetilde{\bD^{(2r-1)}}\right)}.
\end{align*}
Noting that $\sum_{i=r}^n e_{i-r}\left(\alpha f_{2r-1}\widetilde{\bD^{(2r-1)}}\right)=\sum_{i=0}^{n} e_{i}\left(\alpha f_{2r-1}\widetilde{\bD^{(2r-1)}}\right)=\det\left(\bI + \alpha f_{2r-1}\widetilde{\bD^{(2r-1)}}\right)$ finishes the proof.
\end{enumerate}
\end{proof}

\subsubsection{Distribution of $\Xe$ in the flat limit}
\label{sec:1d-distr-joint}
Now that we have characterised the distribution of $|\Xe|$, we can prove the following:
\begin{theorem}
\label{thm:Xe_varying_univariate}
	Let $p\in\mathbb{N}$, $\alpha>0$, and $\Omega=\{x_1,\ldots,x_n\}$ a set of $n$ distinct points on the real line. Let $\bL_\varepsilon = [\kappa_\varepsilon(x_i,x_j)]_{i,j}$ with $\kappa$ a
	stationary kernel of smoothness order $r\in\mathbb{N}^*$.  Let $\Xe \sim
	DPP(\alpha\varepsilon^{-p}\bL_\varepsilon)$.  In the limit $\flatlim$, the distribution of $\Xe$ depends on the interplay between $p, r$ and $n$. First of all, if $\frac{p+1}{2}\geq n$ then, for any value of $r$, $\Xe$ has limit $\Xs=\Omega$ with probability one.
	If $\frac{p+1}{2}\leq n$, there are three scenarii depending on the value of $r$: 
	\begin{enumerate}
		\item if $r<\frac{p+1}{2}$, then $\Xe$ has limit $\Xs=\Omega$ with probability one.
		\item if $r>\frac{p+1}{2}$, $\Xe$ has a limiting distribution that depends on the parity of $p$:
		\begin{enumerate}
			\item If $p$ is odd, then $\Xe$ has limit $\Xs \sim \mDPP{l}(\bV_{\leq l-1}\bV_{\leq l-1}^\top)$ with $l=\frac{p+1}{2}$
			\item If $p$ is even then $\Xe$ has limit $\Xs \sim \ppDPP \ELE{\frac{\alpha}{\tilde{w}}  \vect{v}_{l+1}\vect{v}_{l+1}^\top}{\bV_{\leq
					l-1}}$ with $l=\frac{p}{2}$ and $\tilde{w}=((\bW_{\leq l})^{-1})_{l+1,l+1}$   
		\end{enumerate}
		\item if $r=\frac{p+1}{2}$, then $\Xe$ has limit $\Xs \sim DPP \ELE{\alpha f_{2r-1}
			\bD^{(2r-1)}}{\bV_{\leq  r-1}}$.
	\end{enumerate}
\end{theorem}

\begin{proof}
We will prove each case sequentially. First of all, for all the cases in Lemma ~\ref{lem:size_Xe_varying_univariate} for which $|\Xe|=n$ in the limit $\flatlim$, the set $\Xe$ obviously tends to $\Omega$. Let us now focus on scenario number 2. 

In the case 2a, we know from Lemma ~\ref{lem:size_Xe_varying_univariate} that $|\Xs|=\frac{p+1}{2}$ with probability one. The limiting process is thus a fixed-size L-ensemble of size $l=\frac{p+1}{2}$. The fixed-size limit applies and theorem~\ref{thm:cs-case-fixed-size-1d} implies the result.

Case 2b needs a bit more work.  First of all, define the integer $l=\frac{p}{2}$ and consider an orthonormal basis $\bQ\in\mathbb{R}^{n\times l}$ of $\text{span}(\bV_{\leq l-1})$. Also, consider the vector $\vect{q}_{l+1}$ such that $\bQ'=\left[\bQ | \vect{q}_{l+1}\right]\in\mathbb{R}^{n\times (l+1)}$ is an orthonormal basis for ${\rm span}(\bV_{\leq l})$. 
From Lemma ~\ref{lem:size_Xe_varying_univariate} and Theorem ~\ref{thm:cs-case-fixed-size-1d}, we know
that in the limit $\flatlim$,  $\Xe$ is a mixture of two fixed-size L-ensembles (and hence a partial projection DPP): with probability
$\frac{1}{1+\alpha \gamma}$, it has size $l$ and distribution $\mDPP{l}(\bV_{\leq l-1}\bV_{\leq l-1}^\top)$. With probability $\frac{\alpha\gamma}{1+\alpha\gamma}$,
it has size $l+1$ and distribution $\mDPP{l+1}(\bV_{\leq l}\bV_{\leq l}^\top)$. Note that 
by lemma \ref{lem:marginal-kernel-proj}, these distributions are equivalent to
$\mDPP{l}(\bQ\bQ^\top)$ and $\mDPP{l+1}(\bQ'\bQ'^\top)$ respectively. Looking at the mixture representation of pp-DPPs described in Corollary~\ref{cor:mixture-rep-ppDPP-varying}, one observes that this limiting distribution can be succintly described as a pp-DPP $\Xs \sim DPP \ELE{\alpha\gamma\bQ'\bQ'^\top}{\bQ}$. 
Now, by the invariance property of remark~\ref{rem:second_inv}, this is equivalent to $\Xs  \sim \ppDPP \ELE{\alpha\gamma\vect{q}_{l+1}\vect{q}_{l+1}^\top}{\bQ}$. Also, by the invariance
property of remark~\ref{rem:first_inv}, this is in turn equivalent to $\Xs  \sim \ppDPP \ELE{\alpha\gamma\vect{q}_{l+1}\vect{q}_{l+1}^\top}{\bV_{\leq l-1}}$.
Finally, noting that 
\begin{align*}
\frac{\det (\bV_{\leq l}^\top \bV_{\leq l})}{\det (\bV_{\leq l-1}^\top \bV_{\leq l-1})}	\vect{q}_{l+1}\vect{q}_{l+1}^\top= 	\vect{v}_{l+1}\vect{v}_{l+1}^\top
\end{align*} 
and injecting in the expression of $\gamma$ of Eq.~\ref{eq:gamma_orig}, one obtains that
$
\gamma
\vect{q}_{l+1}\vect{q}_{l+1}^\top  
=\tilde{w}^{-1} \; \vect{v}_{l+1}\vect{v}_{l+1}^\top
$, 
finishing the proof that the limit in case 2b is $\Xs \sim \ppDPP \ELE{\frac{\alpha}{\tilde{w}}  \vect{v}_{l+1}\vect{v}_{l+1}^\top}{\bV_{\leq l-1}}$.

Let us finish with case 3. From a mixture point of view, the limiting process can be described by:
\begin{enumerate}
	\item draw the size $m$ of the set according to Eq.~\eqref{distrib:size_Xe_r=} of Lemma~\ref{lem:size_Xe_varying_univariate}:
	\begin{align}
	\Proba(|\Xe| = m) =  \left\{ \begin{array}{ll} 
	0 & \mbox{if }  m< r \\ 
	\frac{e_{m-r}\left(\alpha f_{2r-1}\widetilde{\bD^{(2r-1)}}\right)}{\det\left(\bI+\alpha f_{2r-1}\widetilde{\bD^{(2r-1)}}\right)} & \mbox{if } m\geq r
	\end{array} \right.
	\end{align}
	where $\widetilde{\bD^{(2r-1)}}=(\bI-\bQ\bQ^\top) \bD^{(2r-1)}(\bI-\bQ\bQ^\top)$, $\bQ$ being an orthonormal basis of ${\rm{span}}(\matr{V}_{\leq r -1})$.
	\item conditionally on the size, draw a fixed-size pp-DPP, which, according to theorem~\ref{thm:fs-case-fixed-size-1d}, reads
 $\Xs \sim \mppDPP{m} \ELE{\bD^{(2r-1)}} {\bV_{\leq r-1}}$. 
\end{enumerate}
Noting that  $\Xs \sim \mppDPP{m} \ELE{\bD^{(2r-1)}} {\bV_{\leq r-1}}$ is equivalent to $\Xs \sim \mppDPP{m} \ELE{\alpha f_{2r-1}\bD^{(2r-1)}} {\bV_{\leq r-1}}$, this mixture is precisely the mixture representation of $\Xs \sim \ppDPP \ELE{\alpha f_{2r-1}\bD^{(2r-1)}} {\bV_{\leq r-1}}$, ending the proof.
\end{proof}

\subsection{The multivariate case}
\label{sec:multivar-variable-size}

The multivariate case is a mostly straightforward generalisation of the
univariate case. The size of $\Xe$ is described in the following lemma, which
generalises lemma \ref{lem:size_Xe_varying_univariate}

\begin{lemma}
	\label{lem:size_Xe_varying_multivariate}
	Let $p\in\mathbb{N}$, $\alpha>0$, and $\Omega=\{\vect{x}_1,\ldots,\vect{x}_n\}$ a set of $n$
  distinct points in $\R^d$. Let $\bL_\varepsilon = [\kappa_\varepsilon(\vect{x}_i,\vect{x}_j)]_{i,j}$ with $\kappa$ a
	stationary kernel of smoothness order $r\in\mathbb{N}^*$.  Let $\Xe \sim
	DPP(\alpha\varepsilon^{-p}\bL_\varepsilon)$.  In the limit $\flatlim$, the distribution of the size of $\Xe$ depends on the interplay between $p, r$ and $n$. First of all, $p$ is either even or odd: only one out of the two following values $\left(\frac{p}{2},\frac{p+1}{2}\right)$ is an integer. We call that integer $l$. Now, if $\PP_{l-1,d}\geq n$  then, for any value of $r$, as $\flatlim$, $|\Xe|=n$ with probability one. 
	Otherwise, there are three scenarii depending on the value of $r$: 
		\begin{enumerate}
			\item if $r<\frac{p+1}{2}$, then, as $\flatlim$, $|\Xe|=n$ with probability one.
			\item if $r>\frac{p+1}{2}$, the size of $\Xe$ has a distribution that depends on the parity of $p$:
			\begin{enumerate}
				\item If $p$ is odd ($l=\frac{p+1}{2}$), then, as $\flatlim$, $|\Xe|=\PP_{l-1,d}$ with probability one. 
				\item If $p$ is even ($l=\frac{p}{2}$) then, as $\flatlim$, the distribution tends to:
				\begin{align*}
				\Proba(|\Xs| = m) =  \left\{ \begin{array}{ll} 
				0 & \text{ if } m < \PP_{l-1,d} \text{ or } m > \PP_{l,d}\\ 
				\frac{e_{m -
						\PP_{l-1,d}}(\alpha \widetilde{\bV_{l}\bar{\bW}\bV_{l}^\top})}{\det\left(\matr{I}+\alpha\widetilde{\bV_l \bar{\bW} \bV_l^\top}\right)} & otherwise\end{array} \right. 
				\end{align*}
        where $\bV_{l}\bar{\bW}\bV_{l}^\top$ is as in theorem \ref{thm:general-case-smooth-fixed-size}, and $\widetilde{\bV_{l}\bar{\bW}\bV_{l}^\top}=(\bI-\bQ_l\bQ_l^\top) \bV_{l}\bar{\bW}\bV_{l}^\top(\bI-\bQ_l\bQ_l^\top)$, $\bQ_l$ being an orthonormal basis of ${\rm{span}}(\matr{V}_{\leq l-1})$.
			\end{enumerate}
		\item if $r=\frac{p+1}{2}$, then, as $\flatlim$, the distribution tends to:
		\begin{align}
		\label{distrib:size_Xe_r=_multivar}
		\Proba(|\Xs| = m) =  \left\{ \begin{array}{ll} 
		0 & \mbox{if }  m< \PP_{r-1,d} \mbox{ or } m>n\\ 
		 \frac{e_{m-\PP_{r-1,d}}\left(\alpha f_{2r-1}\widetilde{\bD^{(2r-1)}}\right)}{\det\left(\bI+\alpha f_{2r-1}\widetilde{\bD^{(2r-1)}}\right)} & \mbox{otherwise}
		\end{array} \right. 
		\end{align}
		 where $\widetilde{\bD^{(2r-1)}}=(\bI-\bQ_r\bQ_r^\top) \bD^{(2r-1)}(\bI-\bQ_r\bQ_r^\top)$, $\bQ_r$ being an orthonormal basis of ${\rm{span}}(\matr{V}_{\leq r-1})$.
		\end{enumerate}
\end{lemma}

\begin{proof}
  In appendix, section \ref{sec:proof-lemma-multivar-size}
\end{proof}

With lemma \ref{lem:size_Xe_varying_multivariate} in hand, along with the fixed-size
results in section \ref{sec:multivar-poly} we can prove the following: 
\begin{theorem}
  \label{thm:Xe_varying_multivariate}
  Let $p\in\mathbb{N}$, $\alpha>0$, and $\Omega=\{\vect{x}_1,\ldots,\vect{x}_n\}$ a set of $n$
  distinct points in $\R^d$. Let $\bL_\varepsilon = [\kappa_\varepsilon(\vect{x}_i,\vect{x}_j)]_{i,j}$ with $\kappa$ a
	stationary kernel of smoothness order $r\in\mathbb{N}^*$.  Let $\Xe \sim
	DPP(\alpha\varepsilon^{-p}\bL_\varepsilon)$.  In the limit $\flatlim$, the
  distribution of $\Xe$ depends on the interplay between $p, r$ and $n$. First of all, $p$ is either even or odd: only one out of the two following values $\left(\frac{p}{2},\frac{p+1}{2}\right)$ is an integer. We call that integer $l$. Now, if $\PP_{l-1,d}\geq n$  then, for any value of $r$, $\Xe$ has limit $\Omega$ with probability one. 
  Otherwise, there are three scenarii depending on the value of $r$: 
	\begin{enumerate}
		\item if $r<\frac{p+1}{2}$, then $\Xe$ has limit $\Xs=\Omega$ with probability one.
		\item if $r>\frac{p+1}{2}$, $\Xe$ has a limiting distribution that depends on the parity of $p$:
		\begin{enumerate}
			\item If $p$ is odd ($l=\frac{p+1}{2}$), then $\Xe$ has limit $\Xs \sim \mDPP{\PP_{l-1,d}}(\bV_{\leq l-1}\bV_{\leq l-1}^t)$
			\item If $p$ is even ($l=\frac{p}{2}$) then $\Xe$ has limit $\Xs \sim \ppDPP \ELE{\alpha \bV_{l}\bar{\bW}\bV_{l}^t}{\bV_{\leq
            l-1}}$ with $\bV_{l}\bar{\bW}\bV_{l}^t$ as in theorem \ref{thm:general-case-smooth-fixed-size}.
		\end{enumerate}
		\item if $r=\frac{p+1}{2}$, then $\Xe$ has limit $\Xs \sim ppDPP \ELE{\alpha f_{2r-1}
			\bD^{(2r-1)}}{\bV_{\leq  r-1}}$.
	\end{enumerate}
\end{theorem}
\begin{proof}
  Repeats the univariate proof. 
\end{proof}

\begin{remark}
  The following (non-trivial) limit is universal: for odd $p$ and
  $r>\frac{p+1}{2}$, the limit process is $\Xs \sim \mDPP{\PP_{l-1,d}}(\bV_{\leq
    l-1}\bV_{\leq l-1}^t)$ which does not depend on the Wronskian. This means
  that L-ensembles in the flat limit tend to exhibit ``natural'' sizes, the set
  $\{\PP_{1,d},\PP_{2,d},\ldots \}$.

  Another limit
  exhibits only weak dependency on the Wronskian: if $r=\frac{p+1}{2}$, then
  $\Xe$ has limit $\Xs \sim DPP \ELE{\alpha f_{2r-1}
    \bD^{(2r-1)}}{\bV_{\leq  r-1}}$, where the Wronskian is only present via
  $f_{2r-1}$, a scaling parameter which can be compensated via $\alpha$.
\end{remark}

\section{To conclude}
\label{sec:discussion}

The results in this work can be summarised as follows. Two are very general observations,
namely that partial-projection DPPs form the closure of the set of DPPs under
pertubative limits, and that extended L-ensembles are a natural unifying
representation for DPPs and fixed-size DPPs. The rest concern the flat limit: as
$\flatlim$, L-ensembles formed from stationary kernels stay well-defined (and
meaningfully repulsive). In some cases we obtain universal limits where the
limit process depends only on $r$ and not the Wronskian of the kernel. In
dimension $d > 1$, these universal limits are obtained for certain natural
values of $m$ (for fixed-size L-ensembles) or when rescaling with $\varepsilon^{-p}$
for $p$ odd (varying-size L-ensembles).

The question of how fast L-ensembles converge to the limits given here requires
expansions to the next order, which we do not yet have. Empirically, we observe
that convergence is quite fast in the fixed-size case, but slower in the
varying-size case, at least in some instances. This means that the distribution
of the size of $\Xe$ may converge slowly to its limit. We hope to investigate
this further in future work.

In the interests of space we have left some topics aside. Our results on the
flat limit should apply as well to D-optimal design, and there is an interesting
connection to polyharmonic splines for kernels
with finite $r$ (see section \ref{sec:ppdpp-examples-cpdef}, and \cite{song2012multivariate}). We have also entirely skipped
the topic of computational applications of these results. Finally, the
univariate results point to possible connections with random matrix theory we
have yet to explore.

Directions for future work include extending the results to continuous DPPs, and
in a related vein letting $n \rightarrow \infty$ as $\flatlim$ in discrete DPPs.
This should let one take advantage of some results from the literature on the
asymptotics of Christoffel functions, as in \cite{tremblay2019determinantal}. It
would also be worth investigating the flat limit on Riemannian manifolds, rather
than on $\R^d$ as we do here. 

\section*{Acknowledgments}
We thank Guillaume Gautier for helpful comments on preliminary versions of this
manuscript.

This work was supported by ANR project GenGP (ANR-16-CE23-0008), ANR project
LeaFleT (ANR-19-CE23-0021-01), LabEx PERSYVAL-Lab (ANR-11-LABX-0025-01),
Grenoble Data Institute (ANR-15-IDEX- 02), LIA CNRS/Melbourne Univ Geodesic, and
partial funding from the IRS (Initiatives de Recherche Stratégiques) of the IDEX
Université Grenoble Alpes.

\begin{appendix}
  \label{sec:appendix}

\section{Inclusion probabilities in mixtures of projection DPPs}
\label{sec:marginal-kernel-ppDPPs-proof}

Here, we give formulas for inclusion probabilities valid for mixtures of projection
DPPs. These formulas yield the marginal kernels of L-ensembles and partial-projection
DPPs as a special case. We give a variant of a calculation in \cite{Barthelme:AsEqFixedSizeDPP},
appendix A.2.

Let $\bU$ be a fixed orthonormal basis of $\mathbb{R}^n$. We assume that $\X$ is generated according to the following mixture process:
\begin{enumerate}
\item Sample indices $\Y \sim \Proba(\Y)$
\item Form the projection matrix $\bM = \bU_{:,\Y} (\bU_{:,\Y})^\top$
\item Sample $\X | \Y \sim  \mDPP{m}(\bM)$
\end{enumerate}
We do not specify $\Proba(\Y)$ for now (it may be an L-ensemble, a fixed-size L-ensemble, etc.).

Since $\X$ is a mixture of projection-DPPs we can write 
\begin{eqnarray}
  \label{eq:incl-prob-mixture}
  \Proba(  \cW \subseteq \X)= \E_\Y[ \Proba( \cW \subseteq \X \vert \Y) ]
\end{eqnarray}
where the outer expectation is over $\Y$, the indices of the columns of $\bU$ 
sampled in the mixture process.
Since the innermost quantity is an inclusion probability for a projection DPP,
we have from lemma \ref{lem:marginal-kernel-proj}:
\begin{align*}
  \Proba( \cW \subseteq \X \vert \Y) & = \det \left(  \bM_\cW \right)\\
                                &=  \det \left( \bU_{\cW,\Y}(\bU_{\cW,\Y})^\top \right) \\
                                &=  \sum_{\cA \subseteq \Y,|\cA| = |\cW|}  \det\left( \bU_{\cW,\cA}\right)^2 
\end{align*}
where the last line follows from the Cauchy-Binet lemma (lemma
\ref{lem:cauchy-binet}). Injecting into \ref{eq:incl-prob-mixture}, we find:
\begin{align*}
  \Proba(  \cW \subseteq \X) &= \E_\Y\left[\sum_{\cA \subseteq \Y,|\cA| = |\cW|}  \det\left( \bU_{\cW,\cA}\right)^2 \right] \\
                        &= \sum_{\Y} \Proba(\Y) \sum_{\cA \subseteq \Y,|\cA| = |\cW|}  \det\left( \bU_{\cW,\cA}\right)^2 \\
                        &= \sum_{\Y,\cA \slash |\cA|=|\cW|} \det\left( \bU_{\cW,\cA}\right)^2 \Proba(\Y) \Ind\{\cA \subseteq \Y \} \\
                        &= \sum_{\cA \slash |\cA|=|\cW|} \det\left( \bU_{\cW,\cA}\right)^2 \Proba(\cA \subseteq \Y).
\end{align*}
In the case of L-ensembles and partial projection DPPs, we can go a bit further, since
the distribution of $\Y$ is a Bernoulli process (meaning that each element $i$
is included independently with probability $\pi_i$). In that case 
$\Proba(\cA \subseteq \Y) = \prod_{i \in \cA} \pi_i$, and using the Binet-Cauchy lemma 
once again we find: 
\begin{align*}
  \Proba(  \cW \subseteq \X) &= \sum_{\cA \slash |\cA|=|\cW|} \det\left( \bU_{\cW,\cA}\right)^2 \prod_{i \in \cA} \pi_i \\
                        &= \det \bU_{\cW,:} \diag(\pi_1,\ldots,\pi_n) (\bU_{\cW,:})^\top \\
                        &= \det \bK_{\cW} \numberthis \label{eq:marginal-kernel-generic}
\end{align*}
with $\bK = \bU \diag(\pi_1,\ldots,\pi_n) \bU^\top$.

\section{Size of $\Xe$ in the multivariate case }
\label{sec:proof-lemma-multivar-size}

We prove lemma \ref{lem:size_Xe_varying_multivariate}. In the following, $\bL_{\varepsilon,\X}$ stands for the matrix $\bL_{\varepsilon}$ reduced to its lines and columns indexed by $\X$. 
	First, recall that if $\X \sim DPP(\bL)$, then the marginal distribution of the size $|\X|$ is given by
	Eq.~\eqref{eq:distr-size-dpp}:
	\begin{equation}
	\label{eq:marginal-distr-size_multivar}
	\Proba(|\X| = m) =  \frac{e_m(\bL)}{e_0(\bL) + e_1(\bL) + \ldots + e_n(\bL)}.
	\end{equation}
	where $e_m(\bL)$ is the $m$-th elementary symmetric polynomial of $\bL$ and for
	consistency $e_0(\bL)=1$ for all matrices $\bL$. Here, we consider the $L$-ensemble $\alpha\varepsilon^{-p}\bL_\varepsilon$. 
	Recall that $\det(\alpha\varepsilon^{-p} \bL_{\varepsilon,\X}) =
	\alpha^{|\X|}\varepsilon^{-p|\X|}\det(\bL_{\varepsilon,\X})$. One thus has $\forall i$: $ e_i(\alpha\varepsilon^{-p} \bL_\varepsilon) = \alpha^i \varepsilon^{-ip} e_i(\bL_\varepsilon)$. \\	
	Let $r\in\mathbb{N}^*$, $d\geq 2$ and consider $i\leq \PP_{r-1,d}$. In the flat limit, we can apply theorem 6.1 in~\cite{BarthelmeUsevich:KernelsFlatLimit}. There are two cases: either $i$ is a magic number ($i\in\magicn{d}$) in which case $k\in\mathbb{N}$ will denote the integer verifying $i=\PP_{k,d}$, or it is a muggle number ($i\notin\magicn{d}$) in which case $k\in\mathbb{N}$ denotes the smallest integer such that $i\leq\PP_{k,d}$. In both cases, we denote by $M(i)$ the integer $M(i)=d {k+d \choose d+1}$. Combining points 1 and 2 of Theorem 6.1 in~\cite{BarthelmeUsevich:KernelsFlatLimit}, one has, $\forall \; 1\leq i\leq \PP_{r-1,d}$:
	\begin{align*}
	 e_i(\bL_\varepsilon) &= \sum_{|\X|=i} \det \bL_{\varepsilon,\X} \\
	&= \varepsilon^{2\left(M(i)+k(\PP_{k,d}-i)\right)} \left( \sum_{|\X|=i} \det(\matr{Y}\matr{W}_{\leq k}\matr{Y}^\top) \det(\matr{V}_{\le k-1}(\X)^\top \matr{V}_{\le k-1}(\X)) + \O(\varepsilon) \right)
	\end{align*}
	where $\matr{Y}$ is as in Eq.~\eqref{eq:def_Y}. Denoting 
	\begin{align}
	\label{eq:etilde_multivar}
	\forall\; 1\leq i\leq \PP_{r-1,d}\qquad \tilde{e}_i&=\sum_{|\X|=i} \det(\matr{Y}\matr{W}_{\leq k}\matr{Y}^\top) \det(\matr{V}_{\le k-1}(\X)^\top \matr{V}_{\le k-1}(\X)),
	\end{align}
	one has:
	\begin{align*}
\forall\; 1\leq i\leq \PP_{r-1,d}\qquad e_i(\bL_\varepsilon)=  \varepsilon^{2\left(M(i)+k(\PP_{k,d}-i)\right)} \left( \tilde{e}_i + \O(\varepsilon)  \right).
	\end{align*}
	Also, we can apply theorem 6.3 of~\cite{BarthelmeUsevich:KernelsFlatLimit} for any set $\X$ of size $i\geq \PP_{r-1,d}$:
	\begin{align*}
	\forall i\geq \PP_{r-1,d}\qquad e_i(\bL_\varepsilon) &= \sum_{|\X|=i} \det \bL_{\varepsilon,\X} = \varepsilon^{2d  {r+d-1 \choose d+1} +(2r-1)(i-\PP_{r-1,d})} \left( \sum_{|\X|=i} \tilde{l}(\X) + \O(\varepsilon) \right)\\
	&=\varepsilon^{2d  {r+d-1 \choose d+1} +(2r-1)(i-\PP_{r-1,d})} \left( \bar{e}_i + \O(\varepsilon) \right)
	\end{align*}
	where $\bar{e}_i$ verifies the same equation than in the univariate case, Eq.~\eqref{eq:ebar}, replacing $\matr{W}_{r-1,r-1}$, $\matr{V}_{\le r-1}$ and $\bD^{(2r-1)}$ by their multivariate counterparts. \\	
	Now, injecting into Eq.~\eqref{eq:marginal-distr-size_multivar}, and following the proof scheme of the univariate case, one shows that $\Proba(|\X| = m)$ may be written as:
	\begin{align*}
	\forall m,\qquad \Proba(|\X| = m) &= \frac{\varepsilon^{\eta(m)}\left(f_0(m) + \O(\varepsilon)  \right) }{\sum_{i=0}^n \varepsilon^{\eta(i)}\left(f_0(i) + \O(\varepsilon)  \right) }
	\end{align*}
where $\eta(\cdot)$ and $f_0(\cdot)$ are two $\varepsilon$-independent functions verifying:
	\begin{equation}
	\label{eq:d_of_i_vsize_proof_multivar}
	\eta(i) = \left\{ \begin{array}{ll} 
	\eta_0(i)=0 & \mbox{if }  i=0 \\ 
	\eta_1(i)=i(2-p)-2 & \mbox{if }  0<i\leq \PP_{1,d} \\ 
	\eta_2(i)=i(4-p)-2d-4 & \mbox{if }  \PP_{1,d}\leq i\leq \PP_{2,d} \\ 
	\vdots\\
	\eta_l(i)=i(2l-p)-2 {d+l \choose d+1} & \mbox{if }  \PP_{l-1,d}\leq i\leq \PP_{l,d} \\ 
	\vdots\\
	\eta_{r-1}(i)=i(2r-2-p)-2 {d+r-1 \choose d+1} & \mbox{if }  \PP_{r-2,d}\leq i\leq \PP_{r-1,d} \\ 
	\eta_{r}(i)=i(2r-1-p)- \left(2+\frac{d+1}{r-1}\right){d+r-1 \choose d+1} & \mbox{if }  i\geq \PP_{r-1,d} 
	\end{array} \right. 
	\end{equation}
	and
	\begin{align*}
	f_0(i) = \left\{ \begin{array}{ll} 
	1 & \mbox{if }  i=0 \\ 
	\alpha^i \tilde{e}_i & \mbox{if }  0<i\leq \PP_{r-1,d} \\ 
	\alpha^i \bar{e}_i & \mbox{if } i\geq \PP_{r-1,d}.
	\end{array} \right. 
	\end{align*}
	As in the univariate case, we will make use of lemma~\ref{lem:TV-convergence-diverging}. In order to apply it, one needs to find the integers between $0$ and $n$ for which $\eta(\cdot)$ is minimal:
	$$\argmin_{i\in\mathbb{N}, i\in[0,n]} \eta(i).$$ 
	The answer to this question depends on $p,
	r$ and $n$ which explains the different cases of the theorem. 	
	Let us make first
	a few simple observations on the function $\eta:\mathbb{R}^+\rightarrow \mathbb{R}$ :
	\begin{itemize}
		\item $\eta(\cdot)$ is continuous (everywhere except in $i=0$) and piecewise linear.
		\item the slope of each of the linear pieces of $\eta(\cdot)$ is strictly increasing, starting at $2-p$ for the first piece $0<i\leq \PP_{1,d}$ and finishing at $2r-1-p$ for the last piece $i\geq \PP_{r-1,d}$.
	\end{itemize}
	We shall now explore all the possible cases sequentially. 
\begin{enumerate}
\item if $r<\frac{p+1}{2}$, \textit{i.e.}, $2r-1-p<0$: the slope of all the pieces of $\eta(\cdot)$ are negative, and $\eta(\cdot)$ is thus strictly decreasing. In this case, the integer in $[0,n]$ minimizing $\eta$ is $i=n$. Applying lemma~\ref{lem:TV-convergence-diverging}, as $\flatlim$, $|\Xe|=n$ with probability $1$.
			\item if $r>\frac{p+1}{2}$:
			\begin{enumerate}
				\item if $p$ is odd, then $\frac{p-1}{2}$ is an integer and $\PP_{\frac{p-1}{2},d}$ is well defined. Trivially, $r>\frac{p+1}{2}$ implies $\PP_{\frac{p-1}{2},d}<\PP_{r-1,d}$. Also, note that $\eta(\cdot)$ decreases strictly between $0^+$ and $\PP_{\frac{p-1}{2},d}$, and then increases strictly after $\PP_{\frac{p-1}{2},d}$. The integer in the interval $[0,n]$ minimizing $\eta(\cdot)$ is thus $\min\left(\PP_{\frac{p-1}{2},d},n\right)$. Applying lemma~\ref{lem:TV-convergence-diverging}, as $\flatlim$, $|\Xe|=\min\left(\PP_{\frac{p-1}{2},d},n\right)$ with probability $1$.
				\item if $p$ is even (the case $p=0$ falls into this category, recall that $\PP_{-1,d}$ is by convention set to $0$), then $r > \frac{p+1}{2}$ implies $\frac{p}{2}\leq r-1$ and thus  $\PP_{\frac{p}{2},d}\leq\PP_{r-1,d}$. Also, note that $\eta(\cdot)$ decreases strictly between $0^+$ and $\PP_{\frac{p}{2}-1,d}$, is constant between $\PP_{\frac{p}{2}-1,d}$ and $\PP_{\frac{p}{2},d}$, and then increases strictly after $\PP_{\frac{p}{2},d}$. The integers in the interval $[0,n]$ minimizing $\eta(\cdot)$ are thus: 
				\begin{itemize}
					\item $\{n\}$ if $n\leq  \PP_{\frac{p}{2}-1,d}$. In this case, applying lemma~\ref{lem:TV-convergence-diverging}, as $\flatlim$, $|\Xe|=n$ with probability $1$.
					\item all the integers contained in the interval $\left[\PP_{\frac{p}{2}-1,d},\min\left(\PP_{\frac{p}{2},d},n\right)\right]$ if $n\geq  \PP_{\frac{p}{2}-1,d}$. In the following $I_{p,d}$ is the list of these integers. 
					Applying lemma~\ref{lem:TV-convergence-diverging}, as $\flatlim$:
					\begin{align}
					\forall m\in I_{p,d}\quad	\Proba(|\Xs|=m)=\frac{\alpha^m\tilde{e}_m}{\sum_{i\in I_{p,d}} \alpha^i\tilde{e}_i}
					\end{align}
					Now, using the same arguments as in the proof of theorem~\ref{thm:general-case-smooth-fixed-size}, note that Eq.~\eqref{eq:etilde_multivar} may be re-written as:
					\begin{align}
					\forall i\in I_{p,d}\qquad \tilde{e}_i=
					\det(\matr{W}_{\leq k-1})\sum_{|\X|=i} \det
					\begin{pmatrix}
					\bV_k(\X) \bar{\bW} \bV_k(\X)^\top & \bV_{\leq k-1}(\X) \\
					\bV_{\leq k-1}(\X)^\top  & \matr{0}
					\end{pmatrix}
					\end{align}
					where $\bar{\bW}$ is as in theorem~\ref{thm:general-case-smooth-fixed-size} and $k=\frac{p}{2}$. 
					Now, consider the NNP $(\bV_k \bar{\bW} \bV_k^\top; \matr{V}_{\leq k-1})$ as well as $\widetilde{\bV_k \bar{\bW} \bV_k^\top}$ as defined in Definition~\ref{def:ext_Lens}. Note that the rank of $\widetilde{\bV_k \bar{\bW} \bV_k^\top}$ is $\min\left(\HH_{k,d}, n-\PP_{k-1,d}\right)$. One may apply Corollary~\ref{cor:normalisation-ppDPP} and obtain, for all integer $i\in I_{p,d}$:
					\begin{align}
					\tilde{e}_i=
					\det(\matr{W}_{\leq k-1}) (-1)^{\PP_{k-1,d}} e_{i-{\PP_{k-1,d}}}\left(\widetilde{\bV_k \bar{\bW} \bV_k^\top}\right) \det \left(\left(\matr{V}_{\leq k-1}\right)^\top \matr{V}_{\leq k-1}\right).
					\end{align}
					Simplifying, one obtains:
					\begin{align}
					\label{eq:temp_result}
					\forall m\in I_{p,d}\quad \Proba(|\Xs|=m)=\frac{e_{m-{\PP_{k-1,d}}}\left(\alpha\widetilde{\bV_k \bar{\bW} \bV_k^\top}\right)}{\sum_{i\in I_{p,d}}e_{i-{\PP_{k-1,d}}}\left(\alpha\widetilde{\bV_k \bar{\bW} \bV_k^\top}\right)}.
					\end{align}
					Changing the summing index gives:
					\begin{align}
					\sum_{i\in I_{p,d}}e_{i-{\PP_{k-1,d}}}\left(\alpha\widetilde{\bV_k \bar{\bW} \bV_k^\top}\right)=
					\sum_{i=0}^{\min(\HH_{k,d},n-\PP_{k-1,d})} e_{i}\left(\alpha\widetilde{\bV_k \bar{\bW} \bV_k^\top}\right).
					\end{align}
					Finally, note that, as $\text{rank}\left(\widetilde{\bV_k \bar{\bW} \bV_k^\top}\right) = \min\left(\HH_{k,d}, n-\PP_{k-1,d}\right)$, all the elementary symmetric polynomials $e_i$ for $i>\min\left(\HH_{k,d}, n-\PP_{k-1,d}\right)$ are null. The denominator of Eq.~\eqref{eq:temp_result} is thus $\sum_{i=0}^n e_{i}\left(\alpha\widetilde{\bV_k \bar{\bW} \bV_k^\top}\right)=\det\left(\matr{I}+\alpha\widetilde{\bV_k \bar{\bW} \bV_k^\top}\right)$ and one obtains:
					\begin{align}
					\forall m\in I_{p,d}\quad \Proba(|\Xs|=m)=\frac{e_{m-{\PP_{k-1,d}}}\left(\alpha\widetilde{\bV_k \bar{\bW} \bV_k^\top}\right)}{\det\left(\matr{I}+\alpha\widetilde{\bV_k \bar{\bW} \bV_k^\top}\right)}.
					\end{align}
				\end{itemize}
			\end{enumerate}
			\item if $r=\frac{p+1}{2}$, $\eta(\cdot)$ decreases strictly between $0^+$ and $\PP_{r-1,d}$, and is constant after $\PP_{r-1,d}$. The integers in the interval $[0,n]$ minimizing $\eta(\cdot)$ are thus: 
				\begin{itemize}
					\item $\{n\}$ if $n\leq  \PP_{r-1,d}$. In this case, applying lemma~\ref{lem:TV-convergence-diverging}, as $\flatlim$, $|\Xe|=n$ with probability $1$.
					\item all those contained in the interval $\left[\PP_{r-1,d},n\right]$ if $n\geq  \PP_{r-1,d}$. In the following $I_{r,d}$ is the list of these integers. 
					Applying lemma~\ref{lem:TV-convergence-diverging}, as $\flatlim$:
					\begin{align*}
					\forall m\in I_{r,d}\quad	\Proba(|\Xs|=m)=\frac{\alpha^m\bar{e}_m}{\sum_{i\in I_{r,d}} \alpha^i\bar{e}_i}.
					\end{align*}
					Now, using the same line of arguments as in the proof of lemma~\ref{lem:size_Xe_varying_univariate}, one obtains:
					\begin{align}
					\Proba(|\Xe| = m) =  \left\{ \begin{array}{ll} 
					0 & \mbox{if }  m< \PP_{r-1,d}\\ 
					\frac{e_{m-\PP_{r-1,d}}\left(\alpha f_{2r-1}\widetilde{\bD^{(2r-1)}}\right)}{\det\left(\bI+\alpha f_{2r-1}\widetilde{\bD^{(2r-1)}}\right)} & \mbox{if } m\geq \PP_{r-1,d}\end{array} \right. 
					\end{align}
					where $\widetilde{\bD^{(2r-1)}}=(\bI-\bQ\bQ^\top) \bD^{(2r-1)}(\bI-\bQ\bQ^\top)$, $\bQ$ being an orthonormal basis of ${\rm{span}}(\matr{V}_{\leq r -1})$.
				\end{itemize}
		\end{enumerate}
	Finally, one may see that the three cases just described can in fact be equivalently stated in the form of the Lemma, finishing the proof.

  \section{Equivalence of extended L-ensembles and DPPs }
\label{sec:thm:K_to_L-ens_and_back}
We prove Th.~\ref{thm:L-ens_to_K} and \ref{thm:K_to_L-ens}.

\begin{proof}[{Proof of Th.~\ref{thm:L-ens_to_K} }]

	Let $\ELE{\bL}{\bV}$ be any NNP, and $\tbL$, $\bQ$, $\tbU$, $\tbLam$ and $q$ be as in Definition~\ref{def:ext_Lens}. 
	Let $\X\in\Omega$ be drawn according to the distribution:
	\begin{align}
	\forall X\subseteq\Omega,\qquad \Proba(\X=X) \propto (-1)^p \det
	\begin{pmatrix}
	\bL_{X} & \bV_{X,:} \\
	(\bV_{X,:})^\top & \matr{0}
	\end{pmatrix}.
	\end{align} 
	Using the generalized Cauchy-Binet formula (theorem~\ref{thm:equivalence-extended-spectral}), this can be re-written as
	\begin{align}
	\forall X\subseteq\Omega,\qquad \Proba(\X=X) \propto \det(\bV^\top\bV)\sum_{Y,|Y|=m-p} \det \left( 
	\begin{bmatrix}
	\bQ_{X,:} & \tbU_{X,Y}
	\end{bmatrix}  \right)^2
	\prod_{i \in Y} \widetilde{\lambda}_i.
	\end{align} 
	As made precise by corollary~\ref{cor:mixture-rep-ppDPP-varying}, this equation can be interpreted from a mixture point of view. As such, the generic inclusion probability formulas of Appendix~\ref{sec:marginal-kernel-ppDPPs-proof} are applicable and one obtains the result.	
\end{proof}

\begin{proof}[{Proof of Th.~\ref{thm:K_to_L-ens}}]
	Given a marginal kernel $\matr{K}$, we can always rewrite  its spectral factorisation in the form of  Eq.~\eqref{eq:dpDPPmarginalkernel}, by grouping all the eigenvectors corresponding to the eigenvalue $1$ in $\matr{Q}$; all the remaining eigenvalues can be always represented as $\widetilde{\lambda}_i/(1+\widetilde{\lambda}_i)$.
\end{proof}

\end{appendix}

\bibliographystyle{imsart-number}
\bibliography{flat_limit}

\end{document}